\documentclass[11pt]{amsart}

\usepackage{graphicx}  %%% the recommended graphics package

\usepackage{pinlabel}  %%% the recommended graphics+labelling package

\usepackage[all]{xy}
\usepackage{amscd}
\usepackage{amsbsy}
\usepackage{verbatim}
\usepackage{url}
\usepackage{hyperref}

\DeclareGraphicsExtensions{.eps}

\newtheorem{theorem}{Theorem}[section]
\newtheorem{lemma}[theorem]{Lemma}
\newtheorem{corollary}[theorem]{Corollary}
\newtheorem{proposition}[theorem]{Proposition}

\newtheorem{question}[theorem]{Question}

\theoremstyle{definition}
\newtheorem{definition}{Definition}[section]
\newtheorem{example}{Example}[section]

\numberwithin{equation}{section}

\newcommand{\De}{\mathbb{D}}
\newcommand{\I}{\mathbb{I}}
\newcommand{\hP}{\hat{P}}
\newcommand{\uP}{\widetilde{P}}
\newcommand{\rank}{\mathrm{rank}\,}
\newcommand{\SL}{\mathbf{SL}^{\pm}_{n+1}(\mathbb{R})}

\newcommand{\GL}{\mathbf{GL}_{n+1}(\mathbb{R})}
\newcommand{\G}{\mathbb{G}}
\newcommand{\f}{\vartheta}
\newcommand{\Hom}{\mathrm{Hom}}
\newcommand{\PO}{\mathbf{PO}(1,n)}

\newcommand{\M}{\mathrm{M}}

\newcommand{\PS}{\mathbb{S}^{n}}
\newcommand{\Clo}{\overline{\Omega}}
\newcommand{\PV}{\mathbb{PV}}
\newcommand{\ra}{\rightarrow}
\newcommand{\bZ}{\mathbb{Z}}

\newcommand{\F}{\mathbb{F}}
\newcommand{\FK}{\mathfrak{F}K}
\newcommand{\hol}{\mathrm{hol}}
\newcommand{\PGL}{\mathbf{PGL}_{n+1}(\mathbb{R})}
\newcommand{\R}{\mathbb{R}}
\newcommand{\qua}{\,}

\def\co{\colon\thinspace}

\begin{document}

\title[Projective deformations of hyperbolic Coxeter orbifolds]{Projective deformations of weakly orderable hyperbolic Coxeter orbifolds}

%  First author
%
\author[S Choi]{Suhyoung Choi}
%\givenname{Suhyoung}
%\surname{Choi}
\address{Department of Mathematical Sciences\\
KAIST\\\newline
Daejeon 305-701\\Republic of Korea}
\email{schoi@math.kaist.ac.kr}
\urladdr{http://mathsci.kaist.ac.kr/~schoi/}

%  Second author (uncomment if necessary)
%
\author[G-S Lee]{Gye-Seon Lee}
%\givenname{Gye-Seon}
%\surname{Lee}
%\address{Department of Mathematical Sciences\\
%Seoul National University\\\newline
%Seoul 151-747\\Republic of Korea}
\address{Mathematisches Institut \\
Ruprecht-Karls-Universit\"{a}t Heidelberg\\\newline
D-69120 Heidelberg\\Germany}
\email{lee@mathi.uni-heidelberg.de}
\urladdr{http://www.mathi.uni-heidelberg.de/~lee/}
%
%  (Add a similar block for other authors)
%
%   Title and author both have running head options:
%
%   \title[Running head title]{Main title}
%   \author[Running head author]{Author}
%
% give a separate \keyword and \subject line for each keyword/phrase or
% subject class eg \keyword{framed link} \subject{primary}{msc2000}{57M25}

\keywords{Real projective structure, Orbifold, Moduli space, Coxeter groups, Representations of groups}
%\keyword{Orbifold}
%\keyword{Moduli space}
%\keyword{Coxeter groups}
%\keyword{Representations of groups}

%\subject{primary}{msc2000}{57M50}
%\subject{secondary}{msc2000}{57N16}
%\subject{secondary}{msc2000}{53A20}
%\subject{secondary}{msc2000}{53C15}

\subjclass{[2000]Primary 57M50; Secondary 57N16, 53A20, 53C15}

\date{\today} 

\begin{abstract}    % type your abstract below
A Coxeter $n$--orbifold is  an $n$--dimensional orbifold based on a polytope with silvered boundary facets. 
Each pair of adjacent facets meet on a ridge of some order $m$,
 whose neighborhood is locally modeled on $\R^n$ modulo the dihedral group of order $2m$ generated by two reflections.
For $n \geq 3$, we study the deformation space of real projective structures on a compact Coxeter $n$--orbifold $Q$
admitting a hyperbolic structure. %Let $f$ be the number of facets, and 
Let $e_+(Q)$ be the number of ridges of order $\geq 3$.
A neighborhood of the hyperbolic structure in the deformation space is a cell of dimension $e_+(Q) \!-n$
if $n=3$ and $Q$ is weakly orderable, ie the faces of $Q$ can be ordered so that each face contains
at most $3$ edges of order $2$ in faces of higher indices, or $Q$ is based on a truncation polytope.

%$Q$ satisfies the following two assumptions: the number of $(n-2)$--faces of $Q$ is equal to $nf-\tfrac{n(n+1)}{2}$, and $Q$ is weakly orderable,
%ie the $(n-1)$--faces of $Q$ can be ordered so that each $(n-1)$--face contains
%at most $n$ $(n-2)$--faces of order $2$ in $(n-1)$--faces of higher indices. Moreover, for $n=3$, these assumptions hold often.
\end{abstract}

\maketitle

%%%%%%%%%%%%%%%%%%%%   Start of main body of article

\section{Introduction}
\label{s:intro}
In this paper, an {\em $n$--orbifold} $Q$ is based on a quotient space of a simply connected manifold $\widetilde Q$ by
a discrete group $\Gamma$ acting on $\widetilde Q$ properly discontinuously. An {\em orbifold structure} on 
$Q$ is given by a covering by open sets of form $\phi(U)$ with
a model $(U, H, \phi)$ where $U$ is an open subset of $\widetilde Q$, $H$ is a finite subgroup of $\Gamma$ acting 
on $U$, and $\phi$ induces a homeomorphism $U/H \ra \phi(U)$. 
% but possibly with some fixed points.
%Each point of $Q$ has an associated model $(U, \phi, H)$ where $H$ is a finite group acting on an open subset $U$ 
%of $\R^n$ and $\phi\co U/H \ra Q$ is an
%embedding onto an open set.
Here, $\widetilde Q$ is said to be a {\em universal cover} of $Q$, %and is uniquely determined up to isomorphisms,
%and such an orbifold $Q$ is said to be \emph{good}.
and $\Gamma$ is the {\em fundamental group} and is denoted by $\pi_1(Q)$.
%$Q$ without the charts is the {\em base space} $|Q|$ of the orbifold. %, and $Q$ is {\em closed} if the base space is compact.

A \emph{Coxeter group} is a group having a group presentation
\begin{displaymath}
\langle \; r_i \;|\; (r_i r_j)^{n_{ij} } \; (i,j \in \I) \;\rangle
\end{displaymath}
where $\I$ is a set, $n_{ii}=1$ for each $i \in \I$, and $n_{ij} \in \{2,3,\dotsc, +\infty\}$ is symmetric.
Note that $n_{ij}= +\infty$ means no relation between $r_i$ and $r_j$.

A point in an $n$--orbifold $Q$ is called a {\em silvered point} if it has an open neighborhood of form $\phi(U)$ with a model $(U, \bZ/2\bZ, \phi)$
for an open set $U$ in $\widetilde Q$ and a $\bZ/2\bZ$--action on $U$ fixing a hypersurface in $U$.
A \emph{Coxeter $n$--orbifold} $\hP$ is an $n$--dimensional orbifold whose base space is an $n$--dimensional polyhedron $P$ with
finitely many sides where all interior points of the facets are silvered. % and some sides  of codimension $\geq 3$ are removed.
%It easily follows that the fundamental group of such a Coxeter orbifold is isomorphic to a Coxeter group.
%We will only study good Coxeter $n$--orbifolds since these orbifolds admitting $(G,X)$--structures
%are good (see \cite{Choi04} for details here). Such an orbifold $Q$ is covered by a manifold $M$,
The fundamental group $\pi_1(\hP)$ is isomorphic to a Coxeter group, and is generated by reflections about sides of the fundamental domain $P$.
We will study only compact ones in this paper, ie closed ones.
%that is uniquely determined up to deck transformations.
(More precisely, Davis \cite{Davis2010, Davis2013} calls such an orbifold a {\em Coxeter orbifold of type III}, 
an {\em orbifold of reflection type}, or a {\em reflectofold}.)

%More precisely, let $\Af^n$ be an $n$--dimensional affine space and let $P \subset \Af^n$ be an $n$--dimensional \emph{convex polytope}, ie the convex hull of a finite subset. The faces of codimension one and two are called the \emph{facets} and \emph{ridges} respectively.
%Let $P$ be a fixed $n$--dimensional convex polytope with facets $F_i$ ($i \in \I$), assign an order $n_{ij} \geq 2$
%to each ridge $F_i \cap F_j$ of $P$ (if $F_i$ and $F_j$ are adjacent), and define a Coxeter group
%$\Gamma = \langle r_i \,|\, (r_i r_j)^{n_{ij}} \rangle$, where $n_{ij}= +\infty$ if $F_i$ and $F_j$ are not adjacent.
%For $x \in P$, let $\Gamma_x$ denote the subgroup of $\Gamma$ generated by $\{ r_i : x \in F_i \}$.

Let $V$ be an $(n+1)$--dimensional real vector space.
The {\em projective sphere} $\PS$ is the space of rays in $V$ and is a double cover of $\mathbb{RP}^n$.
Let $$\SL = \{ A \in \GL \co \det(A)= \pm 1 \}.$$
The group $\SL$ acts on $\PS$ effectively in the standard manner and is a double cover of $\PGL$.
The elements of $\SL$ are the \emph{projective automorphisms} of $\PS$
and $\SL$ the \emph{projective automorphism group} of $\PS$.
(We will also think of $\SL$ as a linear group when it is convenient.)
Denote by $\Pi$ the natural projection from $V\backslash\{0\}$ into $\PS$.
A subspace of $\PS$ is the image of a subspace of $V$ with the origin removed.
In particular, a $2$--dimensional subspace of $V$ corresponds to a great circle in $\PS$, 
and an $n$--dimensional subspace gives a great $(n-1)$--sphere in $\PS$.
Furthermore, a component of the complement of a great $(n-1)$--sphere has the canonical structure of an affine $n$--space, as
the complement of a codimension-one subspace of $\mathbb{RP}^n$ is an affine subspace.
We call this an \emph{affine subspace} of $\PS$.

A {\em convex segment} in $\PS$ is a connected arc contained in a great circle but not containing a pair of antipodal points
in its interior.  A subset $A$ of $\PS$ is {\em convex} if any two points of $A$ are connected
by a convex segment in $A$. 
An affine space has  %is the complement of a codimension-one subspace in $\PS$, and has
a notion of geodesics as arcs in one--dimensional affine subspaces.
A subset of an affine subspace of $\PS$ is convex if and only if it is convex in the ordinary affine sense.
A {\em properly convex} subset of $\PS$ is a bounded convex subset of an affine subspace. (See \cite[Chapter 2]{Choi99}.)

A {\em side} of a compact properly convex set $P$
is a maximal convex subset of the boundary of $P$.
A {\em polytope} is a compact properly convex domain in $\PS$ with finitely many sides.
%The fundamental domain of a Coxeter orbifold
%will be homeomorphic to a polytope $P$ with finitely many sides $F_1 \dots, F_f$ of codimension-one
%and some higher codimensional sides removed. This gives rise to a natural cell-complex consisting of sides.
%Define $P^s = \{ x\in P : \Gamma_x \textrm{ is finite }\}$.
%Then $P$ is the subset of $P$ by removing some faces whose codimension is greater than two.
%Denote by $D_m$ the dihedral group of order $2m$.
By a {\em facet} of a polytope, we mean a side of $P$ of codimension-one.
By a {\em ridge} of a polytope, we mean a side of $P$ of codimension-two.
(A facet will be called a {\em face} and a ridge an {\em edge} if $P$ is three--dimensional.)
If $P$ is the base space of a Coxeter orbifold, then
each ridge where the facets $F_i$ and $F_j$ meet will be given an order $n_{ij} \geq 2$; ie, 
a ridge has an {\em order} $n_{ij}$ if a model neighborhood of each interior point of the ridge is given 
the usual product extension of the standard action of the dihedral group $D_{n_{ij}}$ of order $2 n_{ij}$ on the $2$--plane. 
%A Coxeter $n$--orbifold $\hP$ associated with $P$ is given as follows: the interior of each facet $F_i$ of $P$ is silvered,
%ie the singular locus is locally modeled on $\R^n/ \langle R\rangle$
%for a reflection $R$, and the interior of each ridge $F_i \cap F_j$
%is locally modeled on $(\R^n, D_{n_{ij}}, \phi)$ for some chart $\phi$
%where the dihedral group $D_{n_{ij}}$ action is given by two reflections;
%the ridge $F_i \cap F_j$ is said to be \emph{of order} $n_{ij}$.
%Every remaining point of $P$ in boundary of $P$ has
%a neighborhood modeled as above with a finite Coxeter group.

Given a Lie group $G$ acting on a manifold $X$ transitively, we can consider
a \emph{$(G, X)$--structure} on an orbifold $Q$
as a pair of an immersion $D:\widetilde Q \ra X$ and
a homomorphism $h:\pi_1(Q) \ra G$ satisfying
\[ h(\gamma) \circ D = D \circ \gamma \hbox{ for } \gamma \in \pi_1(Q).\]
For a given $(G, X)$-structure, $(D, h)$ is determined only up to the action 
\[ g(D, h(\cdot)) = (g\circ D, gh(\cdot)g^{-1}) \hbox{ for }  g \in G.\]
(In each case we are considering, $D$ is an embedding.)

%by a maximal atlas of charts, 
%each of form $(U, H, \phi)$ where $U$ is an open subset of $X$ and
%$H$ is a finite subgroup of $G$ and $\phi$ is a map $U \rightarrow \phi(U)$ inducing the quotient map $U \rightarrow U/H$
%with the compatibility maps in $G$.

%We study real projective structures and hyperbolic structures on a Coxeter $n$--orbifold $\hP$.
A \emph{real projective structure} on $Q$ is a $(G,X)$--structure on $Q$ with
\begin{displaymath}
G=\SL \quad \text{and} \quad X=\PS.
\end{displaymath}
(See also Section \ref{subs:ele}.)
%Note that there are notions of geodesics in $\PS$ covering projective geodesic arcs in $\rpn$.
%(We can equivalently define a real projective structure as a $(G, X)$--structure with
%$X = \mathbb{RP}^n, G = \PGL$. See P.128 \cite{Thurston97}.)
%A geodesic in $\PS$ is a connected arc in a $1$--dimensional great cicle in $\PS$.
%A hemisphere in $\PS$ is identifiable with an affine space in the manner preserving geodesics.
%A {\em properly convex} set in $\PS$ is a precompact convex subset of an affine space.

We can represent hyperbolic structures on
an $n$--orbifold using the Klein projective model: 
Let the Lorentzian inner product be given by 
\begin{displaymath}
 \langle x, y\rangle = -x_1 y_1 + x_2 y_2 +  \cdots + x_{n+1} y_{n+1},
\end{displaymath}
where $x_i$ for $i=1, \dots, n+1$ are components of $x \in V$ and 
$y_i$ for $i=1, \dots, n+1$ are ones for $y \in V$. 
The hyperbolic space $\mathbb{H}^n$ is an open ball $B$ in $\PS$ that is
the image under $\Pi$ of  positive time-like vectors. 
The group of hyperbolic isometries is the subgroup $\PO$ of $\SL$ acting on $B$.
Hence a hyperbolic Coxeter orbifold, being of form $\mathbb{H}^n/\Gamma$ for 
a discrete subgroup $\Gamma$ of $\PO$, naturally has an induced real projective structure.

Real projective structures have been studied by many mathematicians including Kuiper \cite{Kuiper53}, Benz{\'e}cri \cite{Benzecri60}, Koszul \cite{Koszul68}, Vinberg \cite{Vinberg71}, Goldman \cite{Goldman90, cg}, Choi \cite{cdcr1, cdcr2}, and Benoist \cite{Benoist01}.
Sometimes the topic is studied as the theory of linear representations of discrete groups by Koszul, Vinberg, Benoist, and so on. % \marginpar{added a senten}
Kac and Vinberg \cite{Vinberg67} were the first to discover hyperbolic Coxeter $2$--orbifolds
where the induced real projective structures deform into families of real projective structures that are not induced from hyperbolic structures.
Johnson and Millson \cite{Johnson87} constructed projective bending deformations of compact hyperbolic manifolds with embedded hypersurfaces.
Cooper, Long and Thistlethwaite \cite{Cooper06,Cooper07} investigated whether the closed hyperbolic $3$--manifolds of the Hodgson--Weeks census could be deformed and showed some occurrence of deformability. Benoist \cite{Benoist06}, 
Choi \cite{Choi06}, Marquis \cite{Marquis10}, and Choi, Hodgson and Lee \cite{Choi12} investigated classes of deformable projective Coxeter orbifolds. 
Heusener and Porti \cite{Heusener11} provided infinite families of hyperbolic $3$--manifolds that are projectively rigid
by Dehn filling. (See also Ballas \cite{Ballas2012}.) 
Surveys on real projective structures can be found in \cite{Benoist08} and \cite{Cbook}.

The \emph{deformation space $\De(Q)$ of real projective structures} on a closed orbifold $Q$ is the quotient space of 
the space of real projective structures on $Q$ by the action of the group of isotopies of $Q$.
The space has a natural $C^s$--topology for $s \geq 1$. (See \cite{Choi04} or \cite[Chapter 6]{Cbook}.)

%Let $\hP$ be a compact Coxeter $n$--orbifold.
%A point $p$ of $\De(\hat{P})$ gives a fundamental polytope $P$ in $\PS$,
%well defined up to projective automorphisms since the generating reflections determine $P$.
%The subspace of $p \in \De(\hat{P})$ giving a projectively fixed fundamental polytope $P$,
%which is called the \emph{restricted deformation space} of real projective structures on $\hat{P}$,
%was studied on by Choi \cite{Choi06} and Choi, Hodgson and Lee \cite{Choi12}.

%In this paper we shall only consider the deformation space $\De(\hP)$ of real projective structures on $\hP$ \emph{without} the restrictions.
Now we fix the dimension $n \geq 3$.
Let $P$ be an $n$--dimensional complete hyperbolic convex polytope with
dihedral angles that are submultiples of $\pi$; we call $P$ a \emph{hyperbolic Coxeter $n$--polytope}.
Then $P$ naturally has a Coxeter orbifold structure $\hP$ by silvering the facets. 
When a ridge has the dihedral angle $\tfrac{\pi}{n_{ij}}$,  the ridge has the order $n_{ij}$. 
%a model neighborhood of each point of the ridge is given 
%the usual product extension of the standard action of the dihedral group $D_{n_{ij}}$ of order $2 n_{ij}$ on the $2$--plane. 
%when the ridge has the dihedral angle $\tfrac{\pi}{n_{ij}}$.
%(See Proposition \ref{prop:hpolytope}.)
The point $t$ in $\De(\hP)$ is \emph{hyperbolic} if  a hyperbolic structure on $\hP$ represents $t$.

\begin{definition}
Let $P$ be a compact hyperbolic Coxeter $n$--polytope, and let $\hP$ denote $P$ with its Coxeter orbifold structure.
Suppose that $t$ is the corresponding hyperbolic point of $\De(\hP)$.
We call a neighborhood of $t$ in $\De(\hP)$ the \emph{local deformation space} of $\hP$ at $t$.
We say that $\hP$ is \emph{projectively deformable} at $t$, or simply \emph{deforms} at $t$,
if the dimension of its local deformation space at $t$ is positive.
Conversely, we say that $\hP$ is \emph{locally projectively rigid} at $t$, or \emph{locally rigid} at $t$, if the dimension of
its local deformation space at $t$ is $0$.
\end{definition}

%Given an ordering of facets of an $n$--polytope,
%we may just assume that $F_i < F_j$ if and only if $i < j$ by renumbering if necessary.
%\begin{definition}
%A Coxeter $n$--orbifold $\hP$ is \emph{weakly orderable} if the facets $F_1, \dots, F_f$ of $P$ can be labeled
%by integers $\{1, \dots, f\}$ so that
%each facet $F_i$ contains no ridges of order $2$ or at most $n$ ridges $e_{i_1}, \dots, e_{i_{m_i}}$ for $m \leq n$ of order $2$
%so that each $e_{i_k}$ is in a facet $F_{j_{i_k}}$ of higher index $j_{i_k} > i$ for $k = 1, \dots, m_i$.
%Thus, each face $F_i$ has a collection $\{F_{j_{i_1}}, \dots, F_{j_{i_{m_i}}}\}$ that could be empty.
%Furtheremore, we require that
%$F_{j_{i_1}}, \dots, F_{j_{i_{m_i}}}$ are in general position for each face $F_i$ if it is not empty.
%\end{definition}

\begin{definition}
Let $\hP$ be a compact Coxeter $3$-orbifold with a base polytope $P$. 
Then $\hP$ is \emph{weakly orderable} if the faces $F_1, \dots, F_f$ of $P$ can be labeled
by integers $\{1, \dots, f\}$ so that
for each face $F_i$, %contains %at most $3$ edges so that for each edge $E$, the edge order $E$ equal to $2$ and $E \subset F_j$ for $j > i$. 
the cardinality of 
\[\mathcal{F}_i :=  \{ F_j|\, j > i \hbox{ and the ridge $F_i \cap F_j$ has order $2$ } \}\]
is less than or equal to $3$. 
\end{definition}
In our case, the base polytope $P$ is always realizable as a convex polytope in an affine space
since $P$ is the fundamental polytope for a properly convex real projective Coxeter orbifold.
%\marginpar{added} 
A compact properly convex $n$--polytope $P$ is called \emph{simple} if exactly $n$ facets meet at each vertex.
Note that compact hyperbolic Coxeter $n$--polytopes are simple. Denote by $e_+(\hP)$ the number of ridges of order $\geq 3$ in $\hP$.

We now state two results of the paper that follows from Theorem \ref{thm:main1}, the main result of the paper. 
%\marginpar{need to change $n=3$. Drop $C1$.
\begin{corollary}\label{cor:main1}
Let $P$ be a compact hyperbolic Coxeter $3$--polytope, and suppose that $\hP$ is the Coxeter orbifold arising from $P$.
Suppose that $\hP$ is weakly orderable.
Then a neighborhood of the hyperbolic point $t$ in $\De(\hP)$ is a cell of dimension $e_+(\hP) \!- 3$.
\end{corollary}
A weakly orderable compact hyperbolic Coxeter $3$--orbifold $\hP$ is projectively deformable at $t$ if
$e_+(\hP) > 3$; otherwise, it is locally rigid at $t$.

%The numbers of facets and ridges of $P$ are denoted by $f$ and $e$ respectively. We introduce an integer
%\begin{equation*}
%\delta_P = e - nf + \tfrac{n(n+1)}{2}
%\end{equation*}
%which depends only on the polytope $P$. Barnette \cite{Barnette73} showed for simple polytopes $P$ that $\delta_P$ are non-negative.

A {\em truncation} at a vertex $v$ of a compact properly convex $n$--polytope $P$ is
an operation where 
\begin{itemize}
\item we take a hyperspace $H$ meeting only the interiors of sides of $P$  incident with $v$, and not $v$ itself,
\item take the component $C$ of $P- H$ containing $v$, and 
\item finally delete $C$. 
\end{itemize} 
%where $H$ is a hyperspace meeting only the sides of $P$  incident with $v$. 
%(This is actually a combinatorial operation well-defined up to combinatorial equivalences.) 
An {\em iterated truncation} of $P$ is the operation of obtaining $P_n$ where
\[ P= P_0 \Rightarrow P_1 \Rightarrow \cdots \Rightarrow P_n\] 
and $P_{i+1}$ is obtained from $P_i$ by truncation at a vertex of $P_i$. 
A \emph{truncation $n$--polytope} is a convex $n$--polytope obtained from an $n$--simplex by iterated truncation.   
%The process of truncation preserves the weak orderability by Lemma \ref{lem:trunc}.
%It is easy to show that the condition $(C1)$ of \fullref{thm:main1} holds for any truncation polytope.
%Moreover, if $\hP$ is any Coxeter orbifold arising from a truncation polytope $P$, then $\hP$ is weakly orderable
%as we can order facets by giving lower indices to all new facets arising from each step of the truncation process from the existing ones.
%\marginpar{need to prove somewhere.}

\begin{corollary}\label{cor:truncation}
Let $P$ be a compact hyperbolic Coxeter $n$--polytope and a truncation polytope. (Assume $n \geq 3$.)
Suppose that $\hP$ is the Coxeter orbifold arising from $P$.
%Suppose that $P$ is a truncation polytope.
Then a neighborhood of the hyperbolic point in $\De(\hP)$ is a cell of dimension $e_+(\hP) \! - n$.
\end{corollary}

%This was proved by Marquis  \cite{Marquis10} for $n=3$.
Earlier Marquis \cite{Marquis10} used the word \emph{ecimahedron} in place of \emph{truncation $3$--polytope}
and showed that if $\hP$ is the Coxeter $3$--orbifold arising from a compact hyperbolic Coxeter ecimahedron $P$,
then $\De(\hP)$ is diffeomorphic to $\mathbb{R}^{e_+(\hP)\!-3}$.
(For $n=2$, Goldman proved the result in his senior thesis.)
%We cannot yet prove the similar results for truncation $n$--polytopes
%for $n \geq 4$.

There is a recent thesis work done by Ryan Greene \cite{greene}
obtaining similar results using algebraic topological methods. Also, 
Theorem 1 of Kapovich \cite{Kapovich94} is an analogous result for the flat conformal structures. 

Our main results, Corollaries \ref{cor:main1} and \ref{cor:truncation}, will follow from Theorem \ref{thm:main1}
to be stated and proved in  Section \ref{subs:proofs}, generalizing the notion of the weak orderability.

%\begin{remark}
%Let $P$ be a simple $n$--polytope, and suppose that $\hP$ is a Coxeter orbifold arising from $P$.
%By Br{\o}ndsted \cite[$\S 19$]{Brondsted83}, for $n \geq 4$,
%$P$ is a truncation $n$--polytope if and only if $\delta_P=0$.
%Hence for $n \geq 4$,
%$P$ is a truncation $n$--polytope if and only if $\hP$ satisfies both $(C1)$ and $(C2)$.
%If $P$ is simple and $n=3$, then the Euler's formula implies that $\delta_P$ is always equal to $0$.
%\end{remark}

%\begin{corollary}
%Let $P$ be a compact hyperbolic Coxeter $3$--polytope, and suppose that $\hP$ is the Coxeter orbifold arising from $P$.
%If $\hP$ is weakly orderable,
%then a neighborhood of the hyperbolic point in $\De(\hP)$ is a smooth manifold of dimension $e_+ \!- 3$.
%\end{corollary}

Almost all the compact hyperbolic $3$--orbifolds arising from some $3$--polytopes are weakly orderable.
To describe this in more detail, we introduce the following terminology. 
An {\em abstract $3$--polyhedron} is a cell complex homeomorphic to a compact $3$--ball with conditions that 
there exists a unique $3$--cell, every $1$--cell belongs to exactly two $2$--cells, a nonempty intersection of 
two $2$-cells is a vertex or a $1$--cell, and every $2$--cell contains no fewer than three $1$--cells. 
(See \cite{Roeder07}). It is {\em simple} if each vertex is contained in three $1$--cells. 
The side structure of a properly convex $3$--polytope $P$ gives it
the structure of an abstract $3$--polytope whose $k$--cells correspond to the $k$--faces of $P$ for
$k=0,1,2,3$. 
The \emph{boundary complex} $\partial P$ of an abstract $3$--polyhedron $P$ is the subcomplex of $P$ consisting of all proper cells.
Let $(\partial P)^*$ be the dual complex of $\partial P$.
A simple closed curve $\beta$ is called a \emph{$k$--circuit} if it consists of $k$ edges of $(\partial P)^*$ for some positive integer $k$.
A circuit $\beta$ is \emph{prismatic} if all endpoints of the edges of $\partial P$ intersected by $\beta$ are distinct.

\begin{theorem}\label{thm:probability}
Let $P$ be a compact properly convex simple $3$--polytope but not a tetrahedron. 
Suppose that $P$ has no prismatic $3$--circuit, and has at most one prismatic $4$--circuit.
Let $H_d(P)$ be the number of compact hyperbolic Coxeter $3$--orbifolds whose base polytopes are combinatorially
equivalent to $P$ and the maximal edge orders are $\leq d$, 
and let $WO_d(P)$ denote the number of weakly orderable ones among them. 
Then
\[ \lim_{d \rightarrow \infty} \frac{ WO_d(P)}{H_d(P)}  =  1.\]
%\begin{equation*}
%\lim_{d \rightarrow \infty} \frac{ | \{ \text{weakly orderable compact hyperbolic Coxeter $3$--orbifolds $\hP$ with $o(\hP) = d$} \} | }{
%| \{ \text{compact hyperbolic Coxeter $3$--orbifolds $\hP$ with $o(\hP) = d$ } \}|  } = 1
%\end{equation*}
%where $o(\hP)$ is the maximum of edge orders of the Coxeter $3$--orbifold $\hP$.
\end{theorem}
In particular, there exist infinitely many weakly orderable hyperbolic Coxeter $3$--orbifolds with base polytopes of type $P$ as above. 
%These are all diffeomorphic by a remark after Question \ref{q:Davis}.

\begin{question}
Does the conclusion of Theorem \ref{thm:probability} still hold if we assume only that $P$ is a convex simple $3$--polytope?
\end{question}

%We organize the remainder of this paper as follows.

Section \ref{s:preliminary} reviews some facts.
In Section \ref{subs:ele} we review some elementary facts on orbifolds and real projective structures.
In Section \ref{subs:Vinberg} we describe needed Vinberg's results on the general properties satisfied by real projective reflection groups.
We turn Vinberg's ``semialgebraic'' conditions (L1) and (L2) into the ``algebraic'' conditions (L1) and (L2)$'$.
%In fact, the work of Vinberg is central to the theory of real projective structures on Coxeter orbifolds.
%In Section \ref{subs:Andreev} we recall Andreev's theorem characterizing the compact hyperbolic $3$--polytope with dihedral angles at most $\tfrac{\pi}{2}$.

Section \ref{s:DeformationSpaces} gives various descriptions of the deformation space of real projective structures on a Coxeter orbifold $\hP$.
In Section \ref{subs:RepresentationSpaces} we introduce a space of representations of the fundamental group $\pi_1(\hP)$ and
identify this representation space with the deformation space of real projective structures.
In Section \ref{subs:VinbergEq} we introduce the solution space of some polynomial equations,
a space of matrices satisfying certain conditions. 
We establish the equivalence of these spaces following Vinberg \cite{Vinberg71}.

Section \ref{s:LocalSpaces} discusses the results on a neighborhood of the hyperbolic structure in the deformation space of real projective structures on a compact Coxeter $n$--orbifold $\hP$.
In Section \ref{subs:TangentSpace} we study the Zariski tangent space of the solution space of polynomial equations giving real projective structures
on $\hP$. 
 In Section \ref{subs:HyperbolicEquations} we introduce the polynomial equations defining hyperbolic structures on $\hP$, and in Section \ref{subs:TangentToHyperbolic} we describe the Zariski tangent space of  %\marginpar{corrected}
the solution space of these polynomial equations. 
In Section \ref{subs:Maintheorem} we state Theorem \ref{thm:main1}, the main result of the paper. 
In Section \ref{subs:MainLemma} we compare the two Zariski tangent spaces at a hyperbolic point,
and in Section \ref{subs:Proof} we combine this with the Weil infinitesimal rigidity to prove Theorem \ref{thm:main1}.

Section \ref{s:Examples} provides several examples. In Section \ref{subs:Examples} we use two combinatorial results of Tutte and Fouquet--Thuillier to prove Theorem \ref{thm:probability}. In Section \ref{subs:C2} and Section \ref{subs:C1} we give examples satisfying only one of the two assumptions in Theorem \ref{thm:main1}.

\subsection*{Acknowledgements}
This work has benefitted from discussions and correspondences with Y Benoist, M Davis, W M Goldman, R Greene, 
O Guichard, S Lim, L Marquis and A Wienhard 
as well as many others. We especially wish to thank C D Hodgson for his advice and guidance to G-S Lee during our visit to
the Department of Mathematics and Statistics at the University of Melbourne. 
We are also grateful to the referee for carefully reading this paper and suggesting several improvements including Question \ref{q:ref}. 
Finally we express our gratitude to B Chriestenson for comments on our English.

S Choi was supported by the Mid--career Researcher Program through the NRF grant funded by the MEST (No. R01-2008-000-10286-0).
G-S Lee was supported by the National Research Foundation of Korea (NRF) grant funded by the Korea government (MEST) (No. 2011-0027952), the Basic Science Research Program through the NRF grant funded by the MEST (No. 2011-0004850), 
the Fondation Sciences Math\'{e}matiques de Paris (FSMP--IHP Program: Geometry and analysis of surface group representations)
 and the DFG research grant \lq\lq Higher Teichm\"{u}ller Theory\rq\rq.

\section{Preliminary}
\label{s:preliminary}

This section reviews the basic background material used in this article.
In Section \ref{subs:ele}, we review some basic materials on orbifolds and geometric structures
which are also given in Chapters 3, 4, and 6 of \cite{Cbook} in detail.
In Section \ref{subs:Vinberg} we describe Vinberg's results. 
%In Section \ref{subs:Andreev}
%we recall Andreev's theorem when a compact $3$--dimensional Coxeter orbifold admits a hyperbolic structure.

\subsection{Notation}\label{subs:ele}

An \emph{orbifold} $Q$ is a second countable Hausdorff space $|Q|$ with an \emph{orbifold structure}, ie, 
 a covering of $|Q|$ by a collection of 
open sets of form $\phi(U)$ for a model $(U, H, \phi)$ where $U$ is an open subset of $\R^n$ with a finite group $H$ acting on it effectively
and $\phi$ induces a homeomorphism $U/H \ra \phi(U)$. % for an open set $\phi(U)$ of $|Q|$.
We require that these models are compatible with one another in the standard way. 
A model $(U, H, \phi)$ is also called a {\em chart} of $Q$. 
A \emph{singular point} is a point $x$ of $Q$ where $H$ is not trivial for every choice of a chart $(U, H, \phi)$ containing $x$.
%with $\phi(U) \ni x$. 
%and the \emph{singular locus} is a set of singular points.

%For example, a manifold is an orbifold with charts that have trivial group actions.

An orbifold $Q_1$ {\em covers} an orbifold $Q_2$ by a {\em covering map} $p$ if each point of $Q_2$ has
a connected open neighborhood $\phi_2(U_2)$ with the model $(U_2, H_2, \phi_2)$ %and the inverse image $U'_1 =p^{-1}(\phi_2(U_2))$ 
%that is an open subset of $Q_1$ 
%where each of the component $\phi_1(U_1)$ of $U'_1$ has a model $(U_1, H_1, \phi_1)$ so that 
such that each component $\phi_1(U_1)$ of the inverse image $U'_1 =p^{-1}(\phi_2(U_2))$ has model $(U_1, H_1, \phi_1)$ where
%there is a lifting of $p\co  \hat U_1 \ra \hat U_2$
%to $\hat p\co  U_1 \ra U_2$ that is a diffeomorphism inducing an injective 
%homomorphism $H_1 \ra H_2$. 
\begin{alignat*}{3}
U_1   & \quad \stackrel{\phi_1}{\ra}  & \quad \phi_1(U_1) \\ 
\hat p \downarrow &    & \quad p \downarrow \\ 
U_2  & \quad \stackrel{\phi_2}{\ra}  & \quad \phi_2(U_2) 
\end{alignat*}
is commutative
for a diffeomorphism $\hat p$ equivariant with respect to an injective homomorphism $H_1 \ra H_2$. 

For an orbifold $Q$, we denote by $|Q|$ the base space of $Q$.
Two orbifolds $Q_1$ and $Q_2$ are {\em diffeomorphic} if a homeomorphism $f\co  |Q_1| \ra |Q_2|$
lifts to a smooth embedding for each choice of local model. 
%so that for each $x \in Q_1$, there exists a neighborhood $U'$ of $x$ with the model
%$(U_1, H_1, \phi_1)$ mapping to a neighborhood $f(U')$ of $f(x)$ with the model $(U_2, H_2, \phi_1)$
%with the commutative diagram
%\begin{alignat}{2} \label{eqn:diff}
%U_1 & \stackrel{f_1}{\longrightarrow}  & U_2 \nonumber \\
% \phi_1 \downarrow  &        & \phi_2  \downarrow \nonumber \\
%U' & \stackrel{ f| U'}{\longrightarrow} & f(U')
%\end{alignat}
%where $f_1$ is a diffeomorphism conjugating $H_1$ to $H_2$.
%We say that $f$ is a {\em diffeomorphism} of $Q_1$ to $Q_2$.
%We say that $f$ is a {\em homeomorphism} if we can choose $f_1$  to be a homeomorphism
%for each such pairs of models.

A {\em good orbifold} is an orbifold $Q$ that is covered by a manifold.
It has a simply connected covering manifold $\widetilde Q$ called a {\em universal cover}
with the covering map $p_Q$.
%It will be unique up to diffeomorphism of $\widetilde Q$ lifting the identity map of $Q$.
The group of diffeomorphisms $f\co  \widetilde Q \ra \widetilde Q$ so that $p_Q \circ f = p_Q$
is said to be the {\em deck transformation group} and is denoted by $\pi_1(Q)$.
The base space of $Q$ is homeomorphic to the quotient space $\widetilde Q/\pi_1(Q)$.

Conversely, given a simply connected manifold $M$ and a discrete group $\Gamma$ acting on it
properly discontinuously (but not necessarily freely), $M/\Gamma$ has a natural structure of
an orbifold. %  given by obtaining charts from open neighborhoods of points of $M$ with finite 
%subgroups of $\Gamma$ acting on them. 
%$M/\Gamma$ is said to be the {\em quotient orbifold}.

%Let $V$ be an $(n+1)$--dimensional real vector space.
%The projective sphere $\PS$ is the space of rays in $V$ and double covers the projective space $\mathbb{RP}^n$.
%The group $\SL$ acts on $\PS$ faithfully in the standard manner and double covers $\PGL$.
%The elements of $\SL$ are said to be the \emph{projective automorphisms} of $\PS$
%and $\SL$ the \emph{projective automorphism group} of $\PS$.
%(We will also think of $\SL$ as a linear group when convenient.)
%Denote by $\pi$ the natural projection from $V\backslash\{0\}$ into $\PS$.
%A subspace of $\PS$ is the image of a subspace of $V$ without the orgin.
%In particular, a $2$--dimensional subspace of $V$ corresponds to a great circle in $\PS$
%and an $n$--dimensional subspace gives a great $(n-1)$--sphere in $\PS$.
%Further, a component of the complement of a great $(n-1)$--sphere can be identified with an affine $n$--space.
%We call this an \emph{affine subspace} of $\PS$.

%\marginpar{Some repetitions here... Add or eliminate?}
A {\em geodesic} in $\mathbb{RP}^n$ is a connected arc  in a $1$--dimensional subspace.
A {\em geodesic} in $\PS$ is a connected arc in a $1$--dimensional great circle in $\PS$, which
is a lift of a geodesic of $\mathbb{RP}^n$. An {\em affine space} $A^n$ is $\mathbb{R}^n$ 
equipped with the affine transformation group acting on it. 
The complement of a codimension-one subspace in $\mathbb{RP}^n$ can be identified
with an affine space $A^n$. The group of projective transformations 
acting on $A^n$ is the affine transformation group. Moreover, the geodesics in $\PS$ 
restricts to geodesics in $A^n$.  %\marginpar{Refree pt...?} 

An open hemisphere in $\PS$ is identifiable with an affine space %in the manner preserving geodesics
%since it maps homeomorphically to the affine space 
under the double--covering map $\PS \ra \mathbb{RP}^n$.
An open hemisphere is said to be an {\em affine subspace} of $\PS$.
%A {\em properly convex} set in $\PS$ is a precompact convex subset of an affine space.
A {\em polytope} is a compact properly convex domain in an affine subspace with finitely many sides.
(For these, the ordinary theory of convex domains in the Euclidean space applies.)

%Let $P$ be an properly convex domain in $\PS$ with closure that is a polytope.
%A Coxeter $n$--orbifold $\hP$ associated with $P$ as a fundamental domain is given as follows:
%the interior of each facet $F_i$ of $P$ is silvered,
%ie the singular locus is locally modeled on $\R^n/ \langle R\rangle$
%for a reflection $R$, and the interior of each ridge $F_i \cap F_j$
%is locally modeled on $(\R^n, D_{n_{ij}}, \phi)$ for some chart $\phi$
%where the dihedral group $D_{n_{ij}}$ action is given by two reflections;
%the ridge $F_i \cap F_j$ is said to be \emph{of order} $n_{ij}$.
%Every remaining point of $P$ in boundary of $P$ has
%a neighborhood modeled as above with a finite Coxeter group.

%The {\em combinatorial polytope} is a compact cell with the structure of a cell-complex realizable as a side complex of
%a convex polytope.

For a Lie group $G$ acting  transitively on a smooth manifold $X$,
a \emph{$(G, X)$--structure} on an $n$-dimensional orbifold $Q$ is 
\begin{itemize}
\item a maximal atlas of charts of form $(U, H, \phi)$ where $U$ is an open subset of $X$ and
$H$ is a finite subgroup of $G$ acting on $U$,
%and $\phi$ is a map $U \rightarrow \phi(U)$ %  inducing the quotient map $U \rightarrow U/H$
where 
\item every inclusion map $\iota\co  \psi(V) \ra \phi(U)$ for charts $(V, J, \psi)$ and $(U, H, \phi)$ lifts to 
a map $k|V: V \ra U$ for $k \in G$ equivariant with respect to a homomorphism $J \ra H$, and 
\item  each point $x \in Q$ is in $\phi(U)$ where $(U, H, \phi)$ is in the maximal atlas of the orbifold structure of $Q$
where $U \subset X$ is identified with an open set in $\R^n$ by a smooth map.  (a compatibility condition) 
%\marginpar{added} 
\end{itemize}
The existence of $(G, X)$--structure implies that $Q$ is good. 
%\[\tilde \iota \circ g = \iota_*(g) \circ \tilde \iota \hbox{ for } g \in H, \iota_*(g) \in J\]
%where $\iota_*\co H \ra J$ is a homomorphism.
%The pair $(\tilde \iota, \iota_*(\cdot))$ is uniquely determined
%up to action \[k\cdot (\tilde \iota, \iota_*(\cdot)) = (k \circ \tilde \iota, k \iota_*(\cdot) k^{-1} ), k \in J.\]

Let $\hP$ denote a compact real projective $n$--orbifold.
Let $\widetilde{P}$ denote the universal cover of $\hP$ and let $\pi_1(\hP)$ denote the group of
deck transformations.
A real projective structure on $\hP$ gives us an immersion $D\co \widetilde{P} \ra \PS$, called a {\em developing map}, and
a homomorphism $h\co \pi_1(\hP) \ra \SL$, called a {\em holonomy homomorphism}, so that
$ D \circ \gamma = h(\gamma) \circ D$, for each $\gamma \in \pi_1(\hP)$.
Here, $(D,h)$ is determined only up to the following action:
\[ (D,h(\cdot)) \longmapsto (g\circ D, g\circ h(\cdot)\circ g^{-1}) \quad \textrm{for } g \in \SL.\]
Conversely, the development pair $(D, h)$ determines the real projective structure.
%Sending a real projective structure to its holonomy homomorphism $h$ gives us a local homeomorphism
%\[ \hol\co \De(\hP) \ra \Hom(\pi_1(\hP), \SL)/\SL, \]
%provided $\hP$ is a compact $n$--orbifold.
%\marginpar{sc: changed here}
(See Choi \cite{Choi04,Choi06, Cbook} and Thurston \cite{Thurston80,Thurston97} for the detail.)
%For quotient spaces, we are always using the quotient topology. Again, the book \cite{Cbook} is  an exposition that covers the basic material here.)

Note the double-covering map $\PS \ra \mathbb{RP}^n$ where the group $\PGL$ acts on $\mathbb{RP}^n$ transitively.
We can equivalently define a real projective structure as a $(G, X)$--structure with
$X = \mathbb{RP}^n, G = \PGL$ since $\SL$ is precisely the group consisting of
automorphisms of $\PS$ lifting elements of $\PGL$.
(See page 143 in \cite{Thurston97} and \cite{cdcr1}.)
%\marginpar{Check this.}

\subsection{Vinberg's results}
\label{subs:Vinberg}

This subsection gives a summary of the ground-breaking article of Vinberg \cite{Vinberg71}.
(See also Benoist \cite{Benoist09}.)
The English translated terminology of Vinberg is slightly different from the current ones.
For example the term ``strictly convex'' is now replaced by ``properly convex.''
%Also, the terminology and methods are somewhat classical here.
%Also, these are commonly used facts in the fields often without or with proofs.

Let $V$ be an $(n+1)$--dimensional real vector space and let $V^*$ its dual vector space.
%The projective sphere $\PS$ is the space of rays in $V$ and double covers the projective space $\mathbb{RP}^n$.
%The group $\SL$ acts on $\PS$ faithfully in the standard manner and double covers $\PGL$.
%The elements of $\SL$ are said to be the \emph{projective automorphisms} of $\PS$
%and $\SL$ the \emph{projective automorphism group} of $\PS$.
%(We will also think of $\SL$ as a linear group when convenient.)
%Denote by $\Pi$ the natural projection from $V\backslash\{0\}$ into $\PS$.
%A subspace of $\PS$ is the image of a subspace of $V$ without the orgin.
%In particular, a $2$--dimensional subspace of $V$ corresponds to a great circle in $\PS$
%and an $n$--dimensional subspace gives a great $(n-1)$--sphere in $\PS$.
%Further, a component of the complement of a great $(n-1)$--sphere can be identified with an affine $n$--space.
%We call this an \emph{affine subspace} of $\PS$.

Let $O$ denote the origin of $V$.
A {\em cone} $C$ in $V$ is a subset of $V$ where $O \in C$ so that if any point $v \in V$ is in $C$, then $s v \in C$ for each $s > 0$.
In our terms, the definition of {\em convex cone} in \cite{Vinberg71} is the following:
A cone $C$ in $V$ is a {\em convex cone}  if $\Pi(C - \{O\})$ is a convex set in $\PS$.

A \emph{reflection} $R$ is an element of order $2$ of $\SL$ which is the identity on a hyperplane of $V$.
All reflections are of the form
\begin{equation*}
R = I_V - \alpha \otimes b
\end{equation*}
for some linear functional $\alpha \in V^*$ and a vector $b \in V$ with $\alpha(b)=2$, and are in $\SL$.
Observe that the kernel of $\alpha$ is the subspace $U$ of fixed points of $R$ and $b$ is
the \emph{reflection vector}, ie an eigenvector corresponding to the eigenvalue $-1$.
Hence a reflection has a subspace of codimension-one in $\PS$ as the set of fixed points and the point corresponding 
to the reflection vector is sent to its antipode, called the {\em antipodal fixed point}.

A \emph{rotation} is an element of $\SL$ which restricts to the identity on
a subspace of codimension-two and acts on the complementary space with a matrix
$\left[ \begin{smallmatrix} \cos\theta & -\sin\theta \\
\sin\theta & \;\;\; \cos\theta \end{smallmatrix} \right]$
with respect to some basis. The real number $\theta$ is the \emph{angle} of the rotation.

%A subset $\Omega$ of $\PS$ is \emph{convex} if its intersection with any great circle is connected.
%Moreover, it is \emph{properly convex} if, in addition, its closure $\overline{\Omega}$ does not contain a pair of antipodal points.
%A subset $P$ of $\PS$ is a \emph{properly convex $n$--polytope in $\PS$} if $P$ is a properly
%convex $n$--polytope in an affine subspace of $\PS$.
%Observe that a convex $n$--polytope in $\PS$ is properly convex.
As a matter of notation, given a convex $n$--polytope $P$ in $\PS$,
$\mathrm{cone}(P)$ will denote the convex polyhedral cone $\Pi^{-1}(P)\cup \{O\}$ in $V$.

Let $P$ be a properly convex $n$--polytope in $\PS$ with sides $F_1, \dots, F_f$ of codimension-one,
and for each facet $F_i$ of $P$,
take a linear functional $\alpha_i$ for $F_i$ and choose a reflection $R_i=I_V-\alpha_i \otimes b_i$ with $\alpha_i(b_i)=2$ which fixes $F_i$.
By making a suitable choice of signs, we assume that  $P$ is defined by the following inequalities
\begin{equation*}
\alpha_i \geq 0, \hspace{10mm} i \in \I=\{1, \ldots, f\}.
\end{equation*}
The group $\Gamma \subset \SL$ generated by all these reflections $R_i$ is called a (\emph{real}) \emph{projective Coxeter group} if
\begin{equation*}
\gamma \mathring{P} \cap \mathring{P} = \emptyset \:\;\:\text{for every}\: \gamma \in \Gamma \backslash \{1\}
\end{equation*}
where $\mathring{P}$ is the interior of $P$. Note that Vinberg \cite{Vinberg71}
used the word \emph{linear Coxeter group} in place of \emph{projective Coxeter group}.
The $f \times f$ matrix $A =(a_{ij})$, $a_{ij}=\alpha_i(b_j)$, is called the \emph{Cartan matrix} of the group $\Gamma$ and
$P$ is called a \emph{fundamental chamber} of $\Gamma$.
%We denote by $P^s$
For $x \in P$, let $\Gamma_x$ denote the subgroup of $\Gamma$ generated by $\{ R_i |\, x \in F_i \}$.
Define $P^s = \{ x\in P |\,  \Gamma_x \textrm{ is finite }\}$.

%%% Feb 27 12:31pm

\begin{theorem}{\rm \cite[Theorem 1 and Propositions 6 and 17]{Vinberg71}} \label{thm:vinbergM}
The following conditions are necessary and sufficient for any group $\Gamma$,
generated by the reflections $R_1, \dots, R_f$ fixing respectively facets $F_1, \dots, F_f$ 
of the properly convex $n$--polytope $P$, 
to be a projective Coxeter group\,{\rm :}
\begin{enumerate}\label{ConditionL}
\item[{\rm (L1)}] $a_{ij} \leq 0$ for $i \neq j$, and $a_{ij}=0$ if and only if $a_{ji}=0$.
\item[{\rm (L2)}] $a_{ii}=2$; and for $i, j$ with $i\ne j$,
\begin{itemize}
\item[{\rm (i)}]  if $F_i$ and $F_j$ are adjacent, ie meet in a ridge, then $a_{ij}a_{ji}=4 \cos^2 \left(\frac{\pi}{n_{ij}}\right) < 4$ for 
an integer $n_{ij} \geq 2$, or  
\item[{\rm (ii)}] else  $a_{ij}a_{ji}\geq 4$. 
\end{itemize}
%\item[(L3)] For  $i, j$ with $i\ne j$ and $F_i \cap F_j = \emp$, we require $a_{ij}a_{ji} \geq 4$.
\end{enumerate}
\end{theorem}
\begin{proof}
Proposition 17 of \cite{Vinberg71} gives the necessity of the conditions (L1) and (L2).

Given (L1) and (L2), Proposition 7 and Theorem 1 of \cite{Vinberg71} show that
$\Gamma$ is a projective Coxeter group with the fundamental chamber $P$.
(See also \cite[Theorem 1.5]{Benoist09}.) 
\end{proof}

In fact, if $a_{ij}a_{ji}=4 \cos^2\left(\frac{\pi}{n_{ij}}\right)$, then the product $R_i R_j$ is a rotation of angle $\frac{2\pi}{n_{ij}}$
and the group generated by the two reflections $R_i$ and $R_j$ is the dihedral group $D_{n_{ij}}$.
In particular, if $a_{ij} = a_{ji} = 0$ then $R_iR_j$ is a rotation of angle $\tfrac{\pi}{2}$ and
$R_i$ and $R_j$ generate a dihedral group of order $4$, ie a Klein four group.
If $a_{ij}a_{ji} \geq 4$ then $R_i$ and $R_j$ would generate an infinite group and $n_{ij} = +\infty$.
(See \cite[Section 2]{Vinberg71}.)

The group generated by $R_1, \dots, R_f$ is isomorphic to a Coxeter group, and is also called
the {\em projective reflection group} generated by $R_1, \dots, R_f$.

For each reflection $R_i$,  $\alpha_i$ and $b_i$ are determined up to a positive scalar by: 
\begin{equation}\label{eqn:diagonal}
\alpha_i \mapsto d_i\alpha_i \;\; \text{and} \;\; b_i \mapsto d_i^{-1}b_i  \text{ with } d_i  > 0.
\end{equation}
Hence the Cartan matrix $A$ of $\Gamma$ is determined up to the conjugation action of a group of diagonal $f\times f$-matrices
with positive diagonal entries.

%For any $x \in P$, let $\Gamma_x$ denote the stabilizer subgroup of $\Gamma$ of $x$.

\begin{theorem}{\rm \cite[Theorem 2]{Vinberg71}}\qua \label{thm:vinberg0}
Let $\Gamma$ be a projective Coxeter group and $P$ its fundamental chamber.
The subset $\{x \in P |\,  \Gamma_x \text{ is finite} \}$ is denoted by $P^s$. Then the following statements hold\/{\rm :}
\begin{itemize}
\item $\Omega_\Gamma=\cup_{\gamma \in \Gamma} \gamma P$ is convex.
\item $\Gamma$ is a discrete subgroup of $\SL$ preserving 
the interior $\mathring{\Omega}_\Gamma$ of $\Omega_\Gamma$.
\item $\mathring{\Omega}_\Gamma \cap P = P^s$, and is homeomorphic to $\mathring{\Omega}_\Gamma / \Gamma$.
\end{itemize}
\end{theorem}

%Because the properly convex open domain has a natural invariant metric, called a Hilbert metric (see \cite{BK}),
An easy consequence of the theorem is that 
the group $\Gamma$ acts on $\mathring{\Omega}_\Gamma$ properly discontinuously.
Thus $\mathring{\Omega}_\Gamma$ gives a convex open subset of the projective sphere $\PS$, and
$\mathring{\Omega}_\Gamma/ \Gamma$ determines a \emph{convex} real projective structure
on the Coxeter $n$--orbifold  with the fundamental domain homeomorphic to $P^s$.
For example, let $P$ be a hyperbolic Coxeter $n$--polytope of finite volume.
Suppose that $\Gamma$ is the discrete group generated by the isometric reflections with respect to facets of $P$ in the hyperbolic space $\mathbb{H}^n$
in the Klein model in $\PS$. 
Then $\mathring{\Omega}_\Gamma=\mathbb{H}^n$ and $\mathring{\Omega}_\Gamma / \Gamma$ is a hyperbolic Coxeter $n$--orbifold.
%(We will study only compact ones here.)

%Here the matrix $(a_{ij})$ is said to be the {\em Cartan matrix} of the Coxeter group $\pi_1(\hP)$ or $\hP$.

A projective Coxeter group $\Gamma$ is \emph{elliptic}, \emph{parabolic} and \emph{hyperbolic}
if $\Gamma$ is conjugate to a discrete group generated by reflections in the sphere, the Euclidean space 
and the hyperbolic space respectively, provided that 
neither any proper plane in the hyperbolic space nor any point at infinity is $\Gamma$--invariant.

A Cartan matrix is  \emph{indecomposable} if it is not a direct sum of two matrices.
Thus every matrix $A$ decomposes into a direct sum of indecomposable matrices, which are  \emph{components} of $A$.
By Frobenius's theorem, any indecomposable matrix $A$ satisfying the condition (L1) has a real eigenvalue 
 (see Gantmacher \cite{Gantmacher59}).
An indecomposable Cartan matrix $A$ is \emph{positive}, \emph{zero} and \emph{negative type} if the smallest real eigenvalue is positive, zero and negative respectively.
Denote by $A^+$ (resp. $A^0$, $A^-$) the direct sum of its components of positive type (resp. zero type, negative type).
Any matrix $A$ satisfying the condition (L1) is the direct sum of $A^+$, $A^0$ and $A^-$.

%We state the following deep results of Vinberg without proofs as they are not trivial.
\begin{proposition}{\rm \cite[Proposition 22]{Vinberg71}}\qua \label{prop:vinberg1}
Let $\Gamma$ be a projective Coxeter group with a properly convex $n$-dimensional fundamental chamber in $\PS$, and let $A$ be the Cartan matrix of $\Gamma$.
Then $\Gamma$ is elliptic if and only if $A=A^+$ if and only if $\Gamma$ is finite.
%\item $\Gamma$ is parabolic if and only if $A=A^0$ and $\rank A = n$.
%\end{itemize}
\end{proposition}

%\marginpar{added} 
\begin{proposition}{\rm \cite[Proposition 23]{Vinberg71}}\qua \label{prop:vinberg1ii}
Let $\Gamma$ be a projective Coxeter group with a properly convex $n$-dimensional fundamental chamber in $\PS$, 
and let $A$ be the Cartan matrix of $\Gamma$.
%\begin{itemize}
%\item $\Gamma$ is elliptic if and only if $A=A^+$ if and only if $\Gamma$ is finite.
Then $\Gamma$ is parabolic if and only if $A=A^0$ and $\rank A = n$.
%\end{itemize}
\end{proposition}

We shall consider only the case when $P=P^s$, or equivalently, $\Omega_\Gamma=\mathring{\Omega}_\Gamma$;
we call $\Gamma$ \emph{perfect}. The following three statements are equivalent: (i) $\Gamma$ is perfect, (ii) the base space $P^s$ of
the associated orbifold $\hP$ equals $P$ exactly, and (iii) $\hP$ is compact.

The following is a fairly well-known and commonly used consequence of \cite{Vinberg71}.
\begin{proposition}{\rm \cite[Lemma 15 and Propositions 19 and 26]{Vinberg71}}\qua \label{prop:vinberg2}
Let $\Gamma$ be a perfect projective Coxeter group with a properly convex $n$-dimensional fundamental chamber $P$ in $\PS$ 
and let $A$ be the Cartan matrix of $\Gamma$.
Then exactly one of the following statements holds\/{\rm :}
\begin{itemize}
\item $\Gamma$ is elliptic.
\item $\Gamma$ is parabolic.
\item $A$ is indecomposable and of negative type, and $\rank A =\dim V= n+1$.
\end{itemize}
Moreover, if $\Gamma$ is neither elliptic nor parabolic, then
$\Gamma$ is irreducible and $\Omega_\Gamma$ is properly convex.
\end{proposition}
\begin{proof}
By Proposition 26 of \cite{Vinberg71}, we have only the above three possibilities or
$\Gamma$ is the direct product of a parabolic group and $\bZ/2\bZ$.
In this case, $\Gamma$ is not perfect as we can see from Lemma 17 of \cite{Vinberg71}.

In the third case, we only have to prove the last statement. 
%By Proposition 18 of \cite{Vinberg71}, we obtain $d = 0$ for $d$ defined as in \cite{Vinberg71}
Since our fundamental domain $P$ is properly convex, 
Proposition 18 of \cite{Vinberg71} implies that $\Gamma$ is reduced. %and the defect $d=0$. % in Vinberg's terminology.
%Since the Cartan matrix $A$ of $\Gamma$ is then indecomposable, %and $d=0$,
%we obtain that $\Gamma$ is irreducible by Corollary to Proposition 19 of \cite{Vinberg71}.
Lemma 15 of  \cite{Vinberg71} implies the last statement.
\end{proof}

Let $P$ be a properly convex $n$--polytope in $\PS$ and the polyhedral cone $K=\mathrm{cone}(P)$ be given.
Again a {\em side} of $K$ is a maximal convex subset of $K$. 
The complex of $K$, denoted by $\FK$, is the set of its (closed) sides, partially ordered by inclusion.
Let $K_1, \dotsc, K_f$ be the facets of $K$, and let $\I =\{1, \dots, f\}$. 
For any sides $L$ of $K$, define 
\[\sigma(L)=\{ i \in \I |\,  K_i \supset L\} \hbox{ and } \sigma(\FK) =\{ \sigma(L) \subset \I |\, L \in \FK\}.\]
For any subset $S$ of $\I$, the \emph{standard subgroup} $\Gamma_S$
of $\Gamma$ is the subgroup generated by the reflection $R_i$, $i \in S$,
and the \emph{principal submatrix} $A_S$ of $A$ is the submatrix of $A$ consisting of the entries $a_{ij}$ for each $i,j \in S$.
Denote by $S^+$ (resp. $S^0$, $S^-$) the subset $T$ of $S$ such that $A_T=A_S^+$ (resp. $A_S^0$, $A_S^-$).
We define $Z(S) := \{ i \in \I |\,  a_{ij} =0 \hbox{ for each } j \in S\}$.

\begin{proposition}{\rm \cite[Theorems 4]{Vinberg71}}\qua \label{prop:vinberg3}
Let $\Gamma$ be a projective Coxeter group, let
$P$ be its fundamental chamber and let $K$ be $\mathrm{cone}(P)$.
Assume that a subset $S$ of $\I$ satisfies two conditions {\rm :} $S=S^0$ and $Z(S)^0 = \emptyset$.
Then $S \in \sigma(\FK)$.
\end{proposition}
%Note that the combinatorial structure of the fundamental chamber of a perfect projective Coxeter group $\Gamma$ is completely determined by
%the abstract group structure of $\Gamma$.

%For a polytope $P$, we will denote by $\I_P$ the set of indices of facet of $P$.

\begin{proposition}{\rm \cite{Vinberg71}}\qua \label{prop:vinberg3ii}
Let $\Gamma$ be a perfect projective Coxeter group,
let $P$ be its fundamental chamber and let $K=\mathrm{cone}(P)$.
%Assume that a subset $S$ of $\I$ satisfies two conditions: $S=S^0$ and $Z(S)^0 = \emptyset$.
Then $S \in \sigma(\FK)$ if and only if $\Gamma_S$ is finite or $S=\I$.
\end{proposition}
\begin{proof} This is the statement of equation 8 in \cite{Vinberg71}.
\end{proof}

\begin{lemma}\label{lem:parabolic}
Let $\Gamma$ be a perfect projective Coxeter group, and let $A$ be the Cartan matrix of $\Gamma$.
If $A$ has a principal submatrix of zero type, then $\Gamma$ is parabolic.
\end{lemma}
\begin{proof}
Suppose that $S=S^0$ for some nonempty $S\subset \I$.
Define %$\text{Z}(S)=\{i \in \I : a_{ij}=0 \text{ for each } j \in S\}$ and
$T:=\text{Z}(S)^0$.
Observe that $S \cup T = (S \cup T)^0$ and $\text{Z}(S \cup T)^0 = \emptyset$, and thus
by Proposition \ref{prop:vinberg3}, $S \cup T \in \sigma(\FK)$ with $K=\mathrm{cone}(P)$.

Suppose that $S \cup T \ne \I$. Then $\Gamma_S$ is finite  by Proposition \ref{prop:vinberg3ii}, and 
$S = S^0$ should be empty, a contradiction.

If we have $S \cup T = \I$, then $\Gamma$ is either elliptic or parabolic by Proposition \ref{prop:vinberg2}
as $\I^0 = \I$. If $\Gamma$ is elliptic, then $\I^+ = \I$, a contradiction as $S$ is not empty.
Hence, $\Gamma$ is parabolic.
%By assumption, $\Gamma$ is not parabolic.
%$\S \cup T$ is not parabolic.
%By Theorem \ref{thm:vinberg1ii},
%\fullref{thm:vinberg1}, \ref{thm:vinberg1ii}, and \ref{thm:vinberg2} show that $S \cup T \neq \I$.
%By \fullref{thm:vinberg3}, $\Gamma_{S \cup T}$ is finite.
%So is $\Gamma_{S}$, ie $S=S^+$ by \fullref{thm:vinberg1} and  \ref{thm:vinberg1ii}. This is a contradiction.
(See also the proof of Theorem 7 of \cite{Vinberg71}.)
\end{proof}

\begin{proposition}\label{prop:parabolic2}
Let $\Gamma$ be a perfect projective Coxeter group, and let $A = (a_{ij})$ be the Cartan matrix of $\Gamma$.
If $\Gamma$ is not parabolic, then $a_{ij}a_{ji} >  4$ holds if $F_i$ and $F_j$ are not adjacent and $i \neq j$.
\end{proposition}
\begin{proof}
If $a_{ij}a_{ji} = 4$ holds for some $i \neq j$, then the principal submatrix
$\bigl[ \begin{smallmatrix}
2 & a_{ij} \\
a_{ji} & 2
\end{smallmatrix} \bigr]
$
of $A$ is of zero type. By Lemma \ref{lem:parabolic}, $\Gamma$ is parabolic.
\end{proof}

Proposition \ref{prop:parabolic2} shows that for negative-type perfect projective Coxeter groups,
we can now replace the semialgebraic condition (L2) with
an open condition (L2)$'$ where we replace (L2)(ii) with
\begin{description}
\item[(L2)$'$(ii)]  if $F_i$ and $F_j $ are not adjacent, then $a_{ij}a_{ji} > 4$.
\end{description}

%The main use of Proposition \ref{prop:parabolic2} is that we can now replace the semialgebraic condtion (L2) with
%an open condition
%\begin{itemize}
%\item[(L2)$'$] $a_{ii}=2$; and for $i, j$ with $i\ne j$,
%\begin{itemize}
%\item[(i)]  if we have $F_i \cap F_j = \emptyset$, then $a_{ij}a_{ji}>  4$,  or
%\item[(ii)] else  $a_{ij}a_{ji}=4 \cos^2 (\frac{\pi}{n_{ij}}) < 4$, $n_{ij}$ an integer $\geq 2$.%
%\end{itemize}
%\end{itemize}
%for perfect groups.
The following was one of the main results of \cite{Vinberg71}.

\begin{theorem}{\rm \cite[Corollary 1]{Vinberg71}}\qua\label{thm:Vinberg}
Let $A$ be an $f \times f$ matrix satisfying {\rm (L1)} and {\rm (L2),} and let $\rank A=n+1$.
Suppose that $A$ has no component of zero type. Then
there exists a projective Coxeter group $\Gamma \subset \SL$ with the Cartan matrix $A$.
Furthermore,
$\Gamma$ is unique up to the conjugations in $\SL$.
\end{theorem}

\section{Deformation spaces of real projective structures}
\label{s:DeformationSpaces}

Through this section, we give three descriptions of the deformation space of real projective structures on 
a compact $n$--dimensional Coxeter orbifold $\hP$,
when $\hP$ admits a real projective structure but does not admit a spherical or Euclidean structure.
In Section \ref{subs:RepresentationSpaces}, we describe the deformation space in terms of representations from $\pi_1(\hP)$ into $\SL$.
In Section \ref{subs:VinbergEq}, we describe this representation space in terms of polynomial equations and Cartan matrices 
following Vinberg respectively.

\subsection{Deformation spaces and the representation spaces}
\label{subs:RepresentationSpaces}

%\subsection{Developing maps of Coxeter $n$--orbifolds} \label{subs:Coxeterorb}

%Let ${\uP}$ denote the universal cover of ${\hP}$.
%The deck transformation group $\pi_1({\hP'})$ is isomorphic to $\pi_1({\hP}')$.
%Given a real projective structure on ${\hP}$, we can define a developing map
%$D\co {\uP} \ra \PS$ equivariant with respect to a holonomy homomorphism
%$h\co \pi_1({\hP}) \ra \SL$.
%The inverse image in $\uP$ of
%the singular set of ${\hP}$ is the union of hypersurfaces fixed by reflection elements of $\pi_1({\hP})$.
%Call these the {\em reflection hypersurfaces}.
%Then ${\uP}$ has a compact fundamental region $F$ bounded by a union of reflection hypersurfaces meeting
%transversally. (The fundamental chamber is the base space as we can see here.) 

We restate the results of Vinberg \cite{Vinberg71} for the perfect groups in the orbifold-viewpoint. 
%We also use 
%the Kuiper completions which simplify the old proofs somewhat in our opinion. 
\begin{proposition}[Vinberg] \label{prop:devCox}
Let $\hP$ be a compact real projective Coxeter $n$--orbifold where
$\hP$ does not admit a spherical or Euclidean structure.
Then each developing map $D$ of the universal cover $\uP$ of $\hP$ is a diffeomorphism to
an open properly convex domain in $\PS$. Furthermore, $D(P)$ is a fundamental chamber 
for the fundamental domain $P$ of $\uP$, a properly convex $n$--polytope, %\marginpar{revised} 
the projective Coxeter group $h(\pi_1(\hP))$ where 
$h: \pi_1(\hP) \ra \SL$ is the holonomy homomorphism associated with $D$. 
\end{proposition}

Given a Coxeter orbifold $\hP$, the choice of the fundamental polytope $P$ in the universal cover gives 
us the fundamental set of generators in $\pi_1(\hP)$ associated with each facet of $P$. 
They are labelled by $r_1, \dots, r_f$ where $f$ is the number of facets of $P$. 
Call these the {\em fundamental generators}. %\marginpar{fundamental generators added}
Since we can imbed $\Hom(\pi_1(\hP), \SL)$ as an algebraic subset of $\SL^f$ for the number $f$ of fundamental generators of $\pi_1(\hP)$,
we let $\Hom(\pi_1(\hP), \SL)$ be a real algebraic set with the standard point-set topology of the subspace.

The $\SL$--action on $\Hom(\pi_1(\hP), \SL)$ by conjugation is not effective since $\pm I_V$ is in the kernel and
\begin{equation*}
\Hom(\pi_1(\hP), \SL)/\SL
\end{equation*}
is equivalent to
\begin{equation*}
\Hom(\pi_1(\hP), \SL)/\PGL.
\end{equation*}
We will study the later space only.

A discrete subgroup $\Gamma$ of $\SL$ is {\em dividing } if $\Gamma$ acts faithfully and properly discontinuously
on a properly convex open subset $\Omega$ of $\PS$ 
so that the quotient $\Omega/\Gamma$ is compact. (See Benoist \cite{Benoist01}.) 
%acts on a properly convex open domain $\Omega \subset \PS$ as a reflection group
%with the compact quotient $\Omega/\Gamma$.
%A Coxeter group acts on a properly convex open
%domain as a {\em reflection group} if each standard generator acts as a reflection with compact fundamental polytope. 
%Denote by $D_{\rm rep}(\hP)$ the subspace of
%the discrete faithful representations of form $h \co \pi_1(\hP) \ra \SL$ where 
%the image $\Gamma=h(\pi_1(\hP))$ \emptyseth{dividing}
%a properly convex open subset $\Omega$ of $\PS$,
Let $D'_{\rm rep}(\hP)$ denote the space of dividing faithful representations of $\pi_1(\hP)$. 
Define the subspace $D_{\rm rep}(\hP)$ by $h \in D_{\rm rep}(\hP)$ if $h$ is discrete and faithful and 
$h(\pi_1(\hP))$ acts as a dividing projective Coxeter group on a properly convex domain. %\marginpar{A bit more precise}

\begin{question}\label{q:ref}
Is $D'_{\rm rep}(\hP) = D_{\rm rep}(\hP)$ for a compact Coxeter $n$-orbifold $\hP$? %(We are unable to answer this.)
\end{question} 

We combine the works of Benoist, Charney--Davis, Choi, Koszul, Qi, and finally Davis to prove the following theorem.
Let $\De(\hP)$ be the deformation space of real projective structures on $\hP$. (Of course, this set could be empty.)

\begin{theorem}\label{thm:defrepspaces}
Let $\hP$ be a compact Coxeter $n$--orbifold.
Assume that $\hP$ admits a real projective structure, but does not admit a spherical or Euclidean structure.
\begin{itemize}
\item $D_{\rm rep}(\hP)$ is a union of components of $\Hom(\pi_1(\hP), \SL)$,
$\PGL$ acts properly and freely on it, and the quotient space $D_{\rm rep}/\PGL$ is a Hausdorff space.
\item $ \De({\hP}) \ra D_{\rm rep}(\hP)/\PGL$  is a homeomorphism.
%where we define ${\hP}' \sim \hP$ if ${\hP}'$ 
%is a Coxeter $n$--orbifold homomorphic to $\hP$ admitting a real projective structure.
%\item The deformation space $\De(\hP)$ of real projective structures on the Coxeter orbifold $\hP$ is homeomorphic to
%a union of components of $D_{\rm rep}(\hP)/\PGL$.
\item For each element $h \co \pi_1(\hP) \ra \SL$ of $D_{\rm rep}(\hP)$, 
the sphere $\PS$ contains a unique properly convex open subset $\Omega$ of $\PS$
so that $\Omega/h(\pi_1(\hP))$ is diffeomorphic to $\hP$. Here
$\Omega$ is determined up to the antipodal map $\mathcal{A}:= -I_V$
 %and is diffeomorphic to $\hP$ if $h$ is in the same component of
%$D_{\rm rep}(\hP)$ as the hyperbolic one.
%\item The component of
%$D_{\rm rep}(\hP)/\PGL$ containing the conjugacy class of the holonomy homomorphism
%of the hyperbolic structure of $\hP$ corresponds to the component of $\De(\hP)$ containing
%the hyperbolic structure.
\end{itemize}
\end{theorem}
\begin{proof}
%Let $D'_{\rm rep}(\hP)$ denote the space of dividing faithful representations of $\pi_1(\hP)$. 
The fundamental group $\pi_1(\hP)$ of $\hP$ is an infinite, non-affine and irreducible Coxeter group by Proposition \ref{prop:vinberg2}.
Hence, by Theorem 1.1 of Qi \cite{Qi07}, the center of any finite-index subgroup of $\pi_1(\hP)$ is trivial, and so
by Benoist \cite[Theorem 2.2]{Benoist01} (or \cite[Theorem 1.1]{Benoist05}), $D'_{\rm rep}(\hP)$ is a union of components of 
$\Hom(\pi_1(\hP), \SL)$, consisting of dividing discrete faithful representations. 
(For each $h \in D'_{\rm rep}(\hP)$, $\Omega/h(\pi_1(\hP))$ for a properly convex domain $\Omega$ is a compact orbifold
by Benoist \cite{Benoist05}.) 

Now, $D_{\rm rep}(\hP)$ is an open subset of $D'_{\rm rep}(\hP)$ by Koszul \cite{Koszul68}. (See \cite{Choi06}.)  %\marginpar{Referee correction} 
The subset is closed in the second space: Let $h_i$ be a sequence of representations in $D_{\rm rep}(\hP)$ 
converging to an element $h$ of $D'_{\rm rep}(\hP)$. 
For a set of fundamental generators $r_j$, $j=1, \dots, f$,
$h_i(r_j)$ is a reflection fixing points of a side $F_{j, i}$ of a compact convex polytope $P_i$. 
%Then $h_i(r_j)$ fixes points of the side $F_{j, i}$ of a compact convex polytope $P_i$ for every $i$. 
Let $\Omega_i$ be a properly convex domain in $\PS$ where $h_i(\pi_1(\hP))$ acts as a projective Coxeter group. 
We have 
$h_i(r_j) \ra h(r_j)$ for each $j$ where $h(r_j)$ is a reflection and fixes points of a hyperspace $H_j$.
Since $h(\pi_1(\hP))$, $h\in D'_{\rm rep}(\hP)$, 
acts on a properly convex open domain $\Omega \subset \PS$, each $H_j$ meets $\Omega$. 
Here, $\{H_l\}_{l=1, \dots, f}$ are mutually distinct since otherwise we loose the faithfulness of the action. 

Denote by $H_{l, i}$ the hyperspace in $\PS$ fixed by $h_i(r_l)$.
%Also, $H_j$ are mutually distinct since otherwise we loose the discrete faithfulness of $h$. 
%If $H_l \cap H_m \cap \Omega =\emptyset$, then $r_l r_m$ is of infinite order since $H_l \cap \Omega$ and $H_m \cap \Omega$ 
%are separating hypersurfaces. Since $\Gamma$ acts properly discontinuously on $\Omega$, 
%$H_l \cap H_m \cap \Omega \ne \emptyset$ if and only if $r_l r_m$ is of finite order.
For a subset $S$ of $\{1, \dots, f\}$, 
let $\Gamma_{S}$ denote the subgroup of $\pi_1(\hP)$ generated by $r_l$ for $l \in S$.
Since $h_i(\pi_1(\hP))$ acts properly discontinuously on $\Omega$,
$h_i(\Gamma_{S})$ is finite if $\bigcap_{l \in S} H_{l, i} \cap \Omega_i \ne \emptyset$.  
The converse is true by Theorem 7 in \cite{Vinberg71}, and the condition also implies $\bigcap_{l \in S} F_{l, i} \ne \emptyset$.
Thus, the combinatorial intersection pattern of $\{H_{l, i} \cap \Omega_i\}_{l=1, \dots, f}$ is the same as 
that of facets $\{F_{l, i}\}_{l=1, \dots, f}$ for $P_i$. (See also Example 7.1.4 of Davis \cite{Davis08}.)
% by Theorem 7 in \cite{Vinberg71}. 
%Hence, $\{H_l\}$ are mutually distinct and do not contain a common subspace of $\dim \geq 0$.  
Similarly, $h(\Gamma_{S})$ is finite  if $\bigcap_{l \in S} H_{l} \cap \Omega \ne \emptyset$.  
The converse is also true: we consider the properly convex cone $\Pi^{-1}(\Omega) \cup \{O\}$. Given a linear finite group action on $\Pi^{-1}(\Omega) \cup \{O\}$
with the subspace $P$ of fixed points, $\dim P \geq 1$, we must have $P \cap \Pi^{-1}(\Omega) \ne \emptyset$.
Hence, $\bigcap_{l \in S} H_{l} \cap \Omega_i \ne \emptyset$. 
%since we can consider $\Pi^{-1}(\Omega)$ and the linear finite group action. 
%The converse is also true: $h_1(\Gamma_{S})$ is finite. Then $\bigcap_{l \in S} H_{l,1} \cap P_1 \ne \emptyset$. 
%Note $\{H_{l} \cap \Omega\}_{l=1, \dots, f} $ has the same pattern. 
%Since $H_{l, i} \ra H_l$ as $i \ra \infty$ in the Hausdorff topology of the compact subsets of $\PS$, 
%and the combinatorial intersection patterns do not change, 
We thus have 
\begin{equation} \label{eqn:ints}
\bigcap_{l \in J} F_{l, i}  \ne \emptyset \hbox{ if and only if } \bigcap_{l \in J} H_{l, i} \cap \Omega_i \ne \emptyset
\hbox{ if and only if }\bigcap_{l \in J} H_{l} \cap \Omega \ne \emptyset
\end{equation} 
for any subset $J$ of $\{1, \dots, f\}$.

We may also assume that the sequence $\{\Clo_i\}$ of the closures  of $\Omega_i$ geometrically converges to 
a compact convex set $K$ by choosing a subsequence. (See Propositions 2.8 and 2.10 of \cite{Choi99}.) 
$K$ is properly convex and has nonempty interior since otherwise $h$ is reducible.  (See Lemma 1 of \cite{cg} and what follows.) 
We may identify $K=\Clo$. Also, we assume that 
$\{P_i\}$ geometrically converges to a compact convex set $P'$ in $\PS$.
Hence, $P'$ is properly convex since $P' \subset \Clo$. 
By taking a subsequence if necessary, we may assume that each sequence $\{S_i\}$ of sides of $P_i$ geometrically 
converges to a compact convex subset $S_\infty$ of $\PS$. If $\dim S_i = 0$ for all $i$, then $S_\infty \in \Omega$ by equation \eqref{eqn:ints}.
Since $S_\infty$ is a properly convex set, we can deduce that $S_\infty$ is the convex hull of its vertices. 
Since $\Omega$ is properly convex also, we obtain $S_\infty \subset \Omega$. 
Any sequence $\{(S_i, T_i)\}$ of disjoint pairs of sides of $P_i$ geometrically converges to a disjoint pair of subsets by equation \eqref{eqn:ints}.
%(See Propositions 2.8 and 2.10 of \cite{Choi99}.) 
%We obtain that $\{P_i\}$ geometrically converges to a compact convex $n$-dimensional polytope $P'$ in $\PS$
%Also, for each $l$,  $\{H_{l, i} \cap P_i\}$ geometrically 
%converges to a $(n-1)$-dimensional polytope $F$ in $\PS$ since  the intersection pattern of $\{H_{l, j} \cap \Omega\}_{l=1, \dots, f}$
%is the same as that of $\{H_l \cap \Omega\}_{l=1, \dots, f}$. (Note that the all of facets of $F$ must meet with $\Omega$.)
%The facets of $P'$ are $(n-1)$-dimensional facet $P' \cap H_l$ by the 
We deduce that $\bigcap_{l \in J} F_{l, 1}  \ne \emptyset$ if and only if $\bigcap_{l \in J} F_{l, \infty}  \ne \emptyset$.
The facets of $P'$ have the same intersection pattern as $\{H_{l, 1} \cap \Omega_1\}_{l=1, \dots, f}$. 
%by the intersection pattern since all facets of nonproperly convex domain meet at a common point. 
%(A nonproperly convex compact domain is a union of an $i$-dimensional hemisphere with common boundary by Proposition 2.4 of \cite{Choi99}.) 
Also, $P' \cap \Omega = P'$ since otherwise $\Omega/h(\pi_1(\hP))$ is not compact. 
Hence $h(\pi_1(\hP))$ is a projective Coxeter group based on $P'$. 

%(We remark that there is alternative proof using Theorem \ref{thm:Vinberg} by the indecomposibility of a limit of a sequence of indecomposible Cartan matrix.
%If the rank droped, then $h$ is reducible.) 

%%% July 21, 10:24pm 2014
%The open and closed argument shows that 
%the elements of the components are represented by Coxeter groups acting as dividing reflection groups. 
By Lemma \ref{lem:properfree}, the conjugation action by $\PGL$ is proper and free.
This proves the first item. 

%Let ${\hP}'$ be a Coxeter $n$--orbifold homeomorphic to $\hP$ with a real projective structure. 
%Let ${\uP}'$ denote the universal cover of ${\hP}'$, and let $\pi_1({\hP}')$ be identified with $\pi_1(\hP)$. 
%The deck transformation group $\pi_1({\hP'})$ is isomorphic to $\pi_1({\hP}')$.
%Given a real projective structure on ${\hP}'$, we can define a developing map
%$D\co {\uP}' \ra \PS$ equivariant with respect to a holonomy homomorphism
%$h\co \pi_1({\hP}') \ra \SL$.
%The singular sets of ${\uP}'$ are in the union of  totally geodesic hyperspace fixed by reflection elements of $\pi_1({\hP}')$.
%Call these the reflection hypersurfaces.
%Then ${\uP}'$ has a compact fundamental region $F'$ bounded by a union of reflection hypersurfaces meeting
%transversally. 
%A developing map $D$ sends ${\uP}'$ diffeomorphically to a properly
%convex domain in $\PS$ and thus associated holonomy homomorphism 
%$h$ is in $D_{\rm rep}(\hP)$ by the above comments. %Proposition \ref{prop:devCox}.
%By Proposition 2.3 \cite{Choi99}, we have that $D(F) = \PS$, $D(F')$ is an affine space, or
%$D(F')$ is a proper convex subset of an affine space.

The holonomy homomorphism $h$ is in $D_{\rm rep}(\hP)$ by Proposition \ref{prop:devCox}.
By Theorem 1 of \cite{Choi04} and the first item, 
the map from a real projective structure to its holonomy homomorphism induces a local homeomorphism
\begin{displaymath}
\hol \co \De({\hP}) \ra D_{\rm rep}(\hP)/\PGL.
\end{displaymath}

Now we show that $\hol$ is injective: 
Suppose that $\Omega_k$  for each $k=1, 2$  is a properly convex open subset of $\PS$
on which $\Gamma :=h(\pi_1(\hP))$ acts for $h \in  D_{\rm rep}(\hP)$ as a dividing projective Coxeter group.
Let $\tilde \Gamma$ denote the torsion-free finite index subgroup by Selberg's lemma. 

If $\Omega_1 \cap \Omega_2 \ne \emptyset$, then  $\Omega' = \Omega_1 \cap \Omega_2$ is a connected properly convex open domain
where $\Gamma$ acts properly discontinuously. 
%Denote the accumulation points of $\Gamma \cdot x$ by $\Lambda$ in the properly convex compact set $K$ the closure of $\Omega_1 \cap \Omega_2$.
%Let $\Omega'$ denote the interior of the convex hull of $\Lambda$ in an affine space containing $K$. 
Each map $\Omega'/\tilde \Gamma \ra \Omega_k/\tilde \Gamma, k=1, 2$  of closed manifolds
is surjective by a homology theory since both are $K(\tilde \Gamma, 1)$-spaces.
%Since $\Omega_k/\Gamma$ is compact, it follows that $\Lambda$ equals the boundary of $\Omega_k$ for $k=1, 2$, and so $\Omega_1=\Omega_2$.
This implies that $\Omega_1 = \Omega_2$ or $\Omega_1 \cap \Omega_2 = \emptyset$ (See the proof of Proposition 2.2 of \cite{Cooper10}).
Since the antipodal map ${\mathcal{A}}\co \PS \ra \PS$ conjugates from $h(\pi_1(\hP))$ to itself,
$\Omega_2 = {\mathcal{A}}(\Omega_1)$ or $\Omega_2 \cap {\mathcal{A}}(\Omega_1) = \emptyset$ by the same reasoning.

Assume that $\Omega_1 \cap \Omega_2 = \emptyset$ and $\Omega_2 \cap {\mathcal{A}}(\Omega_1) = \emptyset$. 
By Benoist \cite[Proposition 1.1]{Benoist00}, $\Gamma$ contains an element $\gamma$
with an attracting fixed point $y$ in the boundary of $\Omega_1$ so that 
the eigenvalue of the vector in the direction of $y$ has a norm strictly greater than those of all other eigenvalues.
$\gamma$ acts on a great $(n-1)$--sphere $S$ whose complement contains $y$. The pair $y$ and its antipode $y_-$ are the unique attracting fixed points of 
the components of $\PS - S$ containing them respectively.
We can choose a point $z$ in $\Omega_2 - S$. As $m \ra \infty$, the sequence $\gamma^m(z)$ converges to $y$
or $y_-$. Thus, $y \in \Clo_2 \cap \Clo_1 \ne \emptyset$ or $y_- \in \Clo_2 \cap \mathcal{A}(\Clo_1) \ne \emptyset$. 
The nonempty set gives a $\Gamma$--invariant convex subset of dimension $< n$;
however, $h$ is irreducible by Proposition \ref{prop:vinberg2}. This is a contradiction.

Therefore, we conclude $\Omega_2 = \Omega_1$ or $\Omega_2 = {\mathcal{A}}(\Omega_1)$.
Hence, $\Omega_2/\Gamma = \Omega_1/\Gamma$ or $\mathcal{A}$ induces
a projective diffeomorphism:  $\Omega_2/\Gamma \ra \Omega_1/\Gamma$.
%As $\mathcal{A}$ is a projective automorphism,
This proves the injectivity of $\hol$.

The surjectivity of $\hol$ is shown as follows: 
By definition, each element of $D_{\rm rep}(\hP)$ acts cocompactly on a properly convex open subset of $\PS$ 
as a projective Coxeter group. We now show that the quotient orbifold is diffeomorphic to $\hP$.
Consider $\hP_h: = \Omega/h(\pi_1(\hP))$ for $h$ in $D_{\rm rep}(\hP)$.
 Since $h(\pi_1(\hP))$ is isomorphic to $\pi_1(\hP)$, 
by Charney and Davis \cite{Charney00}, the Coxeter diagrams are the same for the two groups, 
and a properly convex fundamental domain $F_h$ of $\hP_h$ has the same facial incident relation
as that of a properly convex fundamental domain $F$ of $\hP$.  
(See Davis \cite[Section 13.1]{Davis08} also.)
%(In our case, each fundamental domain is properly convex
%and determines the Coxeter diagram.)
%This implies that $F_h$ is homeomorphic to $F$ by an incidence relation preserving map $\phi$.
%(See Exercise 3.2.9 of Thurston \cite{Thurston97}.)
By Davis \cite[Corollary 1.3]{Davis2013},  
%each point of $D_{\rm rep}(\hP)$ is in the image of $\hol$ for some real projective structure on ${\hP}'$
$\hP_h$ is diffeomorphic to $\hP$. 
(See also Wiemeler  \cite[Corollary 5.3]{Wiemeler}.)
Therefore, $\hol$ is a homeomorphism as $\hol$ is a local homeomorphism.

%Using the homeomorphism $\phi$, we can show that $\hP_h$ is diffeomorphic to $\hP$
%by an induction on dimensions of strata,
%defining smooth diffeomorphisms on neighborhoods of strata and extending to those of the next
%dimensional strata. As each stratum is a cell and the stratification is locally finite,
%it is easy to see that the diffeomorphisms always extend to the next dimensional stratum.
%(See Proposition 3.10.7 \cite{Thurston97}.)
%(See also Example 7.1.4 of \cite{Davis08}.)
%\marginpar{sc: changed here}
%We can show that $\hP_h$ is diffeomorphic to $\hP$ by induction on codimensions of strata
%and removing their tubular neighborhoods.

The third item was proved while proving the second one.

%The last item follows since $\hol$ maps each component of $\De({\hP}')$ into an open subset of $D_{\rm rep}(\hP)$.
\end{proof}

%We could simplify the above theorem if we could prove the following:
%\begin{question}\label{q:Davis}
%Given a compact smooth Coxeter orbifold $\hP$ covered by a manifold, a smooth Coxeter $n$--orbifold homeomorphic to $\hP$
%is diffeomorphic to $\hP$.
%\end{question}
%This follows if $\hP$ is a hyperbolic Coxeter orbifold by the Mostow rigidity.
%\marginpar{refs: cooper, ..., Porti}
%We believe that this is true by the recent paper of Davis \cite{Davis2013}. Thus, we can assume that ${\hP}'$ equal $\hP$. 
%For a hyperbolic Coxeter $3$--orbifold $\hP$, any homeomorphic smooth orbifold
%$\hP'$ is diffeomorphic to $\hP$ always by the Mostow rigidity and the Orbifold Theorem
%(see Theorem 9.1 \cite{Boileau03} and \cite{Cooper00}).
%However, this should be true for each compact
%Coxeter $3$--orbifolds.

%Note that the $\SL$--action by conjugations on $(\SL)^f$ is not faithful.
%Thus, we act by $\PGL$. 
The following lemma is a generalization of Lemma 1 of \cite{Choi06}.

Let $R(n+1)$ denote the subspace $\SL$ of all reflections. 
%For $g \in \PGL$, let $\tilde g \in \SL$ denote any element of $\SL$ corresponding to $g$.
For each element $g\in\SL$, let $[g] \in \PGL$ denote the corresponding element. 
\begin{lemma}\label{lem:properfree}
Let $\mathcal U \subset R(1+n)^f$ denote the subspace of all $(g_1, \ldots, g_f)$ 
generating an irreducible dividing projective Coxeter group $\Gamma$. % in a union of components of $\LO$ or $\SL$.
Then the $\PGL$--action on $\mathcal U$ by conjugation
\begin{equation*}
[g] \circ (g_1, \ldots, g_f) = (g g_1 g^{-1}, \ldots, g g_f g^{-1}), g \in \SL
\end{equation*}
is proper and free.
\end{lemma}
\begin{proof}
The proof for the properness directly generalizes that of Lemma 1 of \cite{Choi06}
as the group $\Gamma$ is irreducible. 

Suppose that an element $\tilde g$ of $\SL$ satisfies $\tilde g g_i = g_i \tilde g$ for $i=1, \dots, f$. 
%This implies that $\tilde g$ acts on each hyperspace of fixed points of $g_i$ and the union of antipodal fixed points of $g_i$. 
We have a compact properly convex polytope $P$ as a properly convex fundamental domain of $\Gamma$ since $\Gamma$ is a dividing projective Coxeter group.
Choosing generators differently if necessary, we may assume without loss of generality that 
each side $S_i$ of $P$ is fixed by $g_i$ for $i=1, \dots, f$.
Since $\tilde g$ commutes with $g_i$, $\tilde g$ acts on the subspace $S'_i \subset \PS$ containing $S_i$ 
and each pair $\{r_i, \mathcal{A}(r_i)\}$ of antipodal fixed points of $g_i$.
%Thus, $\tilde g$ acts on each hyperspace $W_i$ of fixed points of $g_i$  
%Since each side of $P$ is in the direction of $W_i$,
Therefore, $\tilde g$ acts on $\{v_1, \dots, v_m, \mathcal{A}(v_1), \dots, \mathcal{A}(v_m)\}$ for vertices $v_1, \dots, v_m$ of $P$. 
As $P$ has $n+1$ vertices in a general position, 
$\tilde g$ is diagonalizable over $\R$. 
%$V = \bigoplus_{j=1}^{m'} V_j$ for characteristic subspaces $V_j, j=1, \dots, m'$, 
%with mutually distinct nonzero real characteristic values $\lambda_i$ of $\tilde g$.
%Suppose that the decomposition is not trivial with some eigenvalues of norms $> 1$. Then we have an eigenvalues have a norm $> 1$ and one of a norm $< 1$. 
%Since $\tilde g$ acts on $\{r_i, \mathcal{A}(r_i)\}$, 
%the reflection vectors of each $g_i$ must be in some $V_j$. 
%This shows that $\Gamma$ is reducible, which is a contradiction.  
Since $\tilde g: V \ra V$ is a $\Gamma$-module morphism, 
we obtain $\tilde g = \lambda I_V$ for $\lambda = \pm 1$ by Schur's Lemma for $\R$. 
%Therefore, $\tilde g$ has a unique real eigenvalue in $\{\pm 1\}$ and $\tilde g = \pm I_V$. 
%\marginpar{referee: Use Schur's lemma}
\end{proof}

One related question is: 
\begin{question}
Can $\De(\hP)$ for a compact hyperbolic Coxeter orbifold $\hP$ be compact and have dimension $\geq 1$?
\end{question}
This question was first asked by Benoist in 2005 as far as the authors know (see \cite{Marquis10} for examples).
%The answer is known to be positive if $|\hP|$ is a truncation $3$--polytope (see \cite{Marquis10}).
%\marginpar{Read and Add refs Annals 1983 paper of Davis}
%\marginpar{refs: inkang}

\subsection{The reinterpretations of the deformation spaces as solution spaces}
\label{subs:VinbergEq}

Let $V$ be an $(n+1)$--dimensional real vector space. Denote by $\M_{s\times t}(\R)$ the set of $s \times t$ matrices with real entries.
We will identify $V$ and $V^*$ with $\M_{(n+1)\times 1}(\R)=\R^{n+1}$ and $\M_{1\times (n+1)}(\R)=(\R^{n+1})^*$ respectively as follows:
we choose the standard basis $\{e_1, \dotsc, e_{n+1}\}$ of $V$. Let $\{e_1^*, \dotsc, e_{n+1}^*\}$ be its dual basis of $V^*$.
If $\alpha_i = \alpha_{i,1}e_1^* + \dotsm + \alpha_{i,n+1}e_{n+1}^* \in V^*$, then  $\alpha_i$ is identified with
the $1 \times (n+1)$ matrix $(\alpha_{i,1}, \dotsc, \alpha_{i,n+1})$.
Similarly, if $b_j = b_{j,1}e_1 + \dotsm + b_{j,n+1}e_{n+1} \in V$, then
$b_j$ is identified with the $(n+1) \times 1$ matrix $(b_{j,1},\dotsc,b_{j,n+1})^t$,
where a matrix $A^t$ means the transpose of a matrix $A$.
Hence $\alpha_i(b_j)=\alpha_i b_j$ where the right-hand side is the scalar obtained as
the matrix product of a $1 \times (n+1)$ matrix with an $(n+1) \times 1$ matrix.
Denote by $I_{n+1}$ the $(n+1)\times (n+1)$-identity matrix. With this matrix notation, 
a reflection $R$ is of form $I_{n+1}-b\alpha$ for $\alpha \in V^*$ and $b \in V$ with $\alpha b =2$.
%Moreover, if $g$ and $h$ are linear transformations of $V$, then their composition $h \circ g$ is equal to the product $hg$ of two matrices.

Let $\hP$ be a compact Coxeter $n$--orbifold with the fundamental chamber
a properly convex $n$--polytope $P$ with $f$ facets in $\PS$, and let $\I_{\hP} = \{ 1,\ldots, f \}$ be the index set of the facets. 
The orbifold structure of $\hP$ gives us 
the order $n_{ij}$ of the ridge $F_i \cap F_j$, $i, j \in \I_{\hP}, i \ne j$. 
%given by the orbifold structure on $P$. %or that we wish to put on the ridge 
%whenever $F_i \cap F_j \ne \emptyset$.
Let $P$ be given by a system of linear inequalities $\alpha_i \geq 0$ ($i \in \I_{\hP}$) 
for $\alpha_i$  in $V^*$. 
Let $b_i$ be a vector with $\alpha_i b_i=2$ for each $i$, and 
let $R_i$ be the reflection $I_{n+1} - b_i\alpha_i$ for each $i \in \I$, 
and let $\Gamma \subset \SL$ be the group generated by the reflection $R_i$. 
\begin{align*}
\tag*{\text{Define}} E_{1, \hP} & = \{ (i,j)\in \I_{\hP}\times \I_{\hP} \,|\, i=j \}, \\
E_{2, \hP} & = \{ (i,j)\in \I_{\hP}\times \I_{\hP} \,|\, i<j, \text{ $F_i$ and $F_j$ are adjacent in $P$ and $n_{ij}=2$} \}, \\
E_{3, \hP} & = \{ (i,j)\in \I_{\hP}\times \I_{\hP} \,|\, i<j, \text{ $F_i$ and $F_j$ are adjacent in $P$ and $n_{ij} \geq 3$} \} \text{ and }\\
E_{4, \hP} & = \{ (i,j)\in \I_{\hP}\times \I_{\hP} \,|\, i<j, \text{ $F_i$ and $F_j$ are not adjacent in $P$} \}.
\end{align*}
%Fix orders $n_{ij}$ for the ridges of $P$. We consider the deformation space of real projective structures on the corresponding Coxeter orbifold $\hP$ 
%that can be constructed for this data.

%As in Section \ref{subs:RepresentationSpaces}  we assume that $\hP$ admits a real projective structure, but does not admit a spherical or Euclidean structure.
Vinberg's result leads us to solve the following system of polynomial equations:
\begin{itemize}
\item $a_{ii} = \alpha_i b_i = 2$ for $(i,i) \in E_{1, \hP}$.
\item $a_{ij}=\alpha_i b_j= 0$ and $a_{ji}=\alpha_j b_i= 0$ for $(i,j) \in E_{2, \hP}$.
\item $a_{ij}a_{ji}= \alpha_i b_j \alpha_j b_i=4 \cos^2 \left( \frac{\pi}{n_{ij}} \right)$ for $(i,j) \in E_{3, \hP}$.
\end{itemize}
We call these polynomial equations \emph{Vinberg's equations}. The $\alpha_i$'s and $b_i$'s are \emph{variables}.
Denote by $e$ the number of ridges and $e_2$ the number of ridges of order 2. %Observe that $N_{\hP} = f + e + e_2$.
$N_{\hP} = f + e + e_2$ is the number of Vinberg's equations. 
Let  $\{\Phi_k\}_{k=1}^{N_{\hP}}$ be the set of polynomials
%$a_{ii}-2$, $a_{ij}$, $a_{ji}$ or $a_{ij}a_{ji}-4 \cos^2\bigl( \frac{\pi}{n_{ij}} \bigr)$ 
in Vinberg's equations, and
let $\Phi_{\hP} \co (V^*)^f \times V^f \rightarrow \mathbb{R}^{N_{\hP}}$
be the map 
\begin{equation*}
(\alpha_1, \dotsc, \alpha_f, b_1, \dotsc, b_f) \mapsto (\Phi_1, \ldots, \Phi_{N_{\hP}}).
\end{equation*}
Let $\mathbb{R}_+$ be the set of positive real numbers. 
Denote by 
\[\theta: \widetilde{G}:=\mathbb{R}_+^f \times \SL \times (V^*)^f \times V^f \ra  (V^*)^f \times V^f \]
the action given by
\begin{multline}\label{eqn:theta}
(d_1,\dotsc, d_f, g) \cdot (\alpha_1, \dotsc, \alpha_f, b_1, \dotsc, b_f)  \\
= (d_1 \alpha_1 g^{-1}, \dotsc, d_f \alpha_f g^{-1}, d_1^{-1} g b_1, \dotsc, d_f^{-1} g b_f),
\end{multline}
where $d_i \in \mathbb{R}_+$ for each $i \in \I_{\hP}$ and $g \in \SL$.
Then we have the invariance
\begin{equation} \label{eqn:phiinv}
\Phi_{\hP} \circ \theta(d_1, \dots, d_f, g) = \Phi_{\hP}.
\end{equation}
Define an open set
%\begin{equation}
\begin{align} \label{eqn:UP}
 \mathcal{U}_{\hP} = & \{ (\alpha_1, \dotsc, \alpha_f, b_1, \dotsc, b_f) \in (V^*)^f \times V^f \,|\,   
 \exists \vec{v} \ne O, \alpha_i(\vec{v}) > 0 \hbox{ for each } i, \nonumber \\
&  \langle \alpha_1, \dots, \alpha_f \rangle = V^*, 
 a_{ij} < 0 \text{ and } a_{ji} < 0 \text{ if } (i,j) \in E_{3, \hP} \cup E_{4, \hP}, \nonumber\\ 
& \text{ and } a_{ij}a_{ji} > 4 \text{ if } (i,j) \in E_{4, \hP}
\}
\end{align}
where we replaced the condition (L2) with (L2)$'$.

We define the solution set
\begin{displaymath}
\widetilde{\De}(\hP):= \Phi_{\hP}^{-1}(0) \cap \mathcal{U}_{\hP}\, .
\end{displaymath}
%which consists of elements of $(V^*)^f \times V^f$ such that
%the projective Coxeter group $\Gamma$ is generated by $R_i = I_{n+1} - b_i \alpha_i$, $i=1, \dots, f$. 
%satisfies the relations. 
% gives the quotient
%orbifold $\Omega_\Gamma/\Gamma$ with $\Gamma$ isomorphic to $\pi_1(\hat{P})$ .
%This is an algebraic set as opposed to the semialgebraic set because of the form of equation \ref{eqn:UP}.
%(See Proposition \ref{prop:parabolic2}.)

By invariance, $\widetilde{G}$ acts on $\mathcal{U}_{\hP}$ and on $\widetilde{\De}(\hP)$.
Applying the action $\theta(d_i, g)$ on $\widetilde{\De}(\hP)$, we have
\begin{displaymath}
I_{n+1} - (d_i^{-1} gb_i) \,   (d_i\alpha_i g^{-1}) = g (I_{n+1} - b_i \alpha_i)g^{-1} = g R_i g^{-1} \text{ for  } i\in \I_{\hP}.
\end{displaymath}
Hence the action $\theta(d_1, \dots, d_f, g)$ on $\widetilde{\De}(\hP)$ corresponds to the conjugation in $\SL$.

Define $\mathcal M$ as the submanifold of $(V^*)^f \times V^f $ of
elements $(\alpha_1, \dots, \alpha_f, b_1, \dots, b_f)$ where $\alpha_i b_i = 2$ for every $i=1,\dots, f$.
Define the map
\[ {\mathcal{I}'}_{\SL}\co {\mathcal M} \rightarrow R(n+1)^f \]
by sending  $(\alpha_1, \dots, \alpha_f, b_1, ..., b_f)$ to $(r_1, \dots, r_f)$ given by 
\[r_i(\cdot) =I_V- \alpha_i(\cdot) b_i\co  V \ra V \hbox{ for each } i=1, \dots, f.\]
%where $\pi_1(\hP)=\langle r_i \,|\, (r_i r_j)^{n_{ij}} \rangle$.
The map sends the information on the reflection subspace and the vertex to the reflection itself.
Since a reflection is determined by the fixed-point subspace and the antipodal fixed point, 
the group $\mathbb{R}_+^f \times \{\pm I_V\}^f$ acts simply transitively on the fibers of ${\mathcal{I}'}_{\SL}$.
Therefore, %${\mathcal{I}'}_{\SL}$ is a principle bundle map over its image in $\SL^f$.
we obtain a principal fibration
\begin{align} \label{eqn:fib}
\mathbb{R}_+^f \times \{\pm I_V\}^f \longrightarrow  & \quad {\mathcal M}  \nonumber \\
&\quad  \mbox{ \Large $\downarrow$ } \mbox{ \small ${\mathcal{I}'}_{\SL}$ }    \nonumber \\
&\quad  R(n+1)^f. 
\end{align}

\begin{theorem}\label{thm:identification}
Let $\hP$ be a compact Coxeter $n$--orbifold.
Assume that $\hP$ admits a real projective structure, but does not admit a spherical or Euclidean structure.
We consider the solution set
\begin{displaymath}
\widetilde{\De}(\hP):= \Phi_{\hP}^{-1}(0) \cap \mathcal{U}_{\hP} \subset {\mathcal M}
\end{displaymath}
for Vinberg's equations $\Phi_{\hP}$.
\begin{itemize}
\item There exists a $\PGL$--equivariant surjective map
\[{\mathcal{I}}\co  \widetilde{\De}(\hP)/(\mathbb{R}_+^f \times \{\pm I_V\}) \ra D_{\rm rep}(\hP). \]
\item $D_{\rm rep}(\hP)$ is homeomorphic to $\widetilde{\De}(\hP)/(\mathbb{R}_+^f \times \{\pm I_V\})$.
\item The deformation space $\De(\hP)$ of real projective structures on the Coxeter orbifold $\hP$ is homeomorphic to
a union of components of
\begin{displaymath}
\widetilde{\De}(\hP)/\widetilde{G} = D_{\rm rep}(\hP)/\PGL \hbox{ where } \widetilde{G}=\mathbb{R}_+^f \times \SL.
\end{displaymath}
\end{itemize}
\end{theorem}
\begin{proof}
%The first item follows from equation \ref{eqn:fib} since 
The conditions of Equation \eqref{eqn:UP} imply that we have a nontrivial properly convex polytope as a fundamental chamber.
Vinberg's equation, Theorem \ref{thm:vinbergM},
and Proposition \ref{prop:vinberg2} imply that the image points are discrete faithful dividing reflection representations
$\pi_1(\hP) \ra \SL$. 

Conversely, the collection of reflections generating the discrete faithful dividing reflection representation
gives some point in $\widetilde{\De}(\hP)$, ie in $\Phi_{\hP}^{-1}(0) \cap \mathcal{U}_{\hP}$, 
 since it satisfies (L1) and (L2)$'$ as we showed in Section \ref{subs:Vinberg}.
Hence the map is surjective.

A representation given by assigning the fixed points and reflection facets to fundamental generators
has ambiguity understood by Equation \eqref{eqn:theta}.
Thus, the fibers are again given as orbits of $\mathbb{R}_+^f \times \{\pm I_V\}$, and
${\mathcal{I}'}_{\SL}$ restricts to a fibration $\widetilde{\De}(\hP) \ra D_{\rm rep}(\hP)$.
The second item follows.
The third item follows by Theorem \ref{thm:defrepspaces} and the second item.
\end{proof}

Let $\PV(\hP)$ denote the space of $f \times f$ matrix $A=(a_{ij})$ satisfying (L1) and (L2)$'$ with $\rank A = n+1$ and no component of zero type.
We recall from Equation \eqref{eqn:diagonal} that a diagonal matrix group $\mathbb{R}_+^f$ acts on $\PV(\hP)$ by
\begin{equation}\label{eqn:diaga}
(d_1, \ldots, d_f) \circ (a_{ij}) = (d_i d_j^{-1} a_{ij}).
\end{equation}

\begin{corollary} \label{cor:identification}
Let $\hP$ be a compact Coxeter $n$--orbifold.
Assume that $\hP$ admits a real projective structure, but does not admit a spherical or Euclidean structure.
Then there exists a homeomorphism between each pair of the spaces below % except for the first space,
\[\De(\hP) \,\leftrightarrow\, D_{\rm rep}(\hP)/\PGL \,\leftrightarrow\, \widetilde{\De}(\hP)/\widetilde{G}
\,\leftrightarrow\, \PV(\hP)/\mathbb{R}_+^f.\]
%where the first map is a homeomorphism with the union of components of the second space.
\end{corollary}
\begin{proof}
Theorems \ref{thm:defrepspaces} and \ref{thm:identification} give the first and second correspondences.
%The map from the fourth one to the second one is obtained by Theorem \ref{thm:Vinberg} and Proposition \ref{prop:vinberg2}.
The map from the second one to the fourth one is obtained by going to the third one and taking
$\alpha_i(b_j)$ as the entries of the Cartan matrices. 
Theorem \ref{thm:Vinberg} and Proposition \ref{prop:vinberg2} give us the map from the fourth one to the second one. 
These maps are inverses of each other by the uniqueness part of Theorem \ref{thm:Vinberg}.
\end{proof}
%The first one would be a homeomorphism if Question \ref{q:Davis} is true. 

\section{Real projective structures near the hyperbolic structure}
\label{s:LocalSpaces}

We will obtain the information of real projective structures near the hyperbolic structure in terms of Zariski tangent spaces.

Recall in the previous section that real projective structures in the deformation space of a compact Coxeter orbifold 
$\hP$  correspond to solutions to Vinberg's equations.
In Section \ref{subs:TangentSpace} we study the Zariski tangent space to this solution space.
In Section \ref{subs:HyperbolicEquations} we describe the space of hyperbolic structures of $\hP$ in terms of polynomial equations,
forming so-called hyperbolic equations. 
In Section \ref{subs:TangentToHyperbolic} we study the Zariski tangent space to the solution space of the hyperbolic equations.
We compute the rank of the differential of the polynomial map from the hyperbolic equation in Proposition \ref{prop:ident}.  %  in Lemma \ref{lem:ranksum}.
In Section \ref{subs:MainLemma} we compare these two Zariski tangent spaces 
and combine this observation with the weak orderability of $\hP$ to prove Lemma \ref{lem:ranksum}, 
computing the rank of the differential of the polynomial map
from Vinberg's equation.  
Finally, in Section \ref{subs:Proof}, we prove the main result Theorem \ref{thm:main1}. % using Proposition \ref{prop:ident}.

\subsection{The Zariski tangent space to Vinberg's equations}
\label{subs:TangentSpace}

Let $\hP$ be a Coxeter orbifold based on 
a properly convex $n$--polytope $P$ with $f$ facets in $\PS$, and let $\I_{\hP} = \{ 1,\ldots, f \}$ be the index set of the facets.
Assume that $P$ is given by a system of linear inequalities, $\alpha_i \geq 0$ ($i \in \I_{\hP}$), 
for $\alpha_i \in V^*$.
Suppose that each $b_i$, $i=1, \dots, f$, is a reflection vector with $\alpha_ib_i=2$.

As in Section \ref{subs:VinbergEq}, we have variables $\alpha_i \in V^*=(\mathbb{R}^{n+1})^*$
and $b_i \in V=\mathbb{R}^{n+1}$ for $i \in \I_{\hP}= \{1, \ldots, f\}$, and Vinberg's equations are of the following form:
\begin{itemize}
\item $\Phi_{ii} = \alpha_i b_i - 2 = 0$ for $(i,i) \in E_{1, \hP}$.
\item $\Phi^{[1]}_{ij} = \alpha_i b_j =0$ and $\Phi^{[2]}_{ij} = \alpha_j b_i=0$ for $(i,j) \in E_{2, \hP}$.
\item $\Phi_{ij} = \alpha_i b_j \alpha_j b_i - 4 \cos^2 \left(\frac{\pi}{n_{ij}}\right)$ for  $(i,j) \in E_{3, \hP}$.
\end{itemize}
Recall that $N_{\hP}$ is the number of Vinberg's equations, ie $N_{\hP}=f+e+e_2$.
Let 
\[\pi_i^{[1]} \co (V^*)^f \times V^f \to V^* \hbox{ and } \pi_i^{[2]} \co (V^*)^f \times V^f \to V\] 
denote the projections onto the $i$th factor $V^*$ and the $(f+i)$th factor $V$, for every $i \in \I_{\hP}$, respectively.
For each $(i,j) \in E_{3, \hP}$, the derivative of $\Phi_{ij}$ at $p=(\alpha_1, \dotsc, \alpha_f, b_1, \dotsc, b_f)$, considered as a linear map, is:
\begin{equation*}
\begin{split}
D\Phi_{ij}(\dot p) & = a_{ji}{\dot \alpha_i}{b_j} + a_{ij}{\dot \alpha_j}{b_i} + a_{ij}{\alpha_j}{\dot b_i} + a_{ji}{\alpha_i}{\dot b_j} \\
& = a_{ji}{\pi_i^{[1]} (\dot p)}{b_j} + a_{ij}{\pi_j^{[1]} (\dot p)}{b_i} + a_{ij}{\alpha_j}{\pi_i^{[2]} (\dot p)} + a_{ji}{\alpha_i}{\pi_j^{[2]} (\dot p)}
\end{split}
\end{equation*}
for $\dot p = (\dot \alpha_1, \dotsc, \dot \alpha_f, \dot b_1, \dotsc \dot b_f) \in (V^*)^f \times V^f$, 
and entries $a_{ij}$ of the Cartan matrix of $\hP$.
Similarly, for each $(i,i) \in E_{1, \hP}$\,,
\begin{displaymath}
D\Phi_{ii}(\dot p) = \pi_i^{[1]}(\dot p) b_i + \alpha_i \pi_i^{[2]}(\dot p),
\end{displaymath}
and for each $(i,j) \in E_{2, \hP}$\,,
\begin{displaymath}
D\Phi_{ij}^{[1]}(\dot p) = \pi_i^{[1]}(\dot p) b_j + \alpha_i \pi_j^{[2]}(\dot p) \quad \text{and} \quad
D\Phi_{ij}^{[2]}(\dot p) = \pi_j^{[1]}(\dot p) b_i + \alpha_j \pi_i^{[2]}(\dot p).
\end{displaymath}
More explicitly, combining Vinberg's equations gives a function $\Phi_{\hP} \co (V^*)^f \times V^f \to \R^{N_{\hP}}$ and
the rows of the ${N_{\hP}} \times 2(n+1)f$ Jacobian matrix $[D \Phi_{\hP}]$ are
made up of $(n+1)$--entry blocks. % each consisting of $(n+1)$ entries:

For each $(i,i) \in E_{1, \hP}$,
\begin{align} \label{eqn:Phidiff}
[D\Phi_{ii}]   =  &(0, \dotsc, 0, b_{i,1}, \dotsc, b_{i,n+1}, 0, \dotsc, 0, \alpha_{i,1}, \dotsc, \alpha_{i,n+1},0,\dotsc,0) \nonumber \\
= & (0, \dotsc, 0, \quad\:
\underbrace{b_i^t}_{i \text{th block}} \quad\:  ,0, \dotsc, 0, \quad\:  \underbrace{\alpha_i}_{(f+i) \text{th block}} \quad\:  ,0, \dotsc, 0). \nonumber\\
& \hbox{ For } (i,j) \in E_{2, \hP},     \nonumber \\
[D\Phi_{ij}^{[1]}]  = & (0, \dotsc, 0, \underbrace{b_j^t}_{i \text{th}},0, \dotsc, 0, \underbrace{0}_{j \text{th}} ,0, \dotsc, 0, \underbrace{0}_{(f+i) \text{th}},0, \dotsc, 0, \underbrace{\alpha_i}_{(f+j) \text{th}} ,0, \dotsc, 0), \nonumber  \\
[D\Phi_{ij}^{[2]}]  =  & (0, \dotsc, 0, \underbrace{0}_{i \text{th}},0, \dotsc, 0, \underbrace{b_i^t}_{j \text{th}} ,0, \dotsc, 0, \underbrace{\alpha_j}_{(f+i) \text{th}},0, \dotsc, 0, \underbrace{0}_{(f+j) \text{th}} ,0, \dotsc, 0). \nonumber \\
& \hbox{ For  } (i,j) \in E_{3, \hP},    \nonumber \\
[D\Phi_{ij}]  = & (0, \dotsc, 0, \underbrace{a_{ji}b_j^t}_{i \text{th}},0, \dotsc, 0, \underbrace{a_{ij}b_i^t}_{j \text{th}} ,0, \dotsc, 0, \underbrace{a_{ij}\alpha_j}_{(f+i) \text{th}},0, \dotsc, 0, \underbrace{a_{ji}\alpha_i}_{(f+j) \text{th}} ,0, \dotsc, 0).
\end{align}
Suppose that $p$ is a point of $\Phi_{\hP}^{-1}(0)$. Then the \emph{Zariski tangent space at $p$} is the kernel of
the Jacobian matrix $[D \Phi_{\hP}]$ evaluated at $p$.

\subsection{The hyperbolic equations}
\label{subs:HyperbolicEquations}

Let $V$ be an $(n+1)$--dimensional real vector space with coordinate functions $x_1, \ldots, x_{n+1}$, and let $\hP$ be
a compact hyperbolic Coxeter orbifold with the fundamental chamber equal to  
a compact $n$--polytope $P$ in the Klein projective model of the $n$--dimensional hyperbolic space $\mathbb{H}^n$.
Let $P$ have facets $F_i$ for $i \in \I_{\hP}=\{1,2, \dotsc, f\}$.

Denote by $\nu_i \in V$ the inward unit normal to the subspace spanned by vectors in directions of $F_i$
with respect to the Lorentzian inner product on $V$.
%, defined by
%\begin{displaymath}
% \langle x, y\rangle = -x_1 y_1 + x_2 y_2 +  \ldots + x_{n+1} y_{n+1}.
%\end{displaymath}
Then the system of linear inequalities define $P$
\begin{displaymath}
 \langle \nu_i, x \rangle \geq 0 \text{ for each } i \in \I_{\hP} \quad \text{and} \quad x_1=1.
\end{displaymath}
To construct a hyperbolic Coxeter $n$--polytope $P$
with prescribed dihedral angles $\frac{\pi}{n_{ij}}$,  we need to solve the following equations:
\begin{equation}\label{eqn:hyperbolic}
\begin{split}
\langle \nu_i,\nu_i \rangle & = 1 \;\;\text{for each}\;\; i \in \I_{\hP}, \\
\langle \nu_i,\nu_j \rangle & = -\cos\left(\tfrac{\pi}{n_{ij}}\right) \;\;\text{if facets $F_i$ and $F_j$ are adjacent in $P$}.
\end{split}
\end{equation}
We call these equations \emph{hyperbolic equations}.
To compare the hyperbolic equations with Vinberg's equations, the system of linear inequalities defining $P$ is given by %\marginpar{Refree}
\begin{displaymath}
 \alpha_i(x) \geq 0 \text{ for  } i \in \I_{\hP} \quad \text{and} \quad x_1=1, x \in V
\end{displaymath}
where the linear functional $\alpha_i \in V^*$ %is dual to $2\nu_i$ under the Lorentzian inner product.
is given by $\alpha_i(v)=2\langle \nu_i, v \rangle$.
The hyperbolic reflection in the facet $F_i$ is a map
\begin{displaymath}
R_i(v) = v - 2\langle \nu_i, v \rangle \nu_i = v - \alpha_{i}(v)b_i
\end{displaymath}
for $b_i=\nu_i$. Thus taking
$\alpha_i = 2 \langle \nu_i, \,\cdot\, \rangle$ and $b_i = \nu_i$ gives
a \emph{hyperbolic point} $t$ in $\Phi_{\hP}^{-1}(0)$ corresponding to the hyperbolic structure on $\hP$:
We rewrite the equation in another way. 
If facets $F_i$ and $F_j$ are adjacent in $P$, then
\begin{gather*}
a_{ij}=\alpha_i(b_j) = 2\langle \nu_i,\nu_j \rangle = -2\cos\left(\tfrac{\pi}{n_{ij}}\right) \\
\tag*{{and thus}} a_{ii} = 2\langle \nu_i,\nu_i \rangle = 2,  (i,i) \in E_{1, \hP},\\
a_{ij} = 0 \,\text{ and }\, a_{ji}=0,   (i,j) \in E_{2, \hP}, \\
a_{ij}a_{ji} = 4\cos^2( \tfrac{\pi}{n_{ij}} ),   (i,j) \in E_{3, \hP}.
\end{gather*}

\subsection{The Zariski tangent space to the hyperbolic equations}
\label{subs:TangentToHyperbolic}

As in Section \ref{subs:HyperbolicEquations}, we assume that $P$ is a compact hyperbolic Coxeter 
$n$--polytope where the dihedral angle at each ridge $F_{ij}=F_i \cap F_j$
equals $\frac{\pi}{{n_{ij}}}$ for an integer $n_{ij} \geq 2$. Constructing such a hyperbolic $n$--polytope $P$
is the same as solving the system
of hyperbolic equations \eqref{eqn:hyperbolic} for the unit normals $\nu_i$.
Equivalently we can write these equations in terms of the reflection vectors $b_i=\nu_i$.
This gives the following system of $m=f+e$  equations:
\begin{equation}\label{eqn:hyperbolic2}
\begin{split}
\Psi_{ii} & = 2\langle b_i, b_i \rangle - 2 = 0 \text{ for  } (i,i) \in E_{1, \hP} \\
\Psi_{ij} & = 2\langle b_i, b_j \rangle + 2\cos\left(\tfrac{\pi}{n_{ij}}\right) = 0 \text{ for } (i,j) \in E_{2, \hP} \cup E_{3, \hP}.
\end{split}
\end{equation}
Combining these gives a function
$\Psi_{\hP}\co  V^f = \R^{(n+1)f} \to \R^m$, and $\Psi_{\hP}^{-1}(0)$
contains Coxeter $n$--polytopes in $\mathbb{H}^n$ with the desired dihedral angles.

We define an open manifold
\begin{align} \label{eqn:UPH}
 \mathcal{W}_{\hP} & := \{ (b_1, \dotsc, b_f) \in V^f |\langle b_i, b_i \rangle = 1,  i\in \I_{\hP},
  \langle b_i, b_j \rangle < -2,  \text{ if } (i,j) \in E_{4, \hP}\}.
%\nonumber \\
% & \langle b_i, b_j \rangle < 0  \text{ if } (i,j) \in E_{3, \hP} \cup E_{4, \hP}, \text{ and } \langle b_i, b_j \rangle^2 > 4 \text{ if } (i,j) \in E_{4, \hP}
\end{align}
The $f$-tuple $(b_1, \dots, b_f)$ of normal vectors to facets for a compact hyperbolic polytope satisfies equations \eqref{eqn:hyperbolic2} and is in 
$\mathcal{W}_{\hP}$ (see \cite{Roeder07}).

Now we compute the
derivative $D\Psi_{\hP}$ at a hyperbolic point $t$. Setting $\alpha_i= 2\langle \nu_i, \,\cdot\, \rangle$, $i=1, \dots, f$, 
to be the linear functionals defining the facets of $P$, we obtain 
\begin{equation*}
D \Psi_{ij} (\dot b) = 2\langle \dot b_i, b_j \rangle + 2\langle  b_i,  \dot b_j \rangle
 = \alpha_j \dot b_i + \alpha_i \dot b_j \,.
\end{equation*}
 for  $\dot b:=(\dot b_1, \dots, \dot b_f) \in V^f,  \dot b_i \in V, i =1, \dots, f.$
For $i=j$, this becomes
\[D\Psi_{ii} (\dot b) = 2 \alpha_i {\dot b_i}, \hbox{ for } \dot b =(\dot b_1, \dots, \dot b_f) \in V^f.\]
Equivalently, the rows of the $m \times (n+1)f$ Jacobian matrix $[D \Psi_{\hP}]$ consist
of blocks, each consisting of $(n+1)$ entries:

For each $(i,i) \in E_{1, \hP}$,
    \begin{equation*}
    \begin{split}
    [D\Psi_{ii}]  & = (0, \dotsc, 0, 2\alpha_{i,1}, \dotsc, 2\alpha_{i,n+1}, 0, \dotsc, 0) \\
    & = (0, \ldots, 0, \quad \quad \: \underbrace{2\alpha_i}_{i \text{th block}} \quad \quad \: ,0, \ldots, 0)
    \end{split}
    \end{equation*}
and for each $(i,j) \in E_{2, \hP} \cup E_{3, \hP}$,
    \begin{equation*}
   [D\Psi_{ij}]  =  (0, \ldots, 0, \underbrace{\alpha_j}_{i \text{th block}} ,0, \ldots, 0, \underbrace{\alpha_i}_{j \text{th block}}, 0, \ldots, 0).
    \end{equation*}
Then the Zariski tangent space to $\Psi_{\hP}^{-1}(0) \cap \mathcal{W}_{\hP}$ at $t$ is $\ker D\Psi_{\hP}$.

Recall that $\Hom(\pi_1(\hP), \PO)$ is an algebraic subset of the space $\PO^f$ for the number 
of fundamental generators $f$, ie the number of facets of $P$.
We give the standard point-set topology as a subspace.

\begin{proposition} \label{prop:ident}
Let $P$ be a compact hyperbolic Coxeter $n$--polytope. Suppose that $\hP$ is the Coxeter orbifold arising from $P$, 
with the associated holonomy representation $h_0$, and let $\bar b_0$ denote the $f$--tuple of vectors normal to the facets of $P$ in
the Lorentzian spaces.
Then
\begin{itemize}
\item The orbit of $h_0$ under $\PO$ contains an open neighborhood of $h_0$
in $\Hom(\pi_1(\hP), \PO)$, and this is a smooth $\tfrac{n(n+1)}{2}$--manifold in a neighborhood of $h_0$.
\item A neighborhood of $\bar b_0$ at $\Psi_{\hP}^{-1}(0)$ is diffeomorphic to
a neighborhood of $h_0$ in the real algebraic set $\Hom(\pi_1(\hP), \PO)$.
\item $\dim \ker D\Psi_{\hP, \bar b_0}=\dim so(1,n) = \tfrac{n(n+1)}{2}$.
\end{itemize}
\end{proposition}

\begin{proof}
%Let $h_0 \co \pi_1(\hP) \rightarrow \PO$ be a discrete faithful representation
%corresponding to a hyperbolic structure on $\hP$.
Let $\pi_1(\hP)$ act on the Lie algebra $so(1,n)$ of $\PO$ by the representation $Ad \circ h_0$.
By the work of Weil \cite{Weil64}, the Zariski tangent space to  $\Hom (\pi_1(\hP), \PO)$ at $h_0$ is
isomorphic to the vector space $Z^1(\pi_1(\hP), so(1,n)_{Ad\circ h_0})$ of 1--cocyles for computing the group cohomology.
(See also Raghunathan \cite[Chapters 6 and 7]{Raghunathan72}
and Goldman \cite[Section 1]{Goldman84} for a material on cycles and cocyles.)

A neighborhood of $\Hom(\pi_1(\hP), \PO)$ of $h_0$
consists of holonomies of hyperbolic Coxeter orbifolds diffeomorphic to $\hP$ by Theorem 1 of \cite{Choi04}.
The Mostow rigidity shows that a neighborhood of $h_0$ in
 $\Hom(\pi_1(\hP), \PO)$ is inside the orbit of $h_0$ under the conjugation action of $\PO$. 
 The orbit is a smooth $\tfrac{n(n+1)}{2}$--manifold in a neighborhood of $h_0$ by an easy real algebraic group action theory
 since the hyperbolic holonomy group $h_0 (\pi_1(\hP))$
has a trivial centralizer in $\PO$.
This proves the first item.
%\marginpar{Refs for this theory?}

Let $R(1, n)$ denote the subspace of $\PO$ of reflections fixing a hyperplane meeting the positive cone,
and % let $\mathcal{N}$ denote the subspace of $V^f$ where at least one vector is zero.
\[ \mathcal{U}^f:= \{ (b_1, \ldots, b_f) \in V^f| \langle b_i, b_i \rangle = 1, i \in \I_{\hP}\},\]
which is a smooth manifold.
Define the map
\[ {\mathcal{I}'}_{\PO} \co {\mathcal{U}^f} \ra R(1,n)^f\]
by sending  $(b_1, \dots, b_f)$ to $(r_1, \dots, r_f)$ such that
\[r_i(\cdot) =I_V-2 \langle b_i, \,\cdot\,\rangle b_i, \, i=1, \dots, f.\] 
Here, $\{\pm I_V\}^f$ acts on fibers transitively
and the map is a covering map. 
%Here,
%${\mathcal{W}}_{\hP}$ maps to an open subset $U'$ in $R(1,n)^f$.
%Clearly the restriction ${\mathcal{W}}_{\hP} \ra U'$
%is a fibration with the fiber group $\{\pm I_V \}^f$, ie a covering map.
Consider the restriction
\[ {\mathcal{I}''} \co \Psi_{\hP}^{-1}(0) \cap {\mathcal{W}}_{\hP} \subset {\mathcal{W}}_{\hP} 
  \rightarrow \Hom (\pi_1(\hP), \PO) \subset R(1, n)^f ,\]
  where ${\mathcal{W}}_{\hP}$ is an open subset of ${\mathcal U}^f$. 
%by sending  $(b_1, ..., b_f)$ to the homomorphism $h$ such that
%$h(r_i)=I_V-2 \langle b_i, \,\cdot\,\rangle b_i$ with $\pi_1(\hP)=\langle r_i \,|\, (r_i r_j)^{n_{ij}} \rangle$.
%Since a neighborhood of $h_0$ in  $\Hom(\pi_1(\hP), \PO)$ coincides with the orbit of $h_0$ under $\PO$,
%it is clear that $\Psi_{\hP}^{-1}(0)$ is locally isomorphic to $\Hom(\pi_1(\hP), \PO)$ near a hyperbolic point with $h_0$.
The relations defining  $\Psi_{\hP}^{-1}(0)$ and $\Hom (\pi_1(\hP), \PO)$ coincide under $\mathcal{I}''$
and the above restriction ${\mathcal{I}''}$ of ${\mathcal{I}'}_{\PO}$ is a local diffeomorphism to its image.
Here, $\{\pm I_V\}$ acts transitively on fibers. This proves the second item. 
(We are in the situation of diffeomorphic coordinate variable changes, heuristically speaking.)

We also obtain
\[\dim \ker D\Psi_{\hP, \bar b_0}=\dim Z^1(\pi_1(\hP),so(1,n)_{Ad\circ h_0})\]
since the second Zariski tangent space is again given by a system of algebraic equations on $R(1,n)^f$. 
By the Weil infinitesimal rigidity \cite{Weil62}, we have $H^1(\pi_1(\hP), so(1,n)_{Ad\circ h_0}) =0$, and
%$Z^1(\pi_1(\hP), so(1,n)_{Ad})$ is identical to the space of coboundaries $B^1(\pi_1(\hP), so(1,n)_{Ad})$.
it follows that
\[\dim Z^1(\pi_1(\hP),so(1,n)_{Ad\circ h_0}) = \dim B^1(\pi_1(\hP), so(1,n)_{Ad\circ h_0}).\]
Since $\PO$ acts freely on $\Hom(\pi_1(\hP), \PO)$ with smooth orbits,
the dimension $\dim B^1(\pi_1(\hP), so(1,n)_{Ad \circ h_0})$ of the tangent space of the orbit passing $h_0$  is
$ \dim so(1,n) = \tfrac{n(n+1)}{2}.$
This proves the third item.
(See also the proof of Theorem 1 of \cite{Choi12}.)
\end{proof}

\subsection{The main theorem} \label{subs:Maintheorem}

\begin{definition}
A real projective Coxeter $n$--orbifold $\hP$ is \emph{weakly orderable} if the facets of the fundamental polytope $P$ in $\PS$, 
can be labeled by integers $\{1, \dots, f\}$ so that for each facet $F_i$,
%each facet $F_i$ contains ridges $e_{i_1}, \dots, e_{i_{m_i}}$ for $m_i \leq n$ of order $2$
%with the property that each $e_{i_k}$ is in a facet $F_{j_{i_k}}$ of higher index $j_{i_k} > i$ for $k = 1, \dots, m_i$.
%Thus, each face $F_i$ has a collection $\{F_{j_{i_1}}, \dots, F_{j_{i_{m_i}}}\}$ of faces.
\begin{itemize} 
\item  the cardinality of 
the collection %$\{F_{j_{i_1}}, \dots, F_{j_{i_{m_i}}} \hbox{ for }  j_{i_l} > i\}$ containing ridges of $F_i$ of order $2$ 
\[\mathcal{F}_i :=  \{ F_j|\, j > i \hbox{ and the ridge $F_i \cap F_j$ has order $2$ } \}\]
is $\leq n$, and 
\item the collection $\mathcal{F}_i$ is in general position % for each face $F_i$ provided
whenever $\mathcal{F}_i$ is not empty.
\end{itemize}
\end{definition}

Here, the general position for a collection of facets means that 
the defining linear equations of the facets are linearly independent.
For $n=3$, we automatically have the last general position condition by Lemma 3 of \cite{Choi12}. Thus,
the second definition generalizes the earlier definition for $n=3$.

%\marginpar{I added general position}

%\marginpar{no need for this for $n=3$ case.}
Recall that a $n$--polytope $P$ in $\PS$ is simple if exactly $n$ facets meet at each vertex.
Let $f$ and $e$ be the numbers of facets and ridges of $P$ respectively. We introduce an integer
\begin{equation*}
\delta_P = e - nf + \tfrac{n(n+1)}{2}
\end{equation*}
which depends only on the polytope $P$ but not on the orbifold structure. Barnette \cite{Barnette73} 
showed for simple polytopes $P$ that $\delta_P \geq 0$. (See also Greene \cite{greene}.)  
In our context, $\delta_P = 0$ indicates the full rank property of hyperbolic equations. (See Equation \eqref{eqn:deltaP}.)
%Note that compact hyperbolic Coxeter $n$--polytopes are simple. %Denote by $e_+$ the number of ridges of order $\geq 3$ in $\hP$.

%\marginpar{need to change $n=3$. Drop $C1$.}
\begin{theorem}\label{thm:main1}
Let $P$ be a compact hyperbolic Coxeter $n$--polytope, and suppose that $\hP$ is the Coxeter orbifold arising from $P$. 
%, and 
%and $\hP$ does not admit a spherical or Euclidean structure. 
Suppose that %the following two conditions are satisfied\/{\rm :}
\begin{enumerate}
\item[$(C1)$] $\delta_P = 0$ and
\item[$(C2)$] $\hP$ is weakly orderable.
\end{enumerate}
Then a neighborhood of the hyperbolic point in $\De(\hP)$ is homeomorphic to a cell of dimension $e_+(\hP) \!- n$.
\end{theorem}
%In other words, the compact hyperbolic Coxeter $n$--orbifold $\hP$ satisfying both $(C1)$ and $(C2)$ of Theorem \ref{thm:main1} is projectively deformable if
%$e_+(\hP) > n$; otherwise, it is locally rigid.

%\marginpar{move this}
%\begin{remark}
%Examples in Section \ref{subs:C2} illustrate cases when the conclusion of Theorem \ref{thm:main1} do not hold.
%Examples in \fullref{subs:C2} and \ref{subs:C1} illustrate that \fullref{thm:main1} does not hold without both assumptions.
%Also, Ryan Green and Mike Davis also discovered about the same time the importance of the number $\delta_P$ for the deformation spaces
%of real projective structures on Coxeter orbifolds using algebro-topological techniques.
%\end{remark}
%\marginpar{Ryan Green: refs}

\subsection{The main lemma}
\label{subs:MainLemma}

The proof of Lemma \ref{lem:ranksum} is technical, hence in Example \ref{eg:Mainlemma} 
we will introduce a simple example to explain the procedure.
%We assume that $P$ has an ordering of facets $F_1, \dots, F_f$ so that
%$F_i < F_j$ if and only if $i < j$ for $i, j \in \I_{\hP}$.

%Under the condition $(C1)$ and $(C2)$, we obtain that $\rank D\Psi_{\hP, \bar b}$ is of full rank $f+e$ by the rank plus nullity theorem. 
%By Proposition \ref{prop:ident}, we have $\dim \ker D\Psi_{\hP, \bar b}=\dim so(1,n) = \tfrac{n(n+1)}{2}$. 
%We have $(n+1)f = \dim  \mathcal{W}_{\hP}$, and $\Psi_{\hP}$ has $f+e$ equations, and $(n+1)f = f + e +  \tfrac{n(n+1)}{2}$
%since $\delta_P = 0$. 

%% Feb 27 2:50

\begin{lemma}\label{lem:ranksum}
Let $P$ be a compact hyperbolic Coxeter $n$--polytope,
and suppose that $\hP$ is the Coxeter orbifold arising from $P$.
Let $e_2$ be the number of ridges of order $2$, and let $\bar b_0 \in V^f$ be the $f$--tuple $(b_1, \dots, b_f)$ 
of normal unit vectors for facets of $P$,
and $\bar \alpha_0 \in V^{\ast f}$ the $f$--tuple $(\alpha_1, \dots, \alpha_f)$ of dual vectors $\alpha_i= 2\langle b_i, \,\cdot\, \rangle$.
If $\hP$ is weakly orderable, then
\begin{equation*}
\rank D\Phi_{\hP, (\bar \alpha_0, \bar b_0)} = \rank D\Psi_{\hP, \bar b_0} + e_2.
\end{equation*}
%is of full rank $f+e+e_2$. 
\end{lemma}
\begin{proof}

Since $\hP$ is weakly orderable, we order the facets of $P$ so that each facet contains at most $n$ ridges of order $2$
in facets of higher indices. 
%Let $F_k$ be the $k$th facet,
%and let $q$ be the largest index such that $F_q$ contains ridges of order $2$ in facets of higher indices.
Define
\begin{equation*}
\I_{\hP}(k)=\{ i\in \I_{\hP} |\, i>k \text{ and } F_i \cap F_k \text{ is a ridge of order $2$}\} \text{ and } i(k)=|\I_{\hP}(k)|.
\end{equation*}
The set can be empty and $i(k) = 0$. 
We may enumerate
\begin{displaymath}
\I_{\hP}(k)=\{ \I_{\hP}(k,1), \dotsc, \I_{\hP}(k,i(k)) \}
\end{displaymath}
such that if $s<t$, then $\I_{\hP}(k,s) < \I_{\hP}(k,t)$. 
Clearly,
\begin{equation}\label{eqn:iq}
k < \I_{\hP}(k,l) \hbox{ for } 1 \leq l \leq i(k).
\end{equation}
That is,
\begin{align*}
1 < \I_{\hP}(1) &= \{ \I_{\hP}(1,1) < \I_{\hP}(1,2) < \dotsm < \I_{\hP}(1,i(1))  \} \\
2 < \I_{\hP}(2) &= \{ \I_{\hP}(2,1) < \I_{\hP}(2,2) < \dotsm < \I_{\hP}(2,i(2))  \} \\
     &\vdots  \\
q < \I_{\hP}(q) &= \{ \I_{\hP}(q,1) < \I_{\hP}(q,2) < \dotsm < \I_{\hP}(q,i(q))  \}
\end{align*}
for some $q, 1 \leq q <  f$. 
Then we have
\begin{align*}
E_{2, \hP} = \{ & (1,\I_{\hP}(1,1)), (1,\I_{\hP}(1,2)), \dotsc, (1,\I_{\hP}(1,i(1))), \\
         & (2,\I_{\hP}(2,1)), (2,\I_{\hP}(2,2)), \dotsc, (2,\I_{\hP}(2,i(2))), \\
         & \quad \quad \vdots \\
         & (q,\I_{\hP}(q,1)), (q,\I_{\hP}(q,2)), \dotsc, (q,\I_{\hP}(q,i(q))) \}
\end{align*}
where $i(k) \leq n$. We note that
\begin{equation*}
\sum_{k=1}^{q} i(k) = |E_{2, \hP}| = e_2.
\end{equation*}
Define the $1 \times (n+1)f$ matrices
\begin{equation*}
\alpha_{[i]}^{[j]}=(0, \ldots, 0, \underbrace{\alpha_i}_{j \text{th block}} ,0, \ldots, 0) \quad \text{and} \quad
b_{[i]}^{[j]}=(0, \ldots, 0, \underbrace{b^t_i}_{j \text{th block}} ,0, \ldots, 0).
\end{equation*}
Denote by $J$ the $(n+1) \times (n+1)$--diagonal matrix with diagonal entries $-1, 1, \dotsc, 1$.
(We will now omit from $D\Phi_{ij, (\bar \alpha, \bar b)}$ the subscripts $(\bar \alpha, \bar b)$  to simplify.)

We note that $\alpha_i= 2b_i^t J$ and $a_{ij}=a_{ij}$ at the hyperbolic point by Proposition 24 of \cite{Vinberg71} and
the rows of the ${N_{\hP}} \times 2(n+1)f$-matrix $[D\Phi_{\hP}]$ are as follows:
\begin{align}\label{eqn:Dphi}
[D\Phi_{ii}] & = (b_{[i]}^{[i]},\alpha_{[i]}^{[i]}) , (i,i) \in E_{1, \hP} \nonumber \\
[D\Phi_{ij}^{[1]}] & = (b_{[j]}^{[i]},\alpha_{[i]}^{[j]} ) , (i,j) \in E_{2, \hP} \nonumber  \\
[D\Phi_{ij}^{[2]}] & = (b_{[i]}^{[j]},\alpha_{[j]}^{[i]} ),  (i,j) \in E_{2, \hP} \nonumber \\
[D\Phi_{ij}] & = (a_{ij}b_{[i]}^{[j]}+a_{ji}b_{[j]}^{[i]},a_{ji}\alpha_{[i]}^{[j]}+a_{ij}\alpha_{[j]}^{[i]}),
 (i,j) \in E_{3, \hP}
\end{align}
by Equation \eqref{eqn:Phidiff}. (Here, we merely indicate the rows and not write the whole matrix.)
%where $(i, j)$ are correspond to edges and hence to equations.)

%\marginpar{Is this working?}
Before completing the proof, let us give an example to illustrate.
\begin{example}\label{eg:Mainlemma}
As an example, we use a compact $3$--dimensional hyperbolic tetrahedron to illustrate
the method in the proof of Lemma \ref{lem:ranksum}. See Figure \ref{fig:tetrahedron}.
Here, if an edge is labeled $l$, then its dihedral angle is $\tfrac{\pi}{l}$.
We will simply use the inherited notation here with obvious meaning. 

\begin{figure}[ht]
\labellist
\small\hair 2pt
\pinlabel $F_1$ at 171 97
\pinlabel $F_2$ at 95 105
\pinlabel $F_3$ at 135 45
\pinlabel $F_4$ at 30 140
\pinlabel $2$ at 55 116
\pinlabel $2$ at 215 116
\pinlabel $3$ at 122 139
\pinlabel $3$ at 132 17
\pinlabel $5$ at 76 57
\pinlabel $2$ at 193 57
\endlabellist
\centering
\includegraphics[height=3cm]{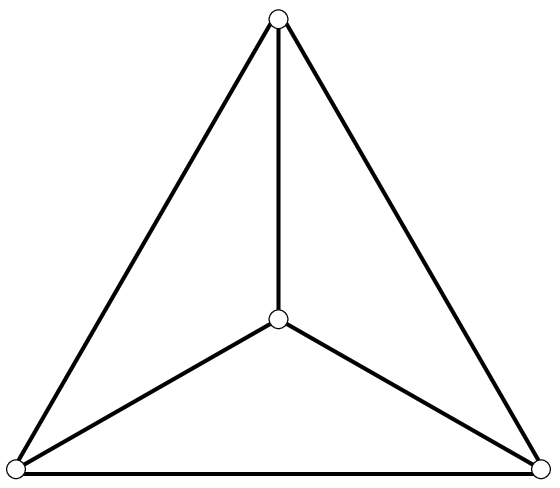}
\caption{A compact hyperbolic tetrahedron}\label{fig:tetrahedron}
\end{figure}
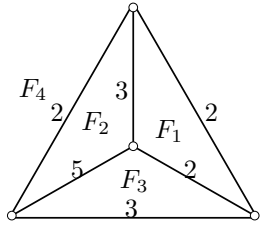

\begin{displaymath}
\tag*{\text{Then}}
\begin{array}{|l|l|}
\hline
\I_{\hP}(1) =  \{\I_{\hP}(1,1)=3 < \I_{\hP}(1,2)=4 \} & E_{2, \hP} =\{ (1,3), (1,4), \\
\I_{\hP}(2) = \{\I_{\hP}(2,1)=4 \}              & \quad \quad \;\;\; (2,4) \} \\
\hline
\end{array}
\end{displaymath}
\begin{displaymath}
E_{3, \hP} =\{(1,2), (2,3), (3,4) \} \quad\text{and}\quad E_{1, \hP} =\{(1,1),(2,2),(3,3),(4,4)\}
\end{displaymath}
and hence
{\footnotesize
\begin{equation*}
[D\Phi_{\hP}]
=
\left[
\begin{array}{c}
\\
D\Phi_{ij}^{[1]}, (i,j) \in E_{2, \hP} \\
\\
\\
D\Phi_{ij}^{[2]}, (i,j) \in E_{2, \hP} \\
\\
\\
D\Phi_{ij} , (i, j) \in E_{3, \hP}\\
\\
\\
D\Phi_{ii},  (i,i) \in E_{1, \hP}\\
\\
\\
\end{array}
\right]
=
\left[
\begin{array}{l}
D\Phi_{13}^{[1]} \\
D\Phi_{14}^{[1]} \\
%\hline
D\Phi_{24}^{[1]} \\
\hline
D\Phi_{13}^{[2]} \\
D\Phi_{14}^{[2]} \\
D\Phi_{24}^{[2]} \\
\hline
D\Phi_{12}       \\
D\Phi_{23}       \\
D\Phi_{34}       \\
D\Phi_{11}       \\
D\Phi_{22}       \\
D\Phi_{33}       \\
D\Phi_{44}
\end{array}
\right]
\end{equation*}

\begin{equation*}
= \left[
\begin{array}{llll|llll}
% 1         2            3            4            1               2               3               4
b_3^t       &0           &0           &0           &0              &0              &\alpha_1       &0              \\
b_4^t       &0           &0           &0           &0              &0              &0              &\alpha_1       \\
%\hline
0           &b_4^t       &0           &0           &0              &0              &0              &\alpha_2       \\
\hline
0           &0           &b_1^t       &0           &\alpha_3       &0              &0              &0              \\
0           &0           &0           &b_1^t       &\alpha_4       &0              &0              &0              \\
0           &0           &0           &b_2^t       &0              &\alpha_4       &0              &0              \\
\hline
a_{21}b_2^t &a_{12}b_1^t &0           &0           &a_{12}\alpha_2 &a_{21}\alpha_1 &0              &0              \\
0           &a_{32}b_3^t &a_{23}b_2^t &0           &0              &a_{23}\alpha_3 &a_{32}\alpha_2 &0              \\
0           &0           &a_{43}b_4^t &a_{34}b_3^t &0              &0              &a_{34}\alpha_4 &a_{43}\alpha_3 \\
b_1^t       &0           &0           &0           &\alpha_1       &0              &0              &0              \\
0           &b_2^t       &0           &0           &0              &\alpha_2       &0              &0              \\
0           &0           &b_3^t       &0           &0              &0              &\alpha_3       &0              \\
0           &0           &0           &b_4^t       &0              &0              &0              &\alpha_4
\end{array}
\right]
\end{equation*}
}

where $0$ is the zero $1 \times 4$--matrix.

First, for each $(i,j) \in E_{2, \hP}$, add a row $[D\Phi_{ij}^{[1]}]$ of $[D\Phi_{\hP}]$ to another row $[D\Phi_{ij}^{[2]}]$:
{\footnotesize
\begin{equation} \label{eqn:firstmove}
\left[
\begin{array}{llll|llll}
% 1         2            3            4            1               2               3               4
b_3^t       &0           &0           &0           &0              &0              &\alpha_1       &0              \\
b_4^t       &0           &0           &0           &0              &0              &0              &\alpha_1       \\
%\hline
0           &b_4^t       &0           &0           &0              &0              &0              &\alpha_2       \\
\hline
{\pmb b_3^t}&0           &b_1^t       &0           &\alpha_3       &0              &{\pmb \alpha_1}&0              \\
{\pmb b_4^t}&0           &0           &b_1^t       &\alpha_4       &0              &0              &{\pmb \alpha_1}\\
0           &{\pmb b_4^t}&0           &b_2^t       &0              &\alpha_4       &0              &{\pmb \alpha_2}\\
\hline
a_{21}b_2^t &a_{12}b_1^t &0           &0           &a_{12}\alpha_2 &a_{21}\alpha_1 &0              &0              \\
0           &a_{32}b_3^t &a_{23}b_2^t &0           &0              &a_{23}\alpha_3 &a_{32}\alpha_2 &0              \\
0           &0           &a_{43}b_4^t &a_{34}b_3^t &0              &0              &a_{34}\alpha_4 &a_{43}\alpha_3 \\
b_1^t       &0           &0           &0           &\alpha_1       &0              &0              &0              \\
0           &b_2^t       &0           &0           &0              &\alpha_2       &0              &0              \\
0           &0           &b_3^t       &0           &0              &0              &\alpha_3       &0              \\
0           &0           &0           &b_4^t       &0              &0              &0              &\alpha_4
\end{array}
\right].
\end{equation}
}

Second, for  $(i,j) \in E_{3, \hP}$, multiply a row $[D\Phi_{ij}]$ of $[D\Phi_{\hP}]$ by $a_{ij}^{-1}$:
{\footnotesize
\begin{displaymath}
\left[
\begin{array}{llll|llll}
% 1         2            3            4            1               2               3               4
b_3^t       &0           &0           &0           &0              &0              &\alpha_1       &0              \\
b_4^t       &0           &0           &0           &0              &0              &0              &\alpha_1       \\
%\hline
0           &b_4^t       &0           &0           &0              &0              &0              &\alpha_2       \\
\hline
b_3^t       &0           &b_1^t       &0           &\alpha_3       &0              &\alpha_1       &0              \\
b_4^t       &0           &0           &b_1^t       &\alpha_4       &0              &0              &\alpha_1       \\
0           &b_4^t       &0           &b_2^t       &0              &\alpha_4       &0              &\alpha_2       \\
\hline
{\pmb b_2^t}&{\pmb b_1^t}&0           &0           &{\pmb \alpha_2}&{\pmb \alpha_1}&0              &0              \\
0           &{\pmb b_3^t}&{\pmb b_2^t}&0           &0              &{\pmb \alpha_3}&{\pmb \alpha_2}&0              \\
0           &0           &{\pmb b_4^t}&{\pmb b_3^t}&0              &0              &{\pmb \alpha_4}&{\pmb \alpha_3} \\
b_1^t       &0           &0           &0           &\alpha_1       &0              &0              &0              \\
0           &b_2^t       &0           &0           &0              &\alpha_2       &0              &0              \\
0           &0           &b_3^t       &0           &0              &0              &\alpha_3       &0              \\
0           &0           &0           &b_4^t       &0              &0              &0              &\alpha_4
\end{array}
\right].
\end{displaymath}
}

Third, for each $(i,i) \in E_{1, \hP}$, multiply a row $[D\Phi_{ii}]$ of $[D\Phi_{\hP}]$ by $2$:
{\footnotesize
\begin{displaymath}
\left[
\begin{array}{llll|llll}
% 1         2            3            4            1               2               3               4
b_3^t       &0           &0           &0           &0              &0              &\alpha_1       &0              \\
b_4^t       &0           &0           &0           &0              &0              &0              &\alpha_1       \\
%\hline
0           &b_4^t       &0           &0           &0              &0              &0              &\alpha_2       \\
\hline
b_3^t       &0           &b_1^t       &0           &\alpha_3       &0              &\alpha_1       &0              \\
b_4^t       &0           &0           &b_1^t       &\alpha_4       &0              &0              &\alpha_1       \\
0           &b_4^t       &0           &b_2^t       &0              &\alpha_4       &0              &\alpha_2       \\
\hline
b_2^t       &b_1^t       &0           &0           &\alpha_2       &\alpha_1       &0              &0              \\
0           &b_3^t       &b_2^t       &0           &0              &\alpha_3       &\alpha_2       &0              \\
0           &0           &b_4^t       &b_3^t       &0              &0              &\alpha_4       &\alpha_3       \\
{\pmb 2b_1^t}&0          &0           &0           &{\pmb 2\alpha_1}&0             &0              &0              \\
0           &{\pmb 2b_2^t}&0          &0           &0              &{\pmb 2\alpha_2}&0              &0              \\
0           &0           &{\pmb 2b_3^t}&0           &0             &0              &{\pmb 2\alpha_3}&0              \\
0           &0           &0           &{\pmb 2b_4^t}&0             &0              &0              &{\pmb 2\alpha_4}
\end{array}
\right].
\end{displaymath}
}

Fourth, multiply the left $16$ columns of $[D\Phi_{\hP}]$ by $2$ and 
the $(4i-3)$th columns $(i \in \I_{\hP}=\{1,2,3,4\})$ of $[D\Phi_{\hP}]$ by $-1$ respectively:
{\footnotesize
\begin{displaymath}
\left[
\begin{array}{llll|llll}
% 1          2             3             4            1               2               3               4
\pmb\alpha_3 &0            &0            &0            &0              &0              &\alpha_1       &0              \\
\pmb\alpha_4 &0            &0            &0            &0              &0              &0              &\alpha_1       \\
%\hline
0            &\pmb\alpha_4 &0            &0            &0              &0              &0              &\alpha_2       \\
\hline
\pmb\alpha_3 &0            &\pmb\alpha_1 &0            &\alpha_3       &0              &\alpha_1       &0              \\
\pmb\alpha_4 &0            &0            &\pmb\alpha_1 &\alpha_4       &0              &0              &\alpha_1       \\
0            &\pmb\alpha_4 &0            &\pmb\alpha_2 &0              &\alpha_4       &0              &\alpha_2       \\
\hline
\pmb\alpha_2 &\pmb\alpha_1 &0            &0            &\alpha_2       &\alpha_1       &0              &0              \\
0            &\pmb\alpha_3 &\pmb\alpha_2 &0            &0              &\alpha_3       &\alpha_2       &0              \\
0            &0            &\pmb\alpha_4 &\pmb\alpha_3 &0              &0              &\alpha_4       &\alpha_3       \\
2\pmb\alpha_1&0            &0            &0            &2\alpha_1      &0              &0              &0              \\
0            &2\pmb\alpha_2&0            &0            &0              &2\alpha_2      &0              &0              \\
0            &0            &2\pmb\alpha_3&0            &0             &0               &2\alpha_3      &0              \\
0            &0            &0            &2\pmb\alpha_4&0             &0               &0              &2\alpha_4
\end{array}
\right]  %= [D\Phi_{\hP}]^{\prime},
\end{displaymath}
}
ie
{\footnotesize
\begin{displaymath}
%[D\Phi_{\hP}]^{\prime} =
 \left[
\begin{array}{llll|llll}
% 1          2             3             4            1               2               3               4
\alpha_3 &0            &0            &0            &0              &0              &\alpha_1       &0              \\
\alpha_4 &0            &0            &0            &0              &0              &0              &\alpha_1       \\
%\hline
0            &\alpha_4 &0            &0            &0              &0              &0              &\alpha_2       \\
\hline
&            &          &\hspace{-13mm}[D\Psi_{\hP}] & &             &              &\hspace{-11mm} [D\Psi_{\hP}]
\end{array}
\right].
\end{displaymath}
}
(See Section \ref{subs:TangentToHyperbolic} for definition of $[D\Psi_{\hP}]$.)
Here we note that
{\footnotesize
\begin{displaymath}
[D\Psi_{\hP}] =
\left[
\begin{array}{c}
\\
D\Psi_{ij}, (i,j) \in E_{2, \hP} \\
\\
\\
D\Psi_{ij}, (i,j) \in E_{3, \hP} \\
\\
\\
D\Psi_{ii}, (i,i) \in E_{1, \hP} \\
\\
\\
\end{array}
\right]
=
\left[
\begin{array}{l}
D\Psi_{13} \\
D\Psi_{14} \\
D\Psi_{24} \\
D\Psi_{12} \\
D\Psi_{23} \\
D\Psi_{34} \\
D\Psi_{11} \\
D\Psi_{22} \\
D\Psi_{33} \\
D\Psi_{44}
\end{array}
\right]
=
\left[
\begin{array}{llll}
\alpha_3       &0              &\alpha_1       &0              \\
\alpha_4       &0              &0              &\alpha_1       \\
0              &\alpha_4       &0              &\alpha_2       \\
\alpha_2       &\alpha_1       &0              &0              \\
0              &\alpha_3       &\alpha_2       &0              \\
0              &0              &\alpha_4       &\alpha_3       \\
2\alpha_1      &0              &0              &0              \\
0              &2\alpha_2      &0              &0              \\
0             &0               &2\alpha_3      &0              \\
0             &0               &0              &2\alpha_4
\end{array}
\right].
\end{displaymath}
}
Finally, using elementary column operations, we obtain
{\footnotesize
\begin{displaymath}
\left[
\begin{array}{l|l|l|l|llll}
% 1          2             3             4             1               2               3               4
\alpha_3     &0            &-\pmb\alpha_1&0            &0              &0              &\alpha_1       &0              \\
\alpha_4     &0            &0            &-\pmb\alpha_1&0              &0              &0              &\alpha_1       \\
%\hline
0            &\alpha_4     &0            &-\pmb\alpha_2&0              &0              &0              &\alpha_2       \\
\hline
O_{10\times 4}& O_{10\times 4}& O_{10\times 4}& O_{10\times 4}& &             &              &\hspace{-11mm} [D\Psi_{\hP}]
\end{array}
\right]
\end{displaymath}
}
where $O_{s\times t}$ is the $s\times t$ zero matrix. 
Hence, the matrix is of rank $=\rank  [D\Psi_{\hP}] +e_2$.
\end{example}

Now, we continue with the proof of Lemma \ref{lem:ranksum}: Using the notation as before,
we recall our matrix $[D\Phi_{\hP}]$ in Equation \eqref{eqn:Dphi}.
Now we use elementary row and column operations of $[D\Phi_{\hP}]$ to obtain a matrix whose rank is easier to compute.
The step will correspond to one after Equation \eqref{eqn:firstmove} in the above example.

First, for  $(i,j) \in E_{2, \hP}$, add a row $[D\Phi_{ij}^{[1]}]$ of $[D\Phi_{\hP}]$ to another row $[D\Phi_{ij}^{[2]}]$:
\begin{displaymath}
(b_{[i]}^{[j]},\alpha_{[j]}^{[i]} ) \rightarrow (b_{[i]}^{[j]}+b_{[j]}^{[i]},\alpha_{[i]}^{[j]}+\alpha_{[j]}^{[i]} ).
\end{displaymath}
Second, for  $(i,j) \in E_{3, \hP}$, multiply a row $[D\Phi_{ij}]$ of $[D\Phi_{\hP}]$ by $a_{ij}^{-1}$:
\begin{displaymath}
(a_{ij}b_{[i]}^{[j]}+a_{ji}b_{[j]}^{[i]},a_{ji}\alpha_{[i]}^{[j]}+a_{ij}\alpha_{[j]}^{[i]})  \rightarrow (b_{[i]}^{[j]}+b_{[j]}^{[i]},\alpha_{[i]}^{[j]}+\alpha_{[j]}^{[i]}).
\end{displaymath}
Recall that for $(i,j) \in E_{3, \hP}$ each $a_{ij}$ is non-zero and $a_{ij}=a_{ji}$ at the hyperbolic point in $\Phi^{-1}(0)$.

Third, for  $(i,i) \in E_{1, \hP}$, multiply a row $[D\Phi_{ii}]$ of $[D\Phi_{\hP}]$ by $2$:
\begin{displaymath}
(b_{[i]}^{[i]},\alpha_{[i]}^{[i]}) \rightarrow (2b_{[i]}^{[i]},2\alpha_{[i]}^{[i]}).
\end{displaymath}

Fourth, multiply the left $(n+1)f$ columns of $[D\Phi_{\hP}]$ by $2$ and the $(i(n+1)-n)$th columns $(i \in \I_{\hP})$ of $[D\Phi_{\hP}]$ by $-1$ respectively
(The fact that $\alpha_i= 2\nu_i^t J$ and $b_{i}=\nu_i$ is used here):
\begin{align*}
[D\Phi_{ij}^{[1]}] & \rightarrow (\alpha_{[j]}^{[i]},\alpha_{[i]}^{[j]} ) ,  (i,j) \in E_{2, \hP} \\
[D\Phi_{ij}^{[2]}] & \rightarrow (\alpha_{[i]}^{[j]}+\alpha_{[j]}^{[i]},\alpha_{[i]}^{[j]}+\alpha_{[j]}^{[i]} ),  (i,j) \in E_{2, \hP} \\
[D\Phi_{ij}] & \rightarrow (\alpha_{[i]}^{[j]}+\alpha_{[j]}^{[i]},\alpha_{[i]}^{[j]}+\alpha_{[j]}^{[i]} ),  (i,j) \in E_{3, \hP}  \\
[D\Phi_{ii}] & \rightarrow (2\alpha_{[i]}^{[i]},2\alpha_{[i]}^{[i]}), (i,i) \in E_{1, \hP}.
\end{align*}

Similarly, the rows of the $(f+e) \times (n+1)f$ Jacobian matrix $[D\Psi_{\hP}]$ are  as follows:
\begin{align*}
[D\Psi_{ij}] & = \alpha_{[i]}^{[j]} + \alpha_{[j]}^{[i]}, (i,j) \in E_{2, \hP} \cup E_{3, \hP} \\
[D\Psi_{ii}] & = 2 \alpha_{[i]}^{[i]} , (i,i) \in E_{1, \hP}.
\end{align*}
(See Section \ref{subs:TangentToHyperbolic} for definition of $[D\Psi_{\hP}]$.)
Comparing these two matrices and rearranging, we observe that $[D\Phi_{\hP}]$ became
\begin{equation*}
\left[
\begin{array}{c|c}
\alpha_{[\I_{\hP}(1,1)]}^{[1]}    & \alpha_{[1]}^{[\I_{\hP}(1,1)]} \\
\vdots               & \vdots \\
\alpha_{[\I_{\hP}(1,i(1))]}^{[1]} & \alpha_{[1]}^{[\I_{\hP}(1,i(1))]} \\
%\hline
\vdots                & \vdots \\
\vdots                & \vdots \\
%\hline
\alpha_{[\I_{\hP}(q,1)]}^{[q]} & \alpha_{[q]}^{[\I_{\hP}(q,1)]}   \\
\vdots                & \vdots \\
 \alpha_{[\I_{\hP}(q,i(q))]}^{[q]} & \alpha_{[q]}^{[\I_{\hP}(q,i(q))]} \\
\hline
{[D\Psi_{\hP}]} & {[D\Psi_{\hP}]}
\end{array}
\right]
\end{equation*}
which is an $N_{\hP} \times 2(n+1)f$ matrix divided into two $e_2 \times (n+1)f$--matrices. 
The top ones correspond to the copies of $E_{2, \hP}$--rows, 
and two bottom $(f+e) \times (n+1)f$--matrices equal $[D\Psi_{\hP}]$.

Using elementary column operations, we obtain
%a $N_{\hP} \times 2(n+1)f$ matrix divided into two $N_{\hP} \times (n+1)f$--matrices
{
\begin{equation*}
\left[
\begin{array}{c|c}
\alpha_{[\I_{\hP}(1,1)]}^{[1]}  - \alpha_{[1]}^{[\I_{\hP}(1,1)]}    & \alpha_{[1]}^{[\I_{\hP}(1,1)]} \\
\vdots                & \vdots \\
\alpha_{[\I_{\hP}(1,i(1))]}^{[1]} -\alpha_{[1]}^{[\I_{\hP}(1,i(1))]} & \alpha_{[1]}^{[\I_{\hP}(1,i(1))]} \\
%\hline
\vdots                & \vdots \\
\vdots                & \vdots \\
%\hline
\alpha_{[\I_{\hP}(q,1)]}^{[q]} - \alpha_{[q]}^{[\I_{\hP}(q,1)]}   & \alpha_{[q]}^{[\I_{\hP}(q,1)]} \\
\vdots                & \vdots \\
\alpha_{[\I_{\hP}(q,i(q))]}^{[q]}  - \alpha_{[q]}^{[\I_{\hP}(q,i(q))]} & \alpha_{[q]}^{[\I_{\hP}(q,i(q))]} \\
\hline
0 & {[D\Psi_{\hP}]}
\end{array}
\right].
\end{equation*}
}
Now rewriting this matrix into  the union of $f$ of $N_{\hP} \times (n+1)$--matrices and $N_{\hP} \times f(n+1)$--matrix, we obtain
{
\begin{equation*}
\left[
\begin{array}{c|c|c|c|c|c}
\alpha_{\I_{\hP}(1,1)}  & * & * & * & \dots & *\\
\vdots & \vdots & \vdots & \vdots & \ddots & \vdots \\
\alpha_{\I_{\hP}(1,i(1))} & * & * & * & \dots & * \\
\hline
0 & \alpha_{\I_{\hP}(2,1)} & * & * & \dots &* \\
\vdots & \vdots & \vdots & \vdots & \ddots & \vdots \\
0 & \alpha_{\I_{\hP}(2,i(2))} & * & * & \dots & * \\
\hline
\vdots & \vdots & \ddots & \vdots & \ddots &\vdots \\
\hline
0 & 0 & \dots & \alpha_{\I_{\hP}(q,1)}  & \dots & *\\
\vdots & \vdots & \vdots & \vdots & \ddots & \vdots \\
0 & 0 & \dots & \alpha_{\I_{\hP}(q,i(q))}  & \dots & *\\
\hline
0 & 0 & \dots & 0  & 0 &  {[D\Psi_{\hP}]}\\
\end{array}
\right]
\end{equation*}
}
where $0$'s are zero matrices.
The matrix is so that
\begin{displaymath}
\left[
\begin{array}{l}
\alpha_{\I_{\hP}(k,1)} \\
\alpha_{\I_{\hP}(k,2)} \\
\dotsc \\
\alpha_{\I_{\hP}(k,i(k))}
\end{array} \right]
\end{displaymath}
is in the $k$--th column from rows $\sum_{j=1}^{k-1} i(j)+1$ to $\sum_{j=1}^{k} i(j)$
and every entry below is zero
for $k = 1 , \dots, q$.

%\fullref{lem:independent}
The general position condition of the weak orderability implies that for each $k \in \{1, 2, \dotsc, q\}$
\begin{equation*}
\alpha_{\I_{\hP}(k,1)}, \alpha_{\I_{\hP}(k,2)}, \dotsc, \alpha_{\I_{\hP}(k,i(k))}
\end{equation*}
are linearly independent , ie
all submatrices
\begin{displaymath}
\left[
\begin{array}{l}
\alpha_{\I_{\hP}(k,1)} \\
\alpha_{\I_{\hP}(k,2)} \\
\dotsc \\
\alpha_{\I_{\hP}(k,i(k))}
\end{array} \right]
\end{displaymath}
are of full rank. This establishes the result.
\end{proof}

%The following lemma will be used again in \fullref{subs:Proof}.
%\begin{lemma}\label{lem:independent}
%Let $P$ be a properly
%convex $n$--polytope in $\PS$ defined by linear inequalities $\alpha_i \geq 0$ %
%where $F_i$ be the facet of $P$ determined by $\alpha_i$. Suppose that the facets $F_{i_1}, \dotsc, F_{i_k}$ for $k \leq n$
%are adjacent to the facet $F_{i_{0}}$.
%Then the $(n+1)$ linear functionals $\alpha_{i_0}, \alpha_{i_1}, \dotsc, \alpha_{i_{k}}$ are linearly independent.
%\end{lemma}
%\begin{proof}

%One equality and $n$ linear inequalities
%\begin{displaymath}
%\alpha_{i_0} =0, \quad \alpha_{i_1} \geq 0, \dotsc, \alpha_{i_k} \geq 0
%\end{displaymath}
%give an $(n-1)$--dimensional simplex in the great $(n-1)$--sphere in $\PS$ containing $F_{i_0}$
%by our general position condition.
%Hence there is no non-zero vector $b$ such that
%\begin{equation*}
%\alpha_{i_0}(b)=\alpha_{i_1}(b)=\dotsm=\alpha_{i_{n}}(b)=0
%\end{equation*}
%establishing the result.
%\end{proof}

\subsection{Proof of Theorem \ref{thm:main1}}
\label{subs:Proof}

Let $\hP$ be a compact Coxeter $n$--orbifold admitting a hyperbolic structure. 
Assume that $\hP$ admits a real projective structure, but does not admit a spherical or Euclidean structure. Define
\begin{align*}
\mathcal{V}_{\hP} &:=  \{ p \in \mathcal{U}_{\hP} \subset (V^*)^f \times V^f \,|\,
 \text{$D_p\Phi_{\hP}$ is surjective}  \}, \\
\widetilde{\De}(\hP)_r &:= \widetilde{\De}(\hP) \cap \mathcal{V}_{\hP} = \Phi_{\hP}^{-1}(0) \cap \mathcal{V}_{\hP}.
\end{align*}
 The second one is 
 an open subset since the maximal rank condition is an open condition 
 since the rank condition expresses the independence of
 the row vectors of the differential.

Since $\Phi_{\hP}$ is $\widetilde{G}$--invariant, %%%% 12:04 May 14, 2013
%it follows that 
$\widetilde{G}$ acts on $\widetilde{\De}(\hP)_r$. 
The action of $\widetilde{G}$ on $\widetilde{\De}(\hP)_r$ is induced from the action $\theta$ on $\widetilde{\De}(\hP)$ in Equation \eqref{eqn:theta}.
Recall that ${N_{\hP}}=f+e+e_2$, where $f$, $e$ and $e_2$ are the number of facets, ridges and ridges of order $2$ of $\hP$ respectively.

We use the following steps:

\begin{enumerate}
\item[1.] $\widetilde{\De}(\hP)_r$ is a smooth manifold of dimension $2(n+1)f-{N_{\hP}}$ if $\widetilde{\De}(\hP)_r \ne \emptyset$.
\item[2.] The orbit space $\De(\hP)_r :=\widetilde{\De}(\hP)_r/\widetilde{G}$ is a smooth manifold of dimension $\dim \widetilde{\De}(\hP)_r - \dim \widetilde{G}$,
and it identifies with an open subset of $\De(\hP)$.
\item[3.] Moreover, if $P$ satisfies the condition $(C1)$, then the manifold $\De(\hP)_r$ is of dimension $e_+(\hP)\! -n$.
\item[4.] Furthermore, if $\hP$ admits a hyperbolic structure and satisfies the condition $(C2)$,
the hyperbolic point $t$ is in $\De(\hP)_r$. This will complete the proof of Theorem \ref{thm:main1}.
\end{enumerate}

We start: 
%Now, we prove Theorem \ref{thm:main1} in the above steps.
\begin{enumerate}
\item[1.]  The set $\mathcal{V}_{\hP}$ is an open subset of $\mathcal{U}_{\hP} \subset (V^*)^f \times V^f$ and
the restriction to $ \mathcal{V}_{\hP}$ of the map $\Phi_{\hP}$ is a submersion.
Thus each level set of $\Phi_{\hP} |_{ \mathcal{V}_{\hP}}$
is an embedded submanifold in $\mathcal{V}_{\hP}$ whose codimension is ${N_{\hP}}$.
The conclusion is immediate.
\item[2.] %By \fullref{thm:defrepspaces},
As we defined above,
$\widetilde{\De}(\hP)_r$ is an open subset of $\widetilde{\De}(\hP)$ that is the complement of an algebraic
closed set and $\widetilde{G}$ acts on both sets.
By Lemma \ref{lem:manifold}, $\widetilde{G}$ acts smoothly, freely, and properly on a smooth manifold $\widetilde{\De}(\hP)_r$,
and hence the orbit space $\widetilde{\De}(\hP)_r/\widetilde{G}$ is a smooth manifold of dimension
$\dim \widetilde{\De}(\hP)_r - \dim \widetilde{G}$. 
Therefore, $\widetilde{\De}(\hP)_r/\widetilde G$ identifies with an open subset of
$\De(\hP)$ by Theorem \ref{thm:identification}.

\item[3.] $\dim \widetilde{\De}(\hP)_r - \dim \widetilde{G} = (2(n+1)f-{N_{\hP}})-(f+(n+1)^2-1)= e_{+}-n-2 \delta_P$ holds.
Since $\delta_P=0$, we obtain $\dim \widetilde{\De}(\hP)_r  -\dim \widetilde{G} = e_+(\hP)\! -n$, and hence
the step 2 implies the conclusion.
\item[4.] Proposition \ref{prop:ident} yields $\ker D\Psi_{\hP} =  \tfrac{n(n+1)}{2}$ at the hyperbolic point $t$. 
Hence
\begin{equation}\label{eqn:deltaP}
\rank D\Psi_{\hP} = (n+1)f - \tfrac{n(n+1)}{2} = f+e-\delta_P
\end{equation}
holds where $\delta_P = e-nf+\tfrac{n(n+1)}{2}$.
Since $\delta_P=0$ and $\hP$ is weakly orderable, 
$\rank D\Phi_{\hP}= \rank D\Psi_{\hP}+e_2=f+e+e_2$ at $t$
by Lemma \ref{lem:ranksum}, and so $D\Phi_{\hP}$ at $t$ is of full rank.
%Hence the differential of $\Phi_{\hP}$ at the hyperbolic point $t$ is surjective.
\end{enumerate}

\begin{lemma}\label{lem:manifold}
Let $\hP$ be a compact Coxeter $n$--orbifold.
Assume that $\hP$ admits a real projective structure, but does not admit a spherical or Euclidean structure.
Then $\widetilde{G}$ acts smoothly, freely and properly on a smooth manifold $\widetilde{\De}(\hP)_r$.
%In particular, the orbit space $\widetilde{\De}(\hP)_r/\widetilde{G}$
%is a smooth manifold of dimension $\dim \widetilde{\De}(\hP)_r - \dim \widetilde{G}$, and it is open in $\De(\hP)$.
\end{lemma}

\begin{proof}
We show that $\widetilde{G}$ acts freely on a smooth manifold $\widetilde{\De}(\hP)_r$, a locally compact metric space.
Suppose that
\begin{equation*}
(d_1, \dotsc, d_f, g) \cdot (\alpha_1, \dotsc, \alpha_f, b_1, \dotsc, b_f)  \\ = (\alpha_1, \dotsc, \alpha_f, b_1, \dotsc, b_f)
\end{equation*}
where $d_1, \dotsc, d_f \in \mathbb{R}_+$ and $g \in \SL$. That is,
\begin{equation*}
d_i \alpha_i g^{-1} = \alpha_i \quad \text{and} \quad d_i^{-1} g b_i = b_i \quad \text{for every } i \in \I.
\end{equation*}
Hence, $d_i d_j^{-1} a_{ij} = a_{ij}$ holds, and $d_i = d_j$ if $\alpha_i(b_j) \neq 0$.

By Proposition \ref{prop:vinberg2}, for any holonomy group $\Gamma$ of $\pi_1(\hP)$, the Cartan matrix of $\Gamma$ is indecomposable.
It follows that $d_1 = \dotsm = d_f$. Denote the common value by $d$.
Choose $(n+1)$ linearly independent linear functionals $\alpha_{i_0}, \alpha_{i_1}, \dotsc , \alpha_{i_{n}}$ from the facets of $\hP$
since the fundamental domain is a properly convex polytope by the condition of $\mathcal{U}_{\hP}$. 
Let $S$ be an invertible $(n+1)\times (n+1)$ matrix
\begin{displaymath}
S= \left[
\begin{array}{c}
\alpha_{i_0} \\
\alpha_{i_1} \\
\dotsc \\
\alpha_{i_{n}}
\end{array} \right].
\end{displaymath}
Then $d\, S g^{-1} = S$ and hence $d^{n+1}=\text{det}(g)=1$. 
Observe that $d=1$ and $g=I_{n+1}$ establishing the result.
%Clearly $-\I_V$ does not fix any pont of $\widetilde{\De}(\hP)_r$.

Next, we show that $\widetilde{G}$ acts properly on a smooth manifold $\widetilde{\De}(\hP)_r$.
Suppose that a sequence $\{p_k =(\alpha_{1,k}, \dotsc, \alpha_{f,k}, b_{1,k}, \dotsc, b_{f,k})\}$ in $\widetilde{G}$ is such that 
\[\{ p_k \} \ra (\alpha_1, \dotsc, \alpha_f, b_1, \dotsc, b_f) \in \widetilde{\De}(\hP)_r,\] 
and $\{q_k = (d_{1,k},\dotsc, d_{f,k}, g_k)\}$ is a sequence in $\widetilde{G}$ such that
\[q_k \cdot p_k \ra (\widetilde{\alpha}_1, \dotsc, \widetilde{\alpha}_f, \widetilde{b}_1, \dotsc, \widetilde{b}_f) \in \widetilde{\De}(\hP)_r
\hbox{ as } k \ra \infty.\]
That is,
\begin{equation}\label{eqn:converge}
\{d_{i,k}\alpha_{i,k} g_{k}^{-1}\} \rightarrow \widetilde{\alpha_{i}} 
\quad \text{and} \quad \{d_{i,k}^{-1}g_k b_{i,k}\} \rightarrow \widetilde{b_i} \text{ for each } i \in \I.
\end{equation}

Since we are in a metric space,  
we show that $\{q_k\}$ is bounded to prove the properness of the action:
We have
\[\{d_{i,k}d_{j,k}^{-1}\alpha_{i,k} b_{j,k}\} \rightarrow \widetilde{\alpha}_i \widetilde{b}_j, \hbox{ hence }
\{d_{i,k}d_{j,k}^{-1}\} \rightarrow \widetilde{\alpha}_i \widetilde{b}_j (\alpha_i b_j)^{-1} \hbox{ if } \alpha_i b_j \neq 0.\]
Moreover,
\[\widetilde{\alpha}_i \widetilde{b}_j, \widetilde{\alpha}_j \widetilde{b}_i, \alpha_i b_j,
 \alpha_j b_i < 0 , (i,j) \not\in E_1 \cup E_2.\]
Since the Cartan matrices $A=(a_{ij})$, $a_{ij}=\alpha_i b_j$, 
and $\widetilde{A}=(\widetilde{a}_{ij})$, $\widetilde{a}_{ij}=\widetilde{\alpha_i} \widetilde{b_j}$,
are indecomposable, % it follows that for each $(i,j) \in \I\times \I$
%there exist $c_{ij} > 0$ for each $(i,j) \in \I\times \I$ such that
\begin{equation}\label{eqn:converge2}
\{d_{i,k}d_{j,k}^{-1}\} \rightarrow c_{ij} >0 \hbox{ for every } (i,j) \in \I\times \I.
\end{equation}
%For any facet $F_{i_0}$ of $P$, we can choose $(n+1)$ linearly independent linear functionals $\alpha_{i_0}, \alpha_{i_1}, \dotsc , \alpha_{i_{n}}$
%as $\hP$ has a properly convex polytope as the fundamental domain.
Define $(n+1) \times (n+1)$ matrices
\begin{displaymath}
\widetilde{S}= \left[
\begin{array}{c}
\widetilde{\alpha}_{i_0} \\
\widetilde{\alpha}_{i_1} \\
\vdots \\
\widetilde{\alpha}_{i_{n}}
\end{array} \right],
\quad
S = \left[
\begin{array}{c}
c_{i_0i_0}\alpha_{i_0} \\
c_{i_1i_0} \alpha_{i_1} \\
\vdots \\
c_{i_{n}i_0} \alpha_{i_{n}}
\end{array} \right]
\quad \text{and} \quad
S_k = \left[
\begin{array}{c}
d_{i_0,k} \alpha_{i_0,k} \\
d_{i_1,k} \alpha_{i_1,k} \\
\vdots \\
d_{i_{n},k} \alpha_{i_{n},k}
\end{array}\right].
\end{displaymath}
Since $\widetilde{S}$ and $S$ are invertible, 
Equations \eqref{eqn:converge} and \eqref{eqn:converge2} show that 
\[\{S_k g_k^{-1}\} \rightarrow \widetilde{S} \hbox{ and }\{d_{i_0,k}^{-1} S_k \}\rightarrow S;
\hbox{ hence,  } \{d_{i_0,k} g_k^{-1}\} \rightarrow S^{-1} \widetilde{S}.\]
As $\text{det}\,g_k = \pm 1$,
%Note that two ordered basis $( \widetilde{\alpha}_{i_0}, \widetilde{\alpha}_{i_1}, \dotsc, \widetilde{\alpha}_{i_{n}})$ and
%$(\alpha_{i_0}, \alpha_{i_1}, \dotsc, \alpha_{i_{n}})$ consistently (resp. oppositely) oriented if $\text{det}\,g_k = 1$ (resp. $-1$).
the sequence $\{d_{i_0,k}^{n+1}\}$ converges to a positive number
\begin{displaymath}
|\text{det}(S^{-1} \widetilde{S})|.
\end{displaymath}
Denote by $d_{i_0}$ the positive $(n+1)$th root of this limit.
$\{d_{i_0,k}\} \ra d_{i_0}$ and $\{g_k\}$ limits to 
$\{d_{i_0}\widetilde{S}^{-1}S\}$ respectively. 
Since we can choose a collection of faces to include $\alpha_{i_0}$ for any $i_0$, 
$\{q_k\}$ is convergent. 
%since we can choose a different collection of facets to show
%Since the index $i_0$ in $\I$ can be arbitrarily chosen, the conclusion is immediate.
\end{proof}

%% June 23 8:13 pm
%\marginpar{Would the freeness imply properness? Dont see a ref?}

\subsection{Proofs of  Corollaries \ref{cor:main1} and \ref{cor:truncation}.} \label{subs:proofs}

\begin{proof}[Proof of Corollary \ref{cor:main1}]
%Since $\hP$ admits a hyperbolic structure, $\hP$ does not admit a Euclidean or spherical structure. 
$\delta_P =0$ for any $3$--dimensional simple polytope $P$ by the Euler's formula.
Hence, Theorem \ref{thm:main1} gives us the conclusion.
\end{proof}

\begin{proof}[Proof of Corollary \ref{cor:truncation}]
Let $\hP$ be a hyperbolic truncation Coxeter orbifold with the fundamental polytope $P \subset \PS$. 
For $n=3$, this is the work of Marquis \cite{Marquis10}.
For $n\geq 4$, as shown in Br{\o}ndsted \cite[$\S 19$]{Brondsted83},
$P$ is a truncation $n$--polytope if and only if $\delta_P=0$.

By Lemma \ref{lem:trunc}, an orbifold based on a truncation $n$--polytope $P$ is weakly orderable since 
%$P$ is obtained from an $n$--simplex by iterated truncation and 
a compact Coxeter orbifold based on 
an $n$--simplex is weakly orderable. 
\end{proof} 

\begin{lemma}\label{lem:trunc} 
Let a polytope $P_2$ be obtained from a polytope $P_1$ by iterated truncation. 
Suppose that a compact Coxeter orbifold $\hP_2$ has the base polytope $P_2$ and 
another compact Coxeter orbifold $\hP_1$ has the base polytope $P_1$, 
$\hP_2$ has the ridge orders extending those of $\hP_1$, and $\hP_1$ is weakly orderable. 
Then $\hP_2$ is weakly orderable. 
\end{lemma} 
\begin{proof} 
By induction, suppose that % a compact Coxeter orbifold $\hP_2$ has the base polytope 
$P_2$ obtained from 
%the base polytope 
$P_1$ by a truncation at a vertex $v$ of $P_1$. 
%and $\hP_2$ has the inherited ridge orders from $\hP_1$ and $\hP_1$ is weakly orderable. 
 %If we truncate $P_1$ around a vertex $v$, then 
We give an ordering
of faces of $P_2$ by labeling the new facet to be the lowest one $F_1$ and the remaining ones are to be denoted
$F_{i+1}$ when they were labeled by $F_i$ before.
Then we need to check for $F_1$ only since the other faces already satisfy the weak orderability condition for those faces.
However, $F_1$ meets only $n$ facets by the simplicity of $P_1$ and
since these $n$ facets were meeting at a vertex only, $F_1$ can only be an $(n-1)$--dimensional simplex.
This implies that any collection of the facets meeting $F_1$ are in a general position.
\end{proof}

\section{Examples}
\label{s:Examples}
Section \ref{s:Examples} provides several examples of weakly orderable compact hyperbolic Coxeter $3$--orbifolds and
gives two examples satisfying only one of the two conditions $(C1)$ and $(C2)$
where the conclusion of Theorem \ref{thm:main1} does not hold.

%In Section \ref{subs:Examples} we prove that
%given the collection of compact hyperbolic $3$--orbifolds with no prismatic $3$--circuit and at most one prismatic $4$--circuit,
%almost all are weakly orderable.
%Let $e_+$ be the number of ridges of order $\geq 3$ in $\hP$
%In Sections \ref{subs:C2} and \ref{subs:C1}, we show that the local deformation space of real projective structures on
%$\hP$ is not homeomorphic to a cell of dimension $e_+\! -n$ for
%a compact hyperbolic Coxeter $n$--orbifold $\hP$ if $\hP$ does not satisfy the assumptions in Theorem \ref{thm:main1}

\subsection{Weakly orderable compact hyperbolic Coxeter $3$--orbifolds}
\label{subs:Examples}

Every compact hyperbolic Coxeter $3$--orbifold whose base polytope has the combinatorial type of a cube is weakly orderable.
%However, there exists a non-weakly-orderable compact hyperbolic Coxeter $3$--orbifold with base polytope of type of a dodecahedron.
Theorem \ref{thm:probability} shows that 
almost all compact hyperbolic Coxeter $3$--orbifolds, with the combinatorial type of a dodecahedron, are weakly orderable 
while there are ones not weakly orderable. 

%We think that the following types of $3$--polytopes form some portion of the set of all $3$--polytopes admitting hyperbolic structures.

Before going to the proof of Theorem \ref{thm:probability}, we state Tutte's theorem \cite{Tutte47}.

An $1$--dimensional cell complex $\G$ is a \emph{graph}.
It consists of vertices ($0$--cells) to which edges ($1$--cells) are attached.
We deal with only \emph{simple} graphs, that have no loops and no more than one edge between any two vertices.
The \emph{degree} of a vertex in a graph is the number of edges with which it is incident.
If all the vertices in a graph $\G$ have degree $d$, $\G$ is said to be \emph{regular of degree $d$}.

A \emph{subgraph} of $\G$ is a graph having all of its vertices and edges in $\G$.
A graph $\G$ with at least $k+1$ vertices is \emph{$k$--connected} if every subgraph of $\G$, obtained by omitting from
$\G$ any $k-1$ or fewer vertices and the edges incident to them, is connected.
A \emph{spanning subgraph} of $\G$ is a subgraph containing all the vertices of $\G$.
A \emph{factor} is a spanning subgraph which is regular of degree $1$.

\begin{theorem}{\rm \cite{Tutte47}}\qua
Let $\G$ be a finite graph. 
If $\G$ is a $d$--connected graph having the even number of vertices and is regular of degree $d$, then $\G$ has a factor.
Moreover, if, in addition, $\mathfrak{e}$ is any edge of $\G$, then $\G$ has a factor containing $\mathfrak{e}$.
\end{theorem}

\begin{lemma}\label{lem:coloring}
Let $P$ be a properly convex compact $3$--polytope but not a tetrahedron. 
Suppose that $P$ has no prismatic $3$--circuit and has at most one prismatic $4$--circuit.
Then there exists a compact hyperbolic Coxeter $3$--orbifold $\hP$ with the base polytope $P$ such that 
each vertex is incident with exactly two edges of order $2$. % and one edge not of order $2$.
\end{lemma}
\begin{proof}
Assume that four $2$-cells $F_i$, $F_j$, $F_k$ and $F_l$ of $P$ form a prismatic $4$--circuit.
Denote the edge $F_i \cap F_j$ by $\mathfrak{e}$.
By Steinitz's theorem, the graph $\G= \G(P)$ of $P$ is $3$--connected (See Gr{\"u}nbaum \cite[Chapter 13]{Grunbaum03}). 
Since $P$ is simple,
$\G$ is regular of degree $3$ and the number of vertices is even. 
By Tutte's theorem, $\G$ has a factor $\F$ containing
$\mathfrak{e}$. If $P$ has no prismatic $4$--circuit, then we choose an arbitrary factor $\F$ of $\G$.
Every vertex of $P$ is incident with two edges in $\G \backslash \F$ and one edge in $\F$.
Observe that
\begin{displaymath}
\tfrac{\pi}{k} + \tfrac{\pi}{2} + \tfrac{\pi}{2} > \pi \quad \text{and} \quad \tfrac{\pi}{k}
 + \tfrac{\pi}{2} + \tfrac{\pi}{2} + \tfrac{\pi}{2} < 2\pi \quad \text{for every integer } k \geq 3.
\end{displaymath}
Andreev's theorem (see \cite[Theorem 1.4 and Proposition 1.5]{Roeder07})
yields a compact hyperbolic Coxeter $3$--orbifold $\hP$ such that every edge in
$\G \backslash \F$ (resp. $\F$) is of order $2$ (resp. of order $k \neq 2$), 
corresponding to a dihedral angle $\tfrac{\pi}{2}$ (resp. $\tfrac{\pi}{k}$).
\end{proof}

Let $\G$ be a finite graph, and let $L$ be a set.
Denote by $E(\G)$ the set of edges of $\G$.
A function $\f \co  E(\G) \rightarrow L$ is called an \emph{edge-labeling function}, and
we call a pair $(\G,\f)$ an \emph{edge-labeled graph}.
An edge $\mathfrak{e}$ is called an \emph{$l$--edge} if $\f(\mathfrak{e})=l$.

In this section, we consider the edge-labeled graph $(\G,\f)$ satisfying the following conditions:
\begin{enumerate}
\item[(E1)] $\G$ is simple, planar and $3$--connected. %ie $\G$ is the graph of an abstract $3$--polyhedron.
\item[(E2)] $\G$ is regular of degree $3$.
\item[(E3)] the set $L$ of labels is $\{ 0, 1\}$.
\item[(E4)] every vertex of $\G$ is incident with three edges $\mathfrak{e}_1$, $\mathfrak{e}_2$ and $\mathfrak{e}_3$
such that
\begin{displaymath}
\f(\mathfrak{e}_1)+\f(\mathfrak{e}_2)+\f(\mathfrak{e}_3) \equiv 1 \;\; \text{(mod 2)}.
\end{displaymath}
\end{enumerate}
If $(\G,\f)$ can be ordered so that each face contains at most three $0$--edges in faces of higher indices, 
$(\G,\f)$ is said to be \emph{weakly orderable}. 
(Here (E1) holds if and only if $\G$ is isomorphic to the $1$--skeleton of a properly convex $3$--polytope by Steinitz's theorem.) %\marginpar{added} 

Let $P$ be a properly convex $3$--polytope, and let $\G$ be the $1$--skeleton of $P$ as an abstract $3$--polyhedron.
The graph $\G$ is embedded in the $2$--dimensional sphere $S^2$ homeomorphic to the boundary of $P$.
We call a face of $P$ a {\em face} of $\G$.
The respective numbers of vertices, edges and faces of $\G$ shall be denoted by $v$, $e$ and $f$.

\begin{lemma}\label{lem:labeled}
Let  $(\G,\f)$ be an edge-labeled graph satisfying the condition (E1)--(E4). 
%Then a face $F$ of $\G$ 
%has $0$--edges of $F$.
Then the number of $0$--edges of at least one face $F$ of $\G$ is less than or equal to $3$.  %\marginpar{at least added} 
\end{lemma}
\begin{proof}
Denote by $e_2$ the number of $0$--edges.
(E1) implies that $v-e+f = 2$, (E2) implies that $2e = 3v$, and (E3) and (E4) imply that $2e_2 \leq 2v$.
By an elementary computation, we obtain
$2e_2 \leq 4(f-2) < 4f.$
The conclusion is immediate.
\end{proof}

We define an {\em edge-deletion} for an edge-labeled graph satisfying the condition (E1)--(E4):
%(resp. on graphs satisfying the condition (E1)--(E2)):
When each pair of edges ending at a vertex $\mathfrak{a}$ or $\mathfrak{b}$ of an edge $\mathfrak{e}$ in 
%an edge-labeled graph $(\G,\f)$ (resp. $\G$) has the same label, 
$(\G, \f)$ have the same labels, 
we can \emph{delete} $\mathfrak{e}$ from $(\G,\f)$  %(resp. $\G$)
and amalgamate the pair of edges incident to $\mathfrak{a}$ and the pair for $\mathfrak{b}$ (see Figure \ref{fig:deletion}).
%The possibility of applying the deleting operation to an edge $\mathfrak{e}$ of $(\G,\f)$ 
%we require that the
%two adjacent edges which are adjacent to an edge $\mathfrak{e}$ have the same label.
%% 
We define the {\em edge-deletion} on $\G$ satisfying (E1) and (E2) similarly.
%We note that the conditions (E2)--(E4) (resp. (E2)) are preserved under the edge-deleting operations.
Edge-deletion preserves conditions (E2)--(E4) for $(\G, \f)$ ( just (E2) for $\G$). 

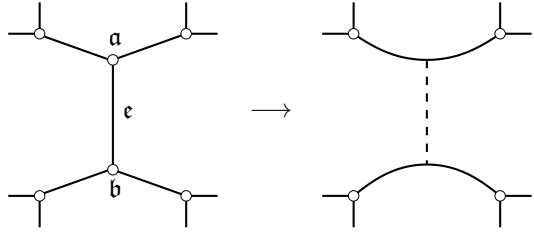
\begin{figure}[ht]
\labellist
\small\hair 2pt
\pinlabel $\mathfrak{a}$ at 92 162
\pinlabel $\mathfrak{b}$ at 92 34
\pinlabel $\mathfrak{e}$ at 103 99
\pinlabel $\longrightarrow$ at 226 99
\endlabellist
\centering
\includegraphics[height=3cm]{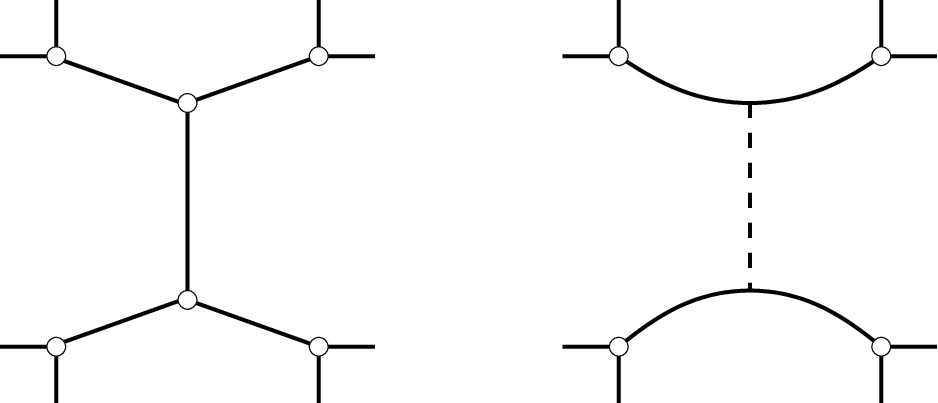}
\caption{An edge-deleting operation}
\label{fig:deletion}
\end{figure}

Let $\G$ be a graph satisfying the condition (E1)--(E2).
An edge $\mathfrak{e}$ of $\G$ is said to be \emph{removable} when the graph obtained from $\G$ 
by deleting the edge $\mathfrak{e}$ remains to satisfy the condition (E1) (and (E2) obviously).  %\marginpar{added} 
%An edge which is not removable is said to be \emph{non-removable}. 
%The following property of the set of non-removable edges was obtained by Fouquet and Thuillier \cite{Fouquet11}.

\begin{theorem}{\rm \cite[Corollary 2.7]{Fouquet11}}\label{thm:exmain}
Let $\G$ be a graph with more than $6$ edges satisfying the condition $(E1)$--$(E2)$, and
let $C$ be a cycle of $\G$. Then $C$ contains at least two removable edges.
\end{theorem}

\begin{lemma}\label{lem:weakly}
Let $P$ be a properly convex compact $3$--polytope, and let $\hP$ be the Coxeter $3$--orbifold arising from $P$.
Assume that every vertex of $\hP$ is incident with two edges of order $2$ and one edge of order $\geq 3$.
Then $\hP$ is weakly orderable.
\end{lemma}
\begin{proof}
Let $\hat \G$ be the graph of the $3$--polytope $P$. Define the edge-labeling function $\hat \f$ by
\begin{equation*}
\hat \f(\mathfrak{e}) = \left \{
\begin{array}{rl}
0 & \text{if the edge $\mathfrak{e}$ is of order $2$},\\
1 & \text{otherwise}.
\end{array} \right.
\end{equation*}
Then the edge-labeled graph $(\hat \G, \hat \f)$ satisfies the conditions (E1)--(E4) by Steinitz's theorem.

%We will show that the faces of $(\G,\f)$ can be ordered so that each face contains at most three $0$--edges in faces of higher indices.
% In that case $(\G,\f)$ is said to be \emph{weakly orderable}. This is equivalent to saying that $\hP$ is weakly orderable.

Given a labeled graph $(\G, \f)$ satisfying (E1)--(E4) and a face $F$, 
we can reverse the label for every edge of $F$ and the new labeling function on $\G$ will still satisfy (E1)--(E4). 

We show that  if $(\G, \f)$ satisfies conditions (E1)--(E4), then $(\G, \f)$ is weakly orderable.

The proof proceeds by induction on the number $f$ of faces of $\G$.
The condition (E1) implies $f\geq 4$.  
We have $f=4$ if and only if $\G$ is the graph of a tetrahedron. In this case $\G$ is weakly orderable.

Now assume that $\G$ has $f$ faces for $f \geq 5$ and that any labeled graph $(\G', \f')$ satisfying (E1)--(E4) 
 is weakly orderable provided that the number of faces is $< f$.
By Lemma \ref{lem:labeled}, $\G$ has a face $F$ such that
the number of $0$--edges of $F$ is less than or equal to  $3$ as $e = 3(f-2)$ by the Euler formula. %\marginpar{added}  
%We call $F$ the first face $F_1$ of $\G$.
By Theorem \ref{thm:exmain}, the cycle $\partial F$ contains a removable edge $\mathfrak{e}$.
\begin{itemize}
\item If we have $\f(\mathfrak{e})=1$, then 
each pair of edges which are adjacent to a vertex of $\mathfrak{e}$ have the same label.
Then let $\f':= \f$. 
\item Otherwise, $\f(\mathfrak{e}) = 0$. 
We relabel every edge in the cycle $\partial F$ to become the edge of the opposite label, and
obtain the new label function $\f'$ of $\G$ such that $\f'(\mathfrak{e})=1$.
The resulting edge-labeled graph $(\G,\f')$ still satisfies the conditions (E1)--(E4).
Also each pair of edges which are adjacent to a vertex of $\mathfrak{e}$ have the same label.
\end{itemize}
%In this situation, the number of vertices of $F$ has to be less than or equal to six. )
%\marginpar{sc: added this}

Denote by $F'$ the face adjacent to $F$ such that $F \cap F' = \mathfrak{e}$.
We can delete the edge $\mathfrak{e}$ of $(\G,\f')$.
Two adjacent faces $F$ and $F'$ are amalgamated into a face $F''$ (See Figure \ref{fig:reduction}).

\begin{figure}[ht]
\labellist
\small\hair 2pt
\pinlabel $F$ at 40 99
\pinlabel $F'$ at 150 99
\pinlabel $\mathfrak{e}$ at 103 99
\pinlabel $\longrightarrow$ at 226 99
\pinlabel $F''$ at 410 99
\endlabellist
\centering
\includegraphics[height=3cm]{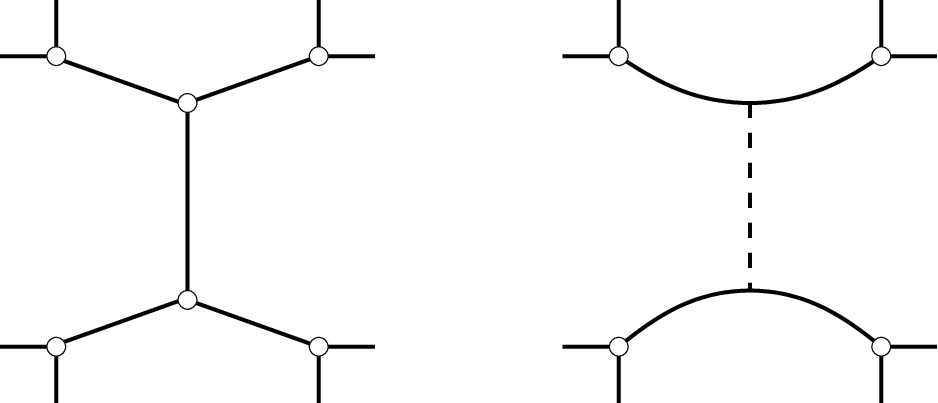}
\caption{Amalgamating two adjacent facets into a facet}
\label{fig:reduction}
\end{figure}

Now, the resulting edge-labeled graph $(\widetilde{\G},\widetilde{\f})$ has fewer faces
but still satisfies all conditions (E1)--(E4) since $\mathfrak{e}$ is removable. 
Using the induction hypothesis, the edge-labeled graph $(\widetilde{\G},\widetilde{\f})$ is weakly orderable, 
and hence we label the faces of $\widetilde{G}$  with the indices $\{2, 3, \dotsc , f\}$.
Now we reinsert $\mathfrak{e}$ and recover the old labels $\f$ of $\G$ 
\begin{itemize} 
\item by doing nothing or 
\item by reversing the labels of the edges of $F$ provided that we reversed the labels of edges of $F$ above. 
\end{itemize} 
%thereby showing that 
Let $F$ be the first face of $(\G,\f)$, and we label all the other faces of $(\G,\f)$ by inheriting the ordering of faces of 
$(\widetilde{\G},\widetilde{\f})$. 
Since the number of $0$-edges of $F$ under $\f$ is less than or equal to $3$, 
$(\G,\f)$ is weakly orderable with the indices $\{1, 2, \dotsc , f\}$.
%This completes the proof. 
%Here, we note that since $F$ is the face of the lowest index in $(\G,\f)$, the process to obtain the new labeling function $\f'$ of $\G$ does not affect the result.
\end{proof}

%Again, find an face $\widetilde{F}$ of $\widetilde{\G}$ such that the number of $0$--edges of $\widetilde{F}$ is less than $4$. We call $\widetilde{F}$ the $2$nd %face $F_2$ of $\G$. Continue the process in this manner.
%We will eventually end up with a tetrahedron,
%and we are done by an induction on the number of edges.
%\marginpar{sc: changed the last para.}

%\marginpar{prob not -> 0 is going to be proved.} 
\emph{Proof of Theorem \ref{thm:probability}.}
By Lemma \ref{lem:coloring}, there exists a compact hyperbolic Coxeter orbifold $\hP$ whose base polytope is combinatorially equivalent to $P$. 
Let $e$ be the number of edges of $P$, and let $p = \tfrac{1}{3}e$. Observe that $p \in \bZ_+$ by the vertex incidence condition. 
Let $\mathcal{N}(d)$ be the set of compact hyperbolic Coxeter orbifolds whose base polytopes are combinatorially equivalent to $P$ and 
whose edge orders are less than or equal to $d$.
For each integer $d \geq 7$ and $j \in \{0, 1, \dotsc , e\}$, we define
\begin{gather*}
\mathcal{N}_\omega(d) = \{\, \hP \in \mathcal{N}(d)  |\,\text{$\hP$ is weakly orderable} \}, \\
\mathcal{N}_j(d) = \{\, \hP \in \mathcal{N}(d) |\, \text{the number of edges of order $\geq 7$ in $\hP$ equals $j$ } \}.
\end{gather*}

Assume that $\hP$ is a compact hyperbolic Coxeter $3$--orbifold.
%By Andreev's theorem (Theorem 1.4(2) of \cite{Roeder07}), 
By the orbifold condition, 
if an edge $\mathfrak{e}$ of $\hP$ is
of order $\geq 7$, then edges which are adjacent to $\mathfrak{e}$ are of order $2$.
Therefore the number of edges of order $\geq 7$ in $\hP$ is less than or equal to $p= \tfrac{1}{3}e$.
In other words, $\mathcal{N}_j (d) = \emptyset$ for every $j > p$.

%By Lemma \ref{lem:coloring}, $\mathcal{N}_p(d) \ne \emptyset$ since we can realize a compact Convex hyperbolic polytope with 
%$j = \tfrac{1}{3}e$ equal to the number of edges of order $d \geq 7$. 
We have
\begin{displaymath}
|\mathcal{N}(d)|= \sum_{j=0}^{p} |\mathcal{N}_j (d)|.
\end{displaymath}
Moreover, observe that $\hP  \in \mathcal{N}_p (d)$ if and only if every vertex of $\hP$ is incident with two edges of order $2$ 
and one edge of order $\geq 7$. For any fixed integers $l, m \geq 2$,
\begin{equation*}
\begin{split}
&\tfrac{1}{k} + \tfrac{1}{l} + \tfrac{1}{m} > 1 \text{ for some integer } k \geq 7 \; \Leftrightarrow \;
\tfrac{1}{k} + \tfrac{1}{l} + \tfrac{1}{m} > 1 \text{ for each integer } k \geq 7.
\end{split}
\end{equation*}
%(compare these inequalities to $(A1)$ of Andreev's theorem).
Consequently for each $d \in \{7, 8, \dotsc \}$ we have
\begin{equation*}
|\mathcal{N}_j (d)| = |\mathcal{N}_j (7)| \cdot (d-6)^{j}.
\end{equation*}
Lemmas \ref{lem:coloring} shows that $\mathcal{N}_p(7) \ne \emptyset$ and Lemma \ref{lem:weakly} implies that
\begin{displaymath}
\frac{|\mathcal{N}_\omega(d)|}{|\mathcal{N}(d)|}
\geq \frac{|\mathcal{N}_p (d)|}{|\mathcal{N}(d)|}
=  \frac{|\mathcal{N}_p (7)| \cdot (d-6)^{p}}{ \sum_{j=0}^{p} |\mathcal{N}_j (7)| \cdot (d-6)^{j}}  %\quad \text{and} \quad
%|\mathcal{N}_p (7)| \neq 0,
\end{displaymath}
establishing the result.
\hfill $\Box$

%\marginpar{Need $\mathcal{N}_p$ has only $2$ and $d$, $d \geq 7$. }

\begin{example}
Let $m$ be an integer $\geq 5$. A L\"{o}bell $3$--polytope $L(m)$ is  a $3$--polytope with $(2m+2)$ faces where 
upper and lower sides are $m$--gons, and the complementary surface is a union of $2m$ pentagons,
arranged similarly as in the dodecahedron. Figure \ref{fig:Lobell} shows the case when $m=6$.

\begin{figure}[ht]
\centering
\includegraphics[height=3cm]{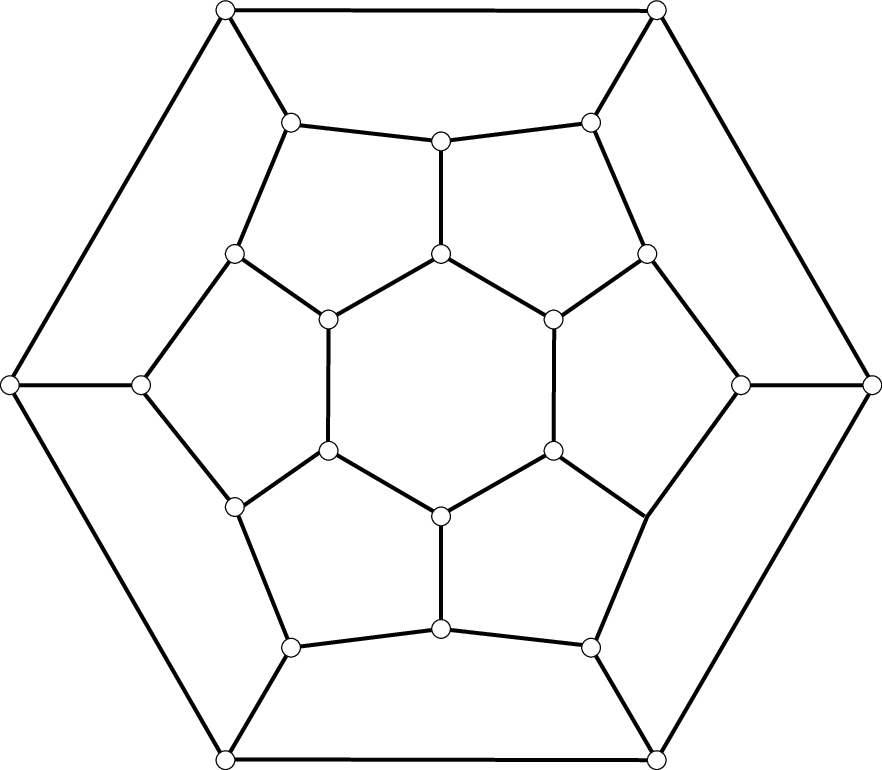}
\caption{A L\"{o}bell $3$--polytope $L(6)$}\label{fig:Lobell}
\end{figure}

For each $m \geq 5$, the L\"{o}bell $3$--polytope $L(m)$ has no prismatic $3$-- or $4$--circuits.
By Theorem \ref{thm:probability} almost all compact hyperbolic Coxeter $3$--orbifolds with the combinatorial type of $L(m)$ are weakly orderable.
\end{example}

\subsection{An example satisfying only the condition $(C1)$}
\label{subs:C2}
Let $d$ be a fixed integer $> 3$. We consider the compact hyperbolic Coxeter $3$--polytope $P$ shown in Figure \ref{fig:doublecube}.
Here, if an edge is labeled $d$, then its dihedral angle is $\tfrac{\pi}{d}$. Otherwise,
its dihedral angle is $\tfrac{\pi}{2}$.

\begin{figure}[ht]
\labellist
\small\hair 2pt
\pinlabel $F$ at 147 151
\pinlabel $F'$ at 339 151
\pinlabel $d$ at 245 308
\pinlabel $d$ at 245 168
\pinlabel $d$ at 245 23
\endlabellist
\centering
\includegraphics[height=2.5cm]{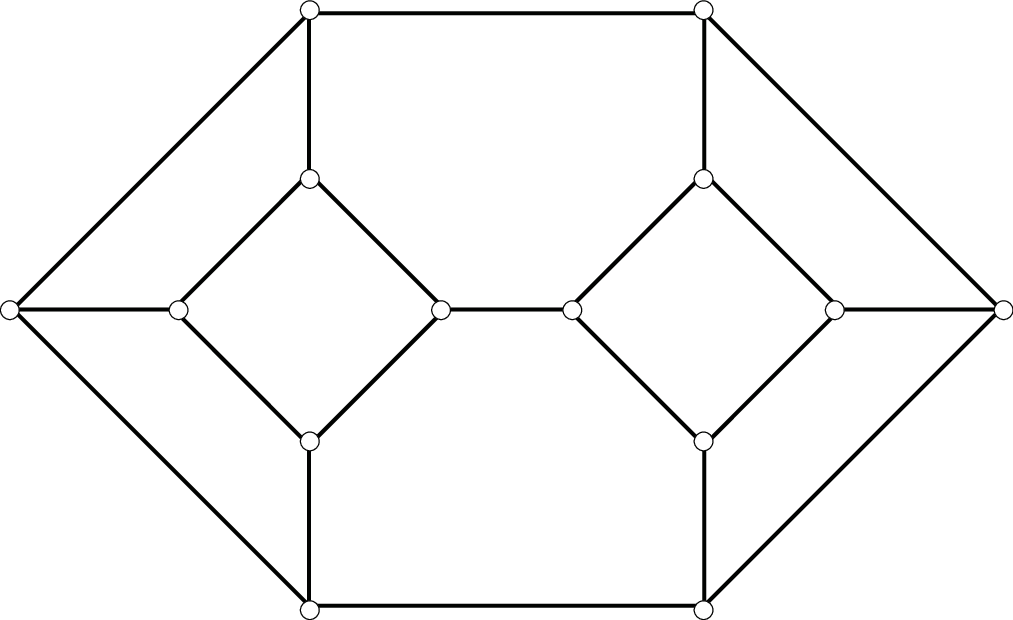}
\caption{A compact hyperbolic Coxeter $3$--polytope}\label{fig:doublecube}
\end{figure}
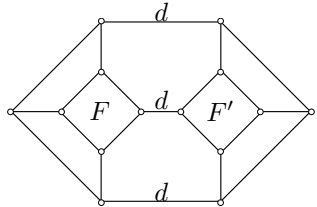

Obviously, $e_+(\hP)\! -3=0$.
However $\hP$ is not weakly orderable, since every facet in $\hP$ contains four edges of order $2$.

Observe that the cell structure of $P$ has a reflection-type topological symmetry interchanging $F$ and $F'$. Hence,
the Coxeter $3$--orbifold $\hP$ arising from $P$ has an order-two isometry fixing an embedded totally geodesic $2$--dimensional suborbifold $S$
by the Mostow rigidity. 
%Consequently there are non-trivial deformations in $\De(\hP)$ obtained by \emptyseth{projective bendings} along $S$,
 \emph{Projective bendings} along $S$ provide
 non-trivial deformations in $\De(\hP)$ by Johnson and Millson \cite[Lemma 5.1]{Johnson87}.
Hence a neighborhood of the hyperbolic point in $\De(\hP)$ is \emph{not} a manifold of dimension $0$
while $\De(\hP)$ could still be a manifold. 
(See Choi, Hodgson and Lee \cite[Theorem 10]{Choi12} also.)

\subsection{An example satisfying only the condition $(C2)$}
\label{subs:C1}
In 1996, Esselmann \cite{Esselmann96} classified all the compact hyperbolic Coxeter polytopes whose 
combinatorial types are the products of two simplices of dimension greater than 1. 
%We consider one of these hyperbolic polytopes.
Let $P$ be the compact hyperbolic Coxeter $4$--polytope whose combinatorial type is the product of two triangles and
whose Coxeter graph is shown in Figure \ref{fig:Esselmann}. See Vinberg \cite{Vinberg85} or Bourbaki \cite{Bourbaki02} for the definition of Coxeter graphs.

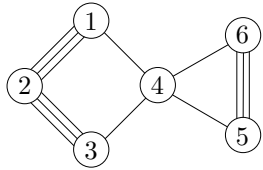
\begin{figure}[ht]
\labellist
\small\hair 2pt
\pinlabel $1$ at 183 317
\pinlabel $2$ at 39 177
\pinlabel $3$ at 183 37
\pinlabel $4$ at 327 177
\pinlabel $5$ at 511 68
\pinlabel $6$ at 511 287
\endlabellist
\centering
\includegraphics[height=2.2cm]{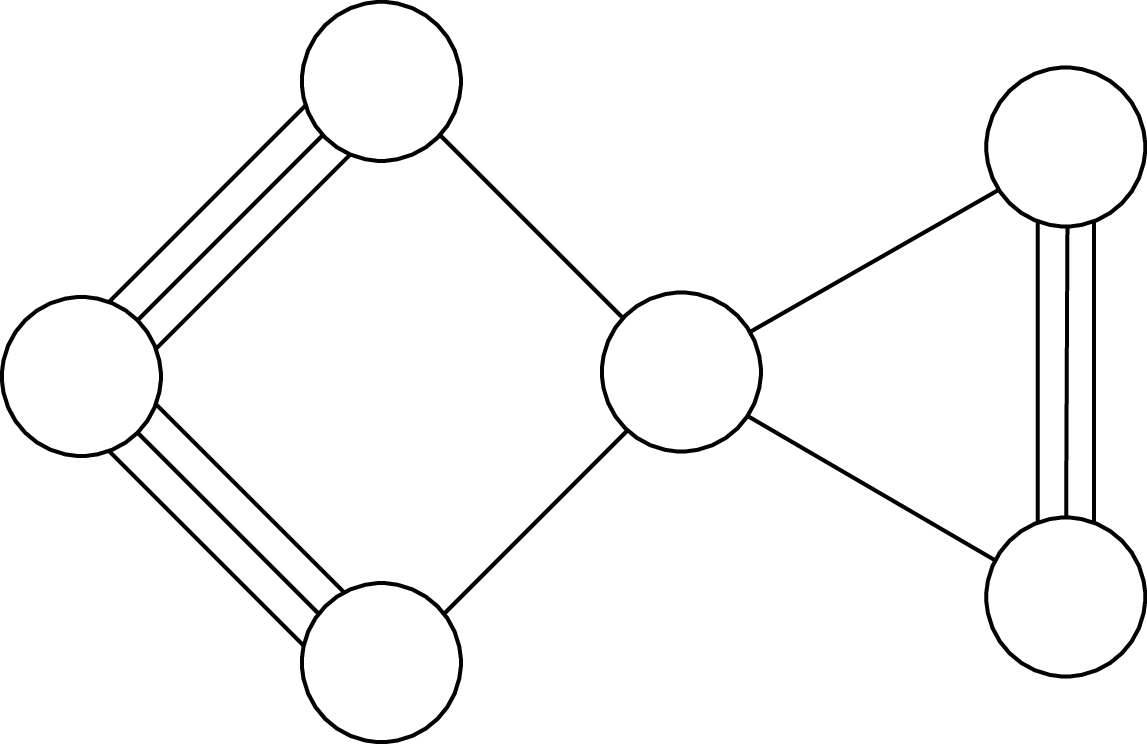}
\caption{One of Esselmann's polytopes}
\label{fig:Esselmann}
\end{figure}

Since the $4$--polytope $P$ has $6$ facets and $15$ ridges,
\begin{displaymath}
\delta_P = e-nf+\tfrac{n(n+1)}{2}=1 \neq 0, \text{ ie $P$ does not satisfy the condition $(C1)$.}
\end{displaymath}

However the Coxeter orbifold $\hP$ arising from $P$ is weakly orderable, ie $\hP$ satisfies the condition $(C2)$.
This can be shown by checking explicitly. 
% $F_3$ is the first facet meeting $F_5$ and $F_6$ with higher indices perpendicularly
%but $F_5$ and $F_6$ meet in an acute angle and hence is in a general position.
%$F_2$ meets with $F_4$, $F_5$ and $F_6$ with higher indices perpendicularly.
%By angle conditions, $F_4, F_5,$ and $F_6$ are in general position. 
%$F_1$ meets with $F_3, F_5,$ and $F_6$ with higher indices perpendicularly.
%By angle conditions, $F_3, F_5,$ and $F_6$ are in general position. 
%However, In $F_1 \cap F_3, F_1 \cap F_5,$ and $F_1 \cap F_6$ must be in general position by the angle conditions.

We show that the hyperbolic point in $\De(\hP)$ for the hyperbolic Coxeter orbifold $\hP$
is singular. 

Assume that $\Gamma$ is a projective Coxeter group so that
$\Omega_\Gamma / \Gamma$ is homeomorphic to $\hP$, and $A$ is the Cartan matrix of $\Gamma$.
We make the Cartan matrix $A$ by a unique diagonal action (see Equation \eqref{eqn:diaga}) so that
\begin{align*}
a_{12} &= a_{21}=-2\cos\left(\tfrac{\pi}{5}\right), \quad a_{23} = a_{32}=-2\cos\left(\tfrac{\pi}{5}\right),
\quad a_{34} = a_{43}=-2\cos\left(\tfrac{\pi}{3}\right), \\
a_{45} &= a_{54}=-2\cos\left(\tfrac{\pi}{3}\right), \quad a_{56} = a_{65}=-2\cos\left(\tfrac{\pi}{5}\right).
\end{align*}
Define $x=-a_{14}$ and $y=-a_{46}$. The Cartan matrix $A=(a_{ij})$ of $\Gamma$ is as follows:
\begin{displaymath}
A = \left[ \begin{array}{cccccc}
2                       & -\tfrac{1+\sqrt{5}}{2} & 0                      & -x      & 0  & 0                      \\
-\tfrac{1+\sqrt{5}}{2}  & 2                      & -\tfrac{1+\sqrt{5}}{2} &  0      & 0  & 0                      \\
0                       & -\tfrac{1+\sqrt{5}}{2} & 2                      & -1      & 0  & 0                      \\
-x^{-1}                 & 0                      & -1                     & 2       & -1 & -y                     \\
0                       & 0                      & 0                      & -1      & 2  & -\tfrac{1+\sqrt{5}}{2} \\
0                       & 0                      & 0                      & -y^{-1} & -\tfrac{1+\sqrt{5}}{2} & 2
\end{array} \right].
\end{displaymath}
Moreover, $\rank A = 5$ if and only if $\det (A) = 0$.
By simple calculation, we obtain
\begin{displaymath}
\det (A) = \tfrac{1}{2xy}(8x-(5+\sqrt{5})y-(6-2\sqrt{5})xy-(5+\sqrt{5})x^2y+8xy^2)  = 0.
\end{displaymath}
Note that $x$ and $y$ are positive. By Corollary \ref{cor:identification}, the deformation space $\De(\hP)$  is homeomorphic to the solution space
\begin{displaymath}
\mathcal{S} = \{(x,y) \in \mathbb{R}_+^2 |\, f(x, y):= 8x-(5+\sqrt{5})y-(6-2\sqrt{5})xy-(5+\sqrt{5})x^2y+8xy^2 = 0 \},
\end{displaymath}
pictured in Figure \ref{fig:notmanifold}. 
\begin{figure}[ht]
\centering
\includegraphics[height=4cm]{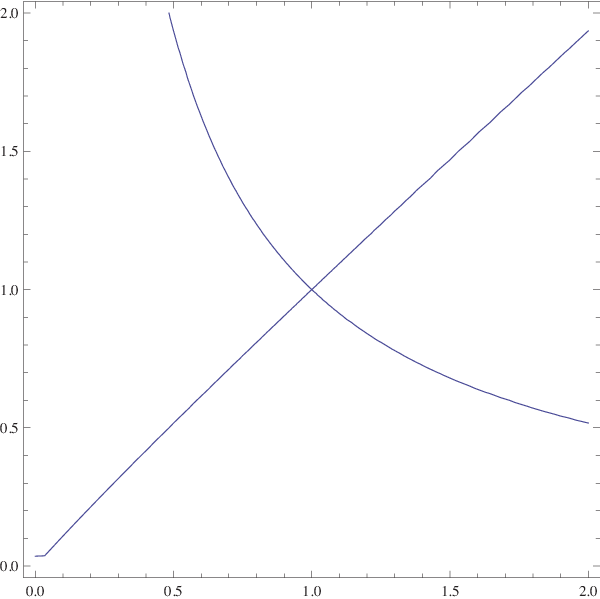}
\caption{$8x-(5+\sqrt{5})y-(6-2\sqrt{5})xy-(5+\sqrt{5})x^2y+8xy^2 = 0$}
\label{fig:notmanifold}
\end{figure}
By Vinberg \cite[Proposition 24]{Vinberg71}, 
$(1,1) \in \mathcal{S}$ corresponds to the unique hyperbolic point in $\De(\hP)$, and hence
any neighborhood of the hyperbolic point of $\De(\hP)$ is singular. (The polynomial $f(x, y)$ is irreducible.)

%%%%%%%%%%%%%%%%%%%%   End of main body of article
%
%                             References
%
%   BiBTeX users uncomment the following line:
\bibliography{mybib}{}
\bibliographystyle{gtart}
\end{document}